\documentclass[10pt]{amsart}

\usepackage{epsfig,amsmath,amsthm,amssymb,graphicx,enumitem,parskip, soul}
\usepackage[top=1.25in, bottom=1.25in, left=1.25in, right=1.25in]{geometry}
\usepackage{caption}
\usepackage{subcaption}
\usepackage{color}
\usepackage[dvipsnames]{xcolor}
\definecolor{lightblue}{rgb}{.1,0.35,.8}
\definecolor{diagramred}{rgb}{0.78039215686,0,0}
\definecolor{diagramblue}{rgb}{0,0,0.78039215686}
\definecolor{diagramgreen}{rgb}{0,0.39215686274,0}
\usepackage[colorlinks,
    linkcolor={lightblue},
    citecolor={red!50!black},
    urlcolor={lightblue}]{hyperref}

\usepackage{cancel}%For \cancel

\usepackage{arydshln}

\usepackage{tikz}
\usetikzlibrary{patterns}
\usetikzlibrary{decorations.markings}
\usetikzlibrary{knots}
\usetikzlibrary{matrix, arrows, cd}
\usepackage{pb-diagram}

\makeatletter
\newtheorem*{rep@theorem}{\rep@title}
\newcommand{\newreptheorem}[2]{%
\newenvironment{rep#1}[1]{%
 \def\rep@title{#2 \ref{##1}}%
 \begin{rep@theorem}}%
 {\end{rep@theorem}}}
\makeatother

\newtheorem{proposition}{Proposition}[section]
\newtheorem{theorem}[proposition]{Theorem}
\newtheorem{corollary}[proposition]{Corollary}
\newtheorem{lemma}[proposition]{Lemma}

\newtheorem*{theorem*}{Theorem}
\newtheorem*{proposition*}{Proposition}
\newtheorem*{lemma*}{Lemma}
\newtheorem*{corollary*}{Corollary}

\theoremstyle{definition}
\newtheorem{definition}[proposition]{Definition}
\newtheorem{question}[proposition]{Question}

\newreptheorem{theorem}{Theorem}
\newreptheorem{proposition}{Proposition}
\newreptheorem{corollary}{Corollary}

\theoremstyle{remark}
\newtheorem{remark}[proposition]{Remark}

\newtheorem*{claim*}{Claim}
\newtheorem*{notation*}{Notation}

\newcommand{\bdry}{\partial}
\newcommand{\sm}{\smallsetminus}

\newcommand{\N}{\mathbb{N}}
\newcommand{\Z}{\mathbb{Z}}

\newcommand{\lk}{\operatorname{lk}}

\newcommand{\D}{\mathcal{D}}
\renewcommand{\int}{\operatorname{int}}

\newcommand{\RR}{\mathbb{R}}

\newcommand{\Sum}{\displaystyle \sum  }

\renewcommand{\Cup}{\displaystyle \bigcup  }

\newcommand{\onto}{\twoheadrightarrow}
\newcommand{\into}{\hookrightarrow}

\newcommand{\pref}[1]{(\ref{#1})}

\newcommand{\lineseg}[2]{\overline{#1#2}}

\keywords{link concordance, Whitney trick, Whitney tower}

\begin{document}
\title[The relative Whitney trick]{The relative Whitney trick and its applications}

\author{Christopher W.\ Davis}
\address{Department of Mathematics, University of Wisconsin--Eau Claire}
\email{daviscw@uwec.edu}

\author{Patrick Orson}
\address{Department of Mathematics, ETH Z\"{u}rich}
\email{patrick.orson@math.ethz.ch}

\author{JungHwan Park}
\address{Department of Mathematical Sciences, KAIST}
\email{jungpark0817@kaist.ac.kr}

\date{\today}

%\subjclass[2020]{57K10, 57N70}
\def\subjclassname{\textup{2020} Mathematics Subject Classification}
\expandafter\let\csname subjclassname@1991\endcsname=\subjclassname
\expandafter\let\csname subjclassname@2000\endcsname=\subjclassname
\subjclass{57K10, 57N70}

\begin{abstract}We introduce a geometric operation, which we call the relative Whitney trick, that removes a single double point between properly immersed surfaces in a $4$-manifold with boundary. Using the relative Whitney trick we prove that every link in a homology sphere is homotopic to a link that is topologically slice in a contractible topological $4$-manifold. We further prove that any link in a homology sphere is order $k$ Whitney tower concordant to a link in $S^3$ for all $k$. Finally, we explore the minimum Gordian distance from a link in $S^3$ to a homotopically trivial link. Extending this notion to links in homology spheres, we use the relative Whitney trick to make explicit computations for 3-component links and  establish bounds in general.

\end{abstract}

\maketitle
%==============================================================
\section{Introduction}
The Whitney trick is a fundamental technique of geometric topology and its general failure in 4-manifolds is widely cited as the reason that topology in this dimension is so interesting and unusual. In ambient dimension four, the (topological) Whitney trick seeks to remove a pair of oppositely signed transverse intersection points between two locally flat immersed surfaces. In this article, we will introduce a geometric technique called the \emph{relative Whitney trick} that removes a single point of intersection between properly immersed locally flat surfaces in a 4-manifold with boundary. The details of the relative Whitney trick are given in Section~\ref{sect: rel Whit trick}, but we sketch the procedure now. 
Suppose $S_1$ and $S_2$ are surfaces in a 4-manifold $W$ and $p\in S_1\cap S_2$. We find an 
immersed disk with an embedded arc of its boundary lying on the boundary of the $4$-manifold and use it to guide a regular homotopy of $S_1$ that removes the point of intersection $p$, at the cost of changing $\bdry S_1$ by a homotopy along that arc in the boundary; see Figure~\ref{fig: rel Whitney trick}.

\begin{figure}[h]
\begin{picture}(310,80)
\put(0,5){\includegraphics[width=37.5mm]{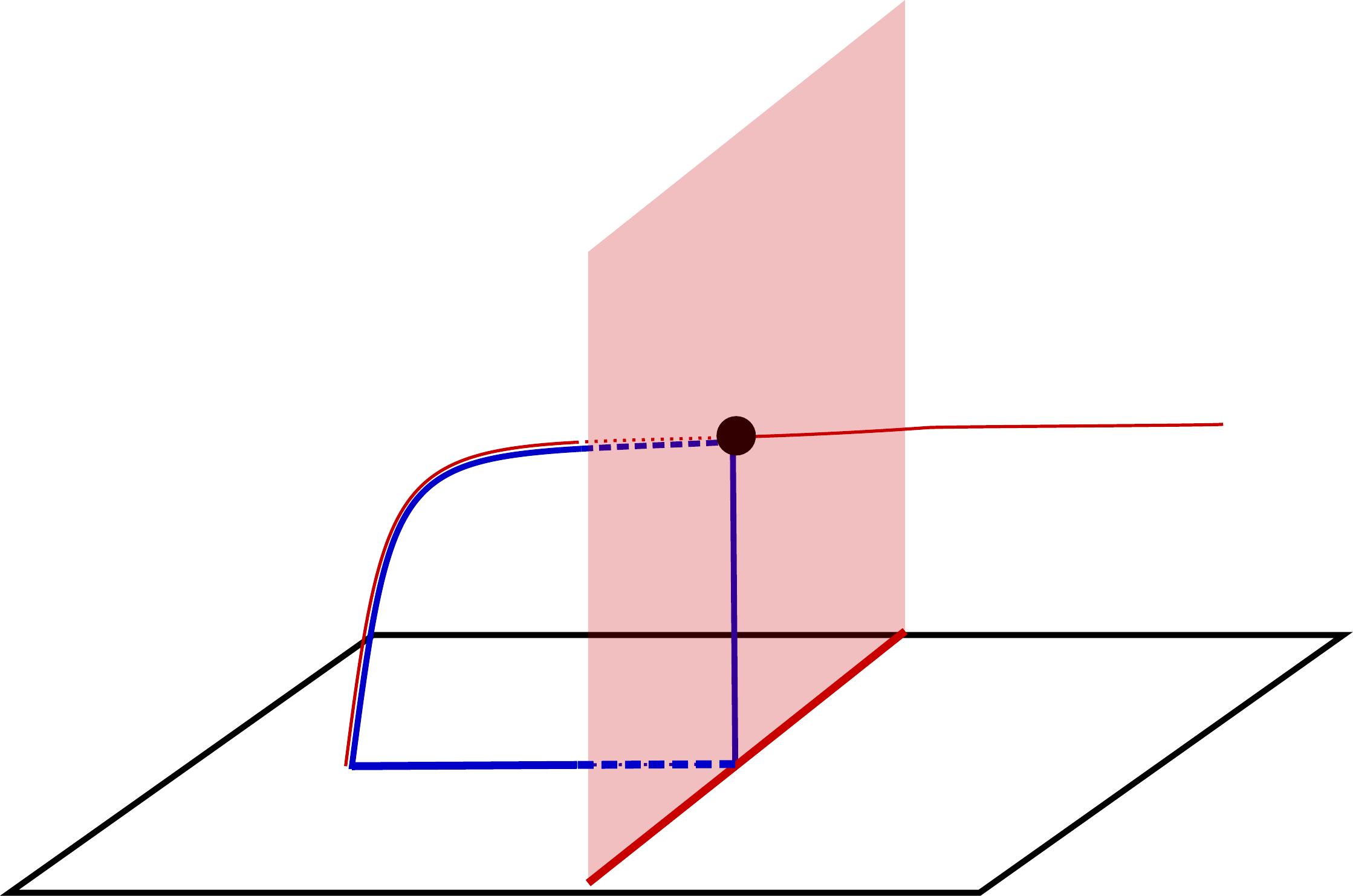}}

\put(70,12){$\bdry W$}
\put(50,68){$S_1$}
\put(82,45){$S_2$}
\put(55,45){$p$}

\put(100,5){\includegraphics[width=37.5mm]{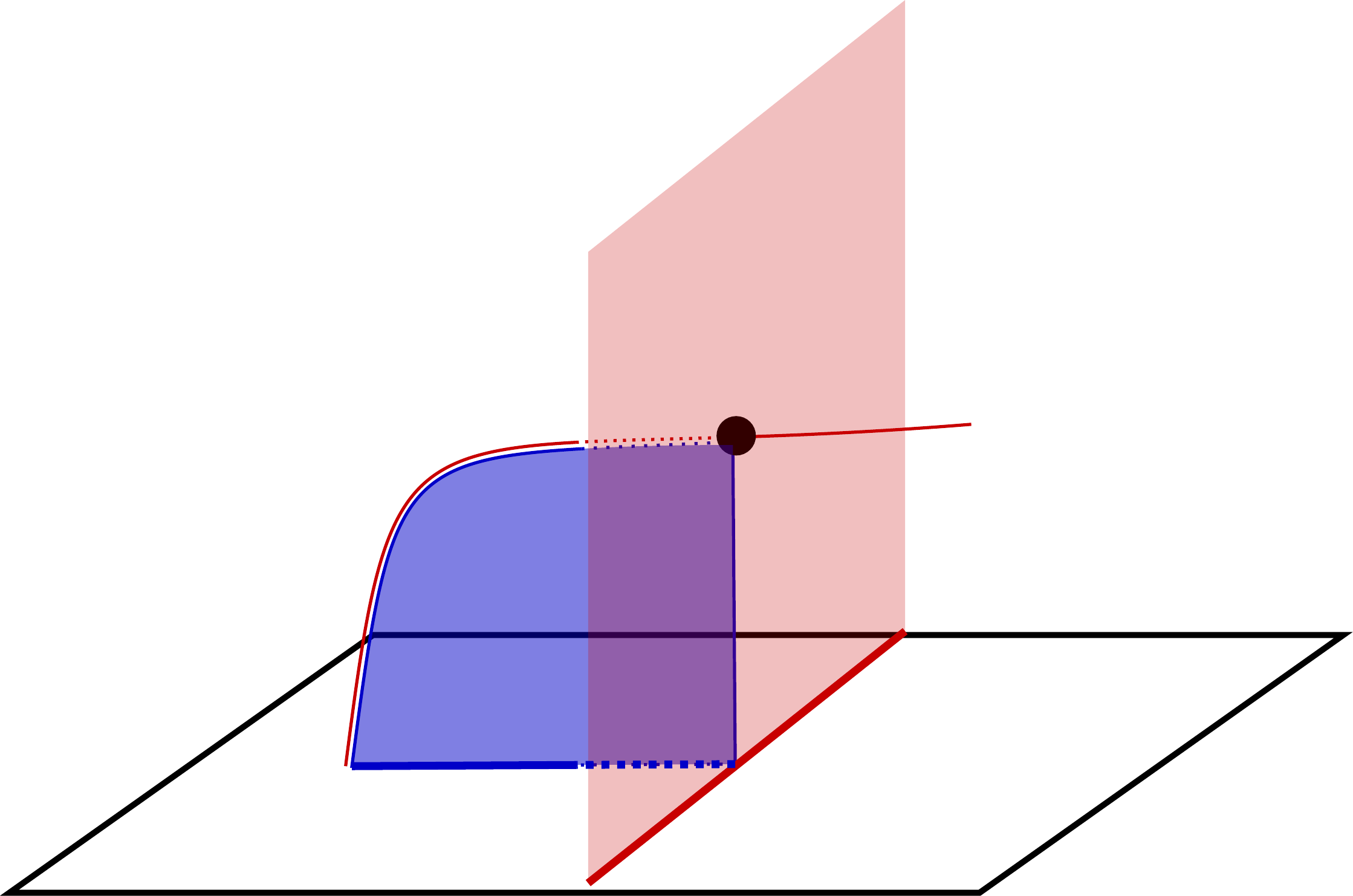}}
\put(155,45){$p$}
\put(135,30){$\Delta_p$}

\put(200,5){\includegraphics[width=37.5mm]{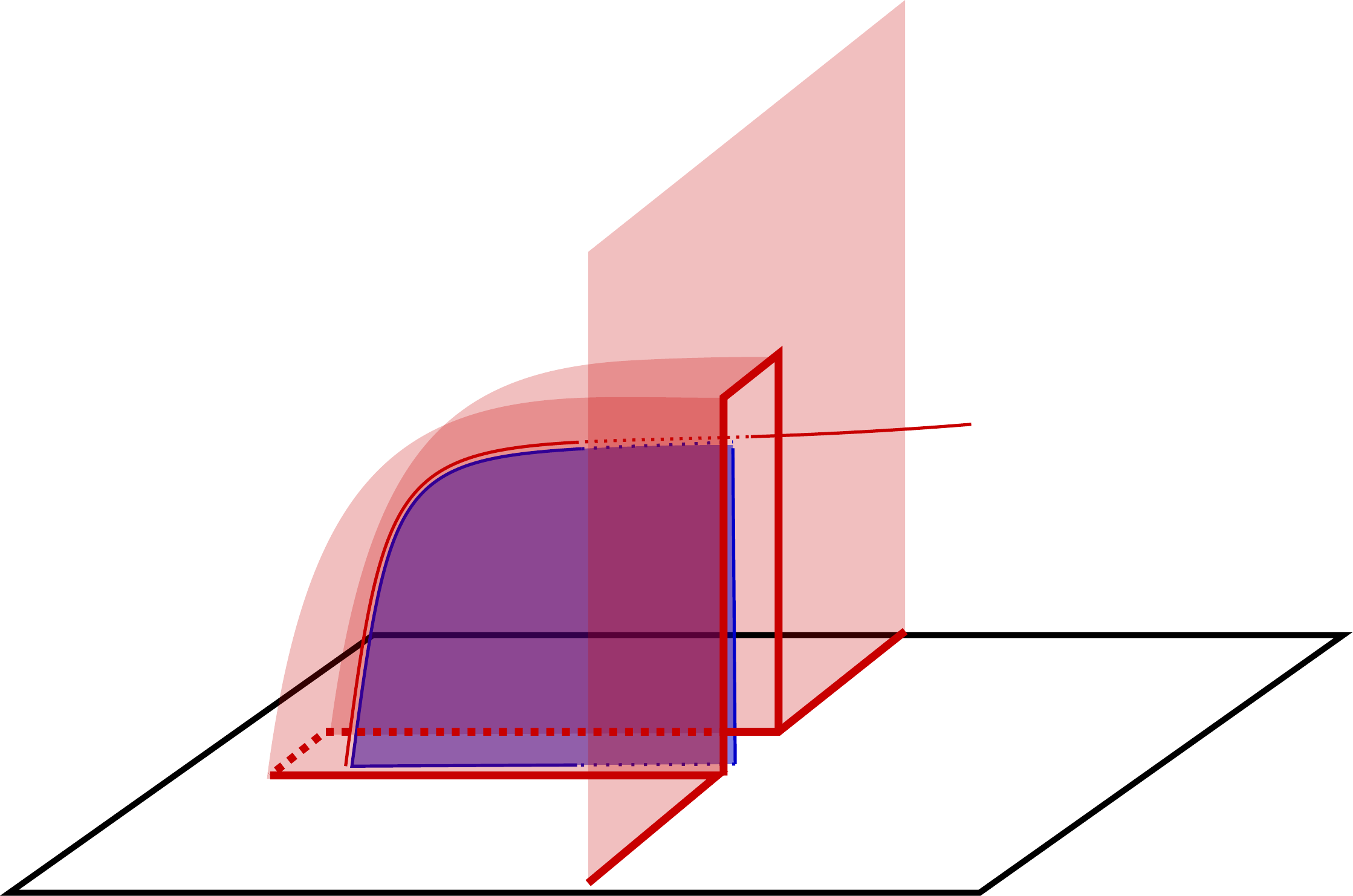}}

\end{picture}
\caption{A schematic for the relative Whitney trick. Possible singularities on the interior of the relative Whitney disk not depicted.}
\label{fig: rel Whitney trick}
\end{figure}
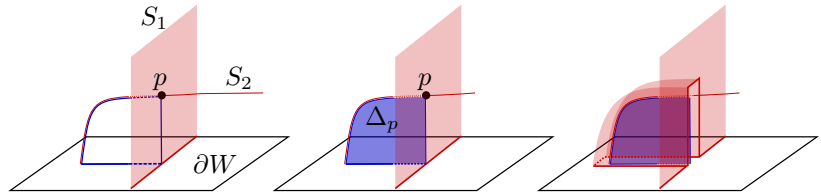

In comparison, the ordinary Whitney trick begins with two intersection points $p, q\in S_1\cap S_2$ with opposite sign that are paired by a Whitney disk. This immersed disk guides the (ordinary) Whitney trick, which is a regular homotopy of $S_1$ with the effect of removing both intersection points $p$ and $q$; see Figure~\ref{fig: Whitney Trick}. Any singularities present in the guiding Whitney disk will yield new singularities in the result of the Whitney move, and similarly for the relative Whitney trick.

\begin{figure}[h]
\begin{picture}(200,80)
%\put(200,80){*}
\put(0,0){\includegraphics[height=27.75mm]{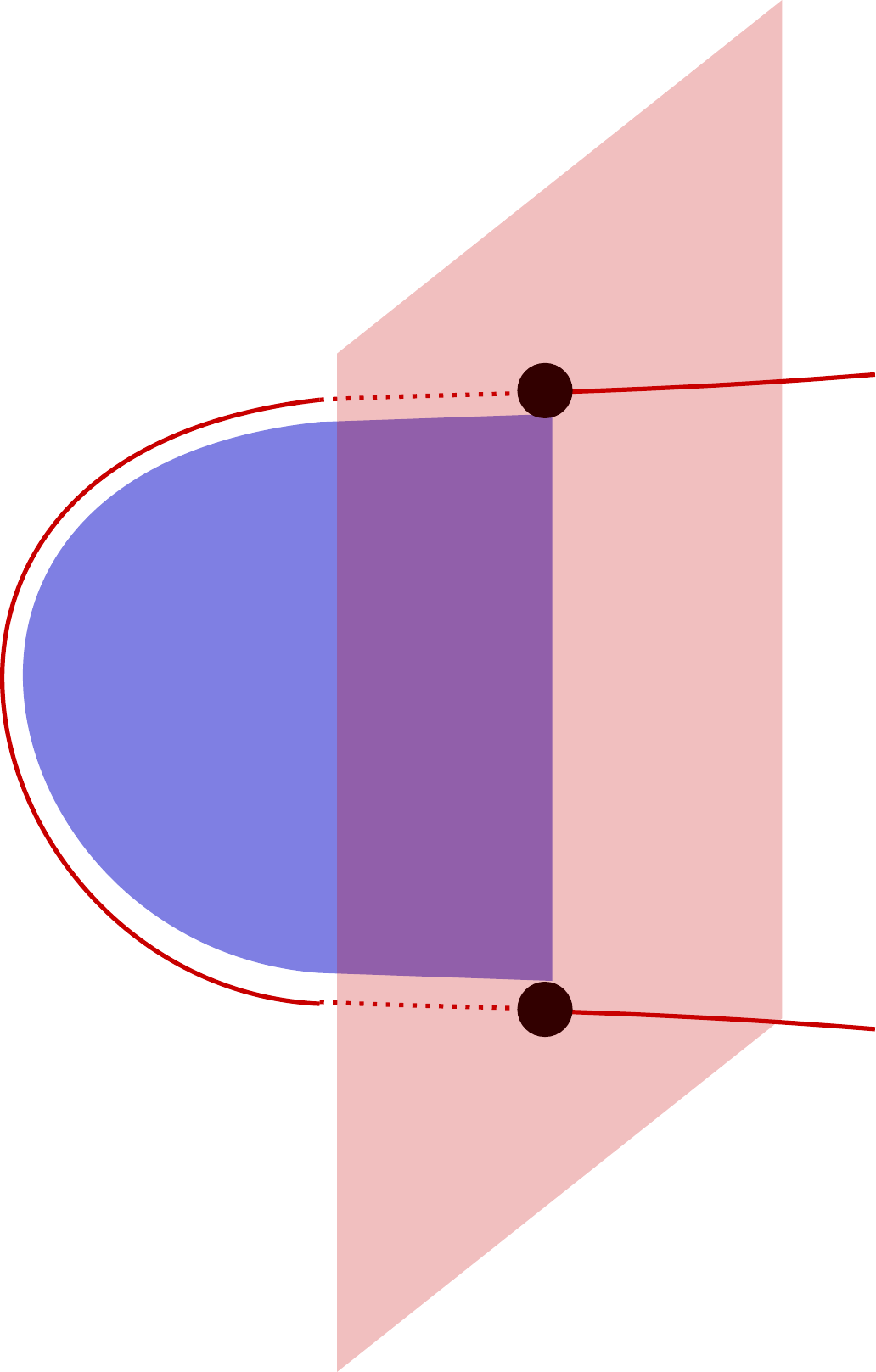}}
\put(120,0){\includegraphics[height=27.75mm]{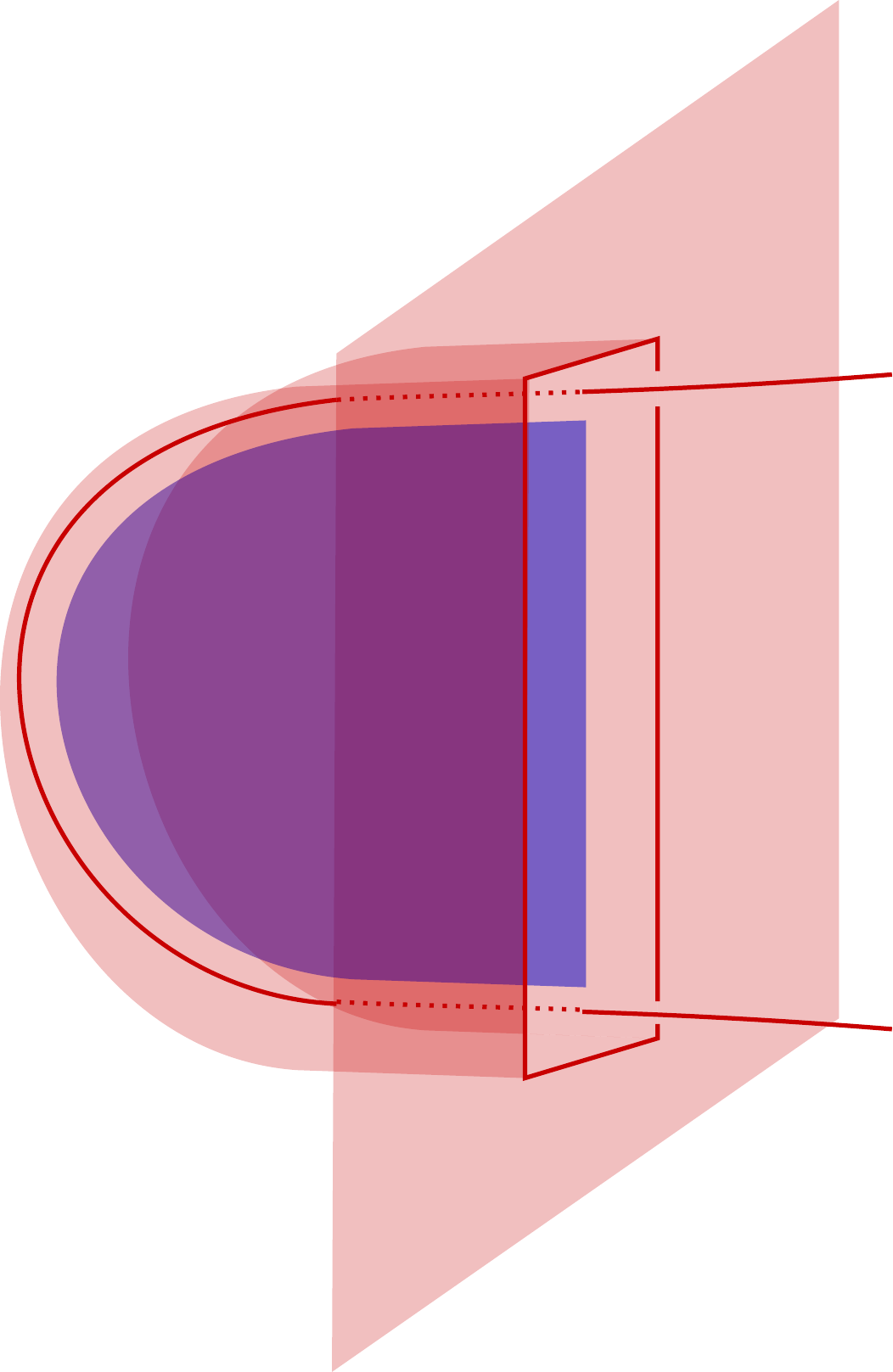}}
\put(30,65){$p$}
\put(30,10){$q$}
\end{picture}
\caption{A schematic for the ordinary Whitney trick. Possible singularities on the interior of the Whitney disk not depicted.}\label{fig: Whitney Trick}
\end{figure}

Our applications of the relative Whitney trick all concern the study of links in homology $3$-spheres. In this article, links $L=L_1\cup\dots\cup L_n$ will be ordered and oriented. Given a homology $3$-sphere $Y$, recall it bounds a contractible topological $4$-manifold~\cite[Theorem 1.4']{MR679066}, which we usually denote by $X$. The uniqueness of this 4-manifold, up to homeomorphisms fixing the boundary, follows from a standard argument using topological surgery theory and the $5$-dimensional $h$-cobordism theorem.

\subsection{Slicing links in homology spheres up to homotopy}

 We say a link $L$ is \emph{slice} if it bounds a collection of disjoint locally flat embedded disks~$D$ in $X$, and is moreover \emph{freely slice} if $\pi_1(X\sm D)$ is a free group generated by the meridians of $L$. Two links $L$ and $L'$ are \emph{freely homotopic} if there exists a continuous function $F\colon S^1\times\{1,\dots,n\}\times[0,1]\to Y$ with $F(S^1,i,0) = L_i$ and~$F(S^1,i,1) = L_i'$ for all $i=1, \dots, n$. 

The relative Whitney trick will be used in the proof of the following theorem, the first main result of the article.

\begin{theorem}\label{thm:main}
Every link in a homology sphere is freely homotopic to a freely slice link.
\end{theorem}

This result should be compared to the work of Austin-Rolfsen \cite{Austin-Rolfsen:1999-1}, who proved that any knot in a homology sphere is freely homotopic to a knot with trivial Alexander polynomial. Combined with a result of Freedman-Quinn~\cite[Theorem 11.7B]{Freedman-Quinn:1990-1}, the Austin-Rolfsen result shows that any knot in a homology sphere can be reduced by a free homotopy to a freely slice knot. (For a historical discussion of Alexander polynomial 1 knots see \cite[\S 21.6.3]{DETBook}.) To extend this to links we will use the relative Whitney trick, together with the methods of topological surgery theory and {a generalization of} results of Cha-Kim-Powell \cite{Cha-Kim-Powell:2020-1} {which give a sufficient condition for a link in a homology sphere to be freely slice.} 
This condition arises from a surgery-theoretic link slicing approach as we now outline.

A link $L=L_1\cup \dots \cup L_n$ in a 3-manifold $Y$ is a \emph{boundary link} if there exists a collection of  pairwise disjoint surfaces $F=F_1\cup\dots\cup F_n$ in $Y$, where $F_i$ is a Seifert surface for $L_i$. Such $F$ is called a \emph{boundary link Seifert surface} for the boundary link $L$. A surgery-theoretic strategy to slice $L$ is to construct a $4$-manifold $W$ with boundary $M_L$, the $0$-surgery on $L$, such that when we glue $2$-handles to the boundary, reversing the $0$-surgery, we obtain a contractible $4$-manifold. The desired slice disks are then the cocores of the $2$-handles.

In an attempt to construct such a $W$, begin by pushing a boundary link Seifert surface $F=F_1\cup\dots\cup F_n$ for $L$ into the contractible $4$-manifold bounded by $Y$. Excise a tubular neighbourhood of the pushed in surface to obtain a $4$-manifold $X_F$ with free fundamental group generated by the meridians of $L$, and whose boundary decomposes as $\partial X_F=(S^1\times \bigcup_i (F_i\setminus \mathring{D^2}))\cup X_L$, where $X_L$ denotes the link exterior. Let $H=H_1\cup\dots\cup H_n$ be a collection of $3$-dimensional handlebodies where $H_i$ has the same genus as $F_i$, and let $\varphi\colon F\cong \partial H$ denote a collection of homeomorphisms $\varphi_i\colon F_i\cong \partial H_i$. Form a new manifold $V_F:=X_F\cup (S^1\times H)$ by using the homeomorphism
\[
\operatorname{id}\times\varphi\colon S^1\times \bigcup_i (F_i\setminus \mathring{D^2})\cong S^1\times\bigcup_i (\partial H_i\setminus \phi(\mathring{D}^2))
\]
to glue only along this part of the boundary. The resulting $4$-manifold $V_F$ has boundary the 0-surgery $M_L$, as desired, and it is moreover possible to choose $\varphi$ so that $\pi_1(V_F)$ is still free and generated by the meridians of $L$. We would now like to know that there exist framed embedded $2$-spheres in $V_F$ that can be removed by surgery to kill the second homology.

As in Cha-Kim-Powell, we will obtain these embedded $2$-spheres via a theorem of Freedman-Quinn~\cite[Theorem 6.1]{Freedman-Quinn:1990-1}. This theorem states that the presence of certain configurations of framed immersed spheres (specifically \emph{$\pi_1$-null immersions of a union of transverse pairs with algebraically trivial intersections}; see Appendix \ref{appendix}) imply the existence of the embedded spheres required for surgery. To build these configurations, Cha-Kim-Powell find properly immersed disks in the exterior of the pushed in boundary link Seifert surface, bounded by a basis for that surface, and cap them off with properly embedded discs in $H\subseteq V_F$ that are dual to the cores of the handlebodies $H_i$ in $V_F$. They derive conditions on $L\subseteq S^3$ sufficient to ensure the described immersed sphere collection satisfies the hypotheses of \cite[Theorem 6.1]{Freedman-Quinn:1990-1}. In Section~\ref{sec:homotopictoslice}, we produce a straightforward generalization of their conditions for links in a homology sphere and in Appendix \ref{appendix} we confirm that links in homology spheres satisfying the generalized conditions are freely slice.

Thus our real challenge in the proof of Theorem~\ref{thm:main} becomes finding a way to freely homotope an arbitrary link to one satisfying the generalized Cha-Kim-Powell conditions. For this we will need a mechanism for separating properly immersed disk collections in $4$-manifold, at the expense of changing the link on the boundary by a free homotopy.  In Section~\ref{sec:separating}, we use the relative Whitney trick to achieve this goal, proving the following (in fact, we prove a more general statement in Proposition~\ref{prop:get immersed disks}).

\begin{proposition}\label{prop: disk sep in intro}
If $L$ is a link in a homology sphere $Y$, and $X$ is a contractible 4-manifold bounded by $Y$, then $L$ is freely homotopic to a link whose components bound disjoint locally flat immersed disks in $X$. Moreover, if $X$ is smooth,
then these disks may be smoothly immersed.\end{proposition}

We end this subsection by pointing out that it is not known whether Theorem~\ref{thm:main} can be extended to the smooth category or not. Concretely, we ask the following question.

\begin{question}\label{q:daemi}Let $L$ be a link in a homology sphere $Y$ and $X$ be a \emph{smooth} contractible 4-manifold bounded by $Y$. Is $L$ freely homotopic to a link $J$ in $Y$ so that the link $J$ bounds a collection of disjoint \emph{smooth} disks in $X$? Can this question be answered if $L$ is a knot?
\end{question}

In the case that $L$ is a knot, the question was answered affirmatively by the first named author, under the assumption that $X$ admits a handle structure with no $3$-handles~\cite[Theorem~1.5]{Davis2020b}. Compare this with~\cite[Remark~1.6]{Daemi:2020-1}, where Daemi shows the answer to Question~\ref{q:daemi} is negative if one requires $X$ to be only a homology ball. Indeed, he shows there is a knot in $Y\# -Y$, where $Y$ is the Poincar\'e homology sphere, such that the knot, even up to free homotopy, does not bound a smooth immersed disk in any smooth homology ball bounded by $Y\# -Y$.

We finally remark that if a link $L\subset Y$ bounds a collection of disjoint piecewise linear disks in a 4-manifold then that link may be changed by a homotopy in $Y$ to a link bounding a collection of disjoint smooth disks in that 4-manifold; see e.g.~\cite[Proof of Proposition 1.3]{Levine2016} for a technique to achieve this by ``tubing into the singularities''. Thus we compare Question~\ref{q:daemi} with results proving the non-existence of piecewise linear disks for certain knots in the boundaries of contractible $4$-manifolds~\cite{Zeeman:1964-1, Akbulut:1991-1, Levine2016, HLL2018, Zhou:2020-1}, and suggest our question is a natural refinement of the general problem of finding piecewise linear slice disks in contractible $4$-manifolds.

\subsection{Whitney tower concordance}

Our second application of the relative Whitney trick concerns \emph{homology concordance} of links. Links $L$ and $J$ in homology spheres $Y$ and $Y'$ are \emph{homology concordant} if there is a disjointly embedded union of locally flat  annuli each one bounded by a component of $L$ and a component of $J$ in a homology cobordism from $Y$ to $Y'$. It is conjectured by the first named author~\cite{Davis2020} that every link in a homology sphere is homology concordant to a link in $S^3$. This conjecture is particularly intriguing because the corresponding statement is known to be false in the smooth category~\cite{Levine2016, HLL2018, Zhou:2020-1,Daemi:2020-1}. Evidence for this conjecture was provided in~\cite{Davis2020b,Davis2020} and we provide a similar type of evidence in this article. Our evidence will come in the language of Whitney tower concordance.  A formal definition of a Whitney tower appears in Section~\ref{sect: towers}; see also \cite{CST2012}.  Informally, a \emph{Whitney tower} is a 2-complex given by starting with an immersed surface {in a 4-manifold }(a union of annuli {in a homology cobordism} in our case) and iteratively pairing up intersection points with Whitney disks, while accepting that each added Whitney disk will introduce more intersection points which must be paired with new Whitney disks. The \emph{order} of a Whitney tower records roughly how far into this tower one must go before seeing intersection points which are not paired with Whitney disks.  Two links $L$ and $J$ in homology spheres $Y$ and $Y'$ are \emph{order $k$ Whitney tower concordant} and we write $L\simeq_k J$ if they bound a collection of immersed annuli which extend to an order $k$ Whitney tower in a simply connected homology $S^3\times[0,1]$; cf.~\cite[Definition 3.2]{CST2012}. Hence if two links are homology concordant then they are order $k$ Whitney tower concordant for all $k$. We will use the relative Whitney trick to prove the following.
 
\begin{theorem}\label{thm:WhitneyTowerCobordism}
If $L$ is link in a homology sphere and $k$ is a nonnegative integer, then there is a link $J$ in $S^3$ such that~$L\simeq_k J$.
\end{theorem}

We point out an interesting consequence.  Consider now an $n$-component link $L$ in a homology sphere.  By Theorem~\ref{thm:WhitneyTowerCobordism} there is a link $J$ in $S^3$ so that $L$ and $J$ cobound an order $n$ Whitney tower. According to \cite[Theorem 4]{ST14}, if $n$-component links cobound an order $n$ Whitney tower, {then} this Whitney tower can be used to guide a sequence of Whitney tricks to produce a disjoint union of immersed annuli. As observed in \cite[Remark 3]{ST14}, these annuli can be made smooth. We arrive at the following corollary.

\begin{corollary}~\label{cor: disjoint annuli to S3}
If $L$ is link in a homology sphere $Y$, then there is a link $J$ in $S^3$ and a simply connected homology cobordism from $Y$ to $S^3$ such that the components of $L$ and $J$ cobound disjoint immersed annuli in the cobordism. Moreover, if the cobordism is smooth, then these annuli may be smoothly immersed.\end{corollary}

\subsection{Gordian distance and link homotopy}

Freedman-Teichner~\cite{FT2} say $L\subset Y$ is \emph{4D-homotopically trivial} if it bounds disjoint immersed disks in a contractible 4-manifold. Suppose $L$ intersects a $3$-ball $B\subset Y$ so that $(B,B\cap L)$ is orientation preserving homeomorphic to one of the tangles in Figure~\ref{fig: crossing}. A \emph{crossing change} is the local tangle replacement operation of replacing a positive crossing with a negative crossing, or vice-versa.  The reader may notice that this construction depends on the choice of identification of $B$ to the 3-ball, but this subtlety will not be relevant in our analysis.

\begin{figure}[h]
\begin{picture}(160,60)
%\put(160,60){*}
\put(0,0){\includegraphics[height=18.5mm]{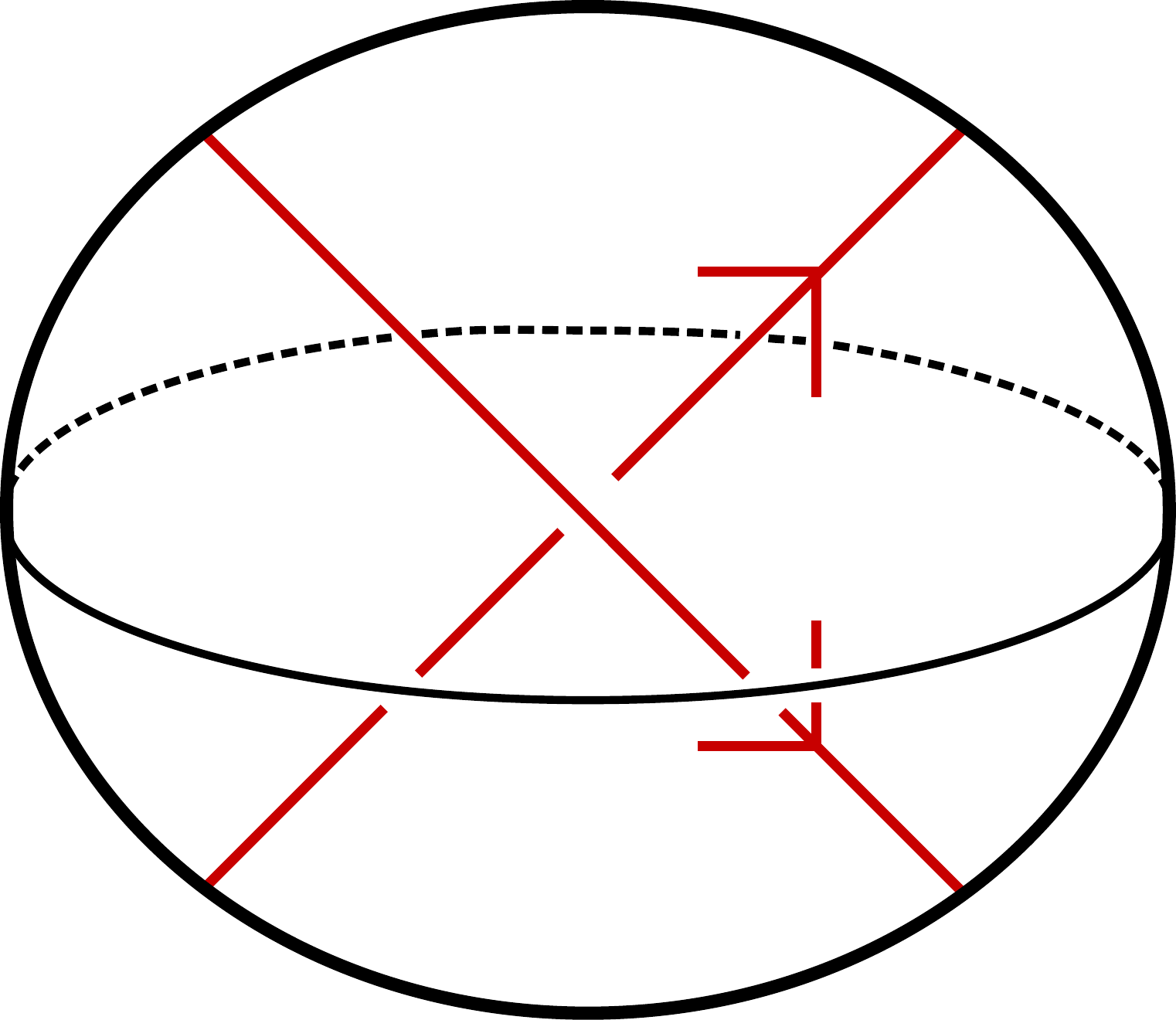}}
\put(100,0){\includegraphics[height=18.5mm]{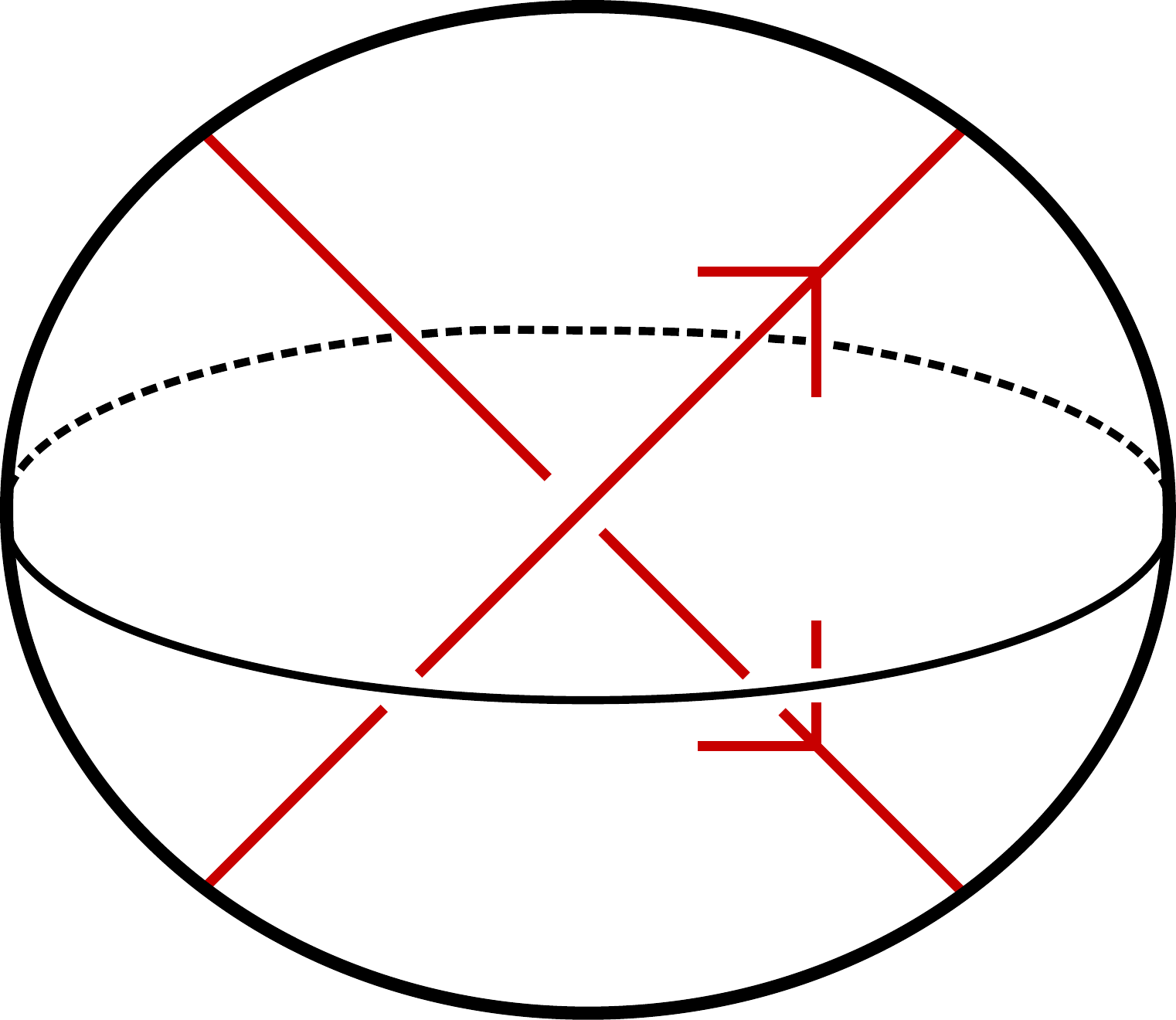}}
\end{picture}
\caption{Left: A positive crossing.  Right: A negative crossing.}\label{fig: crossing}
\end{figure} 

We define the \emph{homotopy trivializing number} $n_h(L)$ to be the minimum number of crossing changes required to transform $L$ into a 4D-homotopically trivial link. If $L$ is a link in $S^3$, then the homotopy trivializing number of $L$ is the minimum Gordian distance from $L$ to a homotopically trivial link (in the sense of Milnor~\cite{M1}).

In Subsection~\ref{subsec:nh=nd}, we use the relative Whitney trick to prove that this number can be computed by counting the number of intersections between distinct generically immersed disks bounded by the link in a contractible 4-manifold. More concretely, we prove the following.

\begin{proposition}\label{prop:nd=nh} 
If $L$ is a link in a homology sphere $Y$, then 
\[
n_h(L)= \min\left\{\textstyle{\sum_{i<j}\#(D_i\cap D_j)}\right\}
\]
where the minimum is taken over all collections of immersed disks $D_1\cup\dots \cup D_n$ in the contractible $4$-manifold bounded by $Y$, with boundary the link, and meeting one-another transversely.
\end{proposition}

In Subsection~\ref{subsec:boundnh}, we combine Proposition~\ref{prop:nd=nh} with results of Habegger-Lin~\cite{HL1} to obtain upper and lower bounds for $n_h(L)$. Moreover, for links with 3 or fewer components we completely determine $n_h(L)$. We remark that the upper bound we establish depends only on the linking numbers and the number of components, and in particular, it is independent of any of the higher order link homotopy invariants of Milnor \cite{M1}.

\begin{theorem}\label{thm: compute nh}
Let $L$ be a link in a homology sphere and $\Lambda(L):= \sum_{i<j}|\lk(L_i,L_j)|$.  If $L$ is a 2-component link, then $n_h(L) = \Lambda(L).$ If $L$ is a 3-component link, then 
$$n_h(L) = \begin{cases}
\Lambda(L)&\text{if }\Lambda(L)\neq 0\\
2&\text{if }\Lambda(L)= 0\text{ and }\mu_{123}(L)\neq 0\\
0&\text{otherwise.}\\
\end{cases}
$$
In general, there is some $C_n\in \N$ so that for every $n$-component link $L$,
$$
\Lambda(L)\le n_h(L)\le \Lambda(L)+C_n.
$$
\end{theorem}

\subsection{Comparison with the ordinary Whitney trick}
In high-dimensions, ordinary Whitney disks are usually assumed to be (and can always be arranged to be) embedded with interiors disjoint from the submanifolds containing their boundary arcs (with the framing condition guaranteed by the opposite signs of the paired intersections).  In dimension four, ordinary Whitney disks are generally assumed to contain self-intersections and intersections with other surfaces (as well as framing obstructions) which are then studied or controlled.

In dimension four, it is not always possible to find a Whitney disk.  Their existence is obstructed by the self-intersection invariant in a quotient of the fundamental group ring of the ambient 4-manifold. In comparison, under appropriate conditions, it is always possible to find relative Whitney disks.  For instance if the map on fundamental groups induced by the inclusion of the boundary is surjective then every point of intersection will admit a relative Whitney disk.  Notice that our setting of homology spheres bounding contractible 4-manifolds certainly satisfies this condition.

\subsection*{Organization of the paper}The reader will have noticed that Theorems~\ref{thm:main}, \ref{thm:WhitneyTowerCobordism}, and \ref{thm: compute nh} seem disparate.  They are related by their reliance on the relative Whitney trick as a means to separate immersed disks.  In Section~\ref{sect: rel Whit trick}, we give a precise description of the relative Whitney trick, and in Section~\ref{sec:separating}, we use it to separate immersed disks, proving Proposition~\ref{prop: disk sep in intro}.  These two sections are prerequisite to the remaining sections of the paper, which are then more or less independent of each other.  In Section~\ref{sec:homotopictoslice}, we state a sufficient condition, Theorem~\ref{thm:CKP}, for freely slicing a boundary link in a homology sphere generalizing \cite[Theorem A]{Cha-Kim-Powell:2020-1}, and use it to prove Theorem~\ref{thm:main}. The proof of Theorem~\ref{thm:CKP} uses the same ideas as appear in \cite{Cha-Kim-Powell:2020-1} and so is delayed until Appendix~\ref{appendix}.  In Section~\ref{sect: towers}, we apply the relative Whitney trick to the construction of Whitney towers and prove Theorem~\ref{thm:WhitneyTowerCobordism}.  In Section~\ref{sect: Homotopy trivializing}, we relate $n_h(L)$ to the minimum number of intersection points amongst immersed disks bounded by a link $L$ and prove Theorem~\ref{thm: compute nh}.

\subsection*{Notation and conventions}
In this article, links are ordered and oriented. All manifolds are assumed oriented and compact. We will denote by $-Y$ the manifold $Y$ with reversed orientation.

\subsection*{Acknowledgements} We are are indebted to Matthias Nagel for his contributions. We also thank Aru Ray for helpful conversations. We thank the anonymous referee for helpful suggestions. PO is supported by the SNSF Grant~181199.

\section{The relative Whitney trick}\label{sect: rel Whit trick}

In this section, we review some of the terminology and conventions we will use throughout the paper, working in the category of topological manifolds. We then give a detailed description of the relative Whitney trick.

\subsection{Conventions for topological manifolds}

We recall some definitions used in~\cite[Section 3]{powell20204dimensional}. A continuous map between topological manifolds is a \emph{generic immersion} if it is locally a smooth immersion (in particular a generic immersion is locally flat away from double points). For $(\Sigma,\partial \Sigma)$ a compact surface with boundary and $(W,\partial W)$ a compact $4$-manifold with boundary, any continuous map $(\Sigma,\partial \Sigma)\to (W,\partial W)$ is homotopic to a generic immersion; this follows from \cite[Theorem 8.2]{Freedman-Quinn:1990-1}, see~\cite[Proposition 3.1]{powell20204dimensional} for an argument. We will henceforth assume generic immersions of surfaces in $4$-manifolds have at worst double points. We will call the image of a generically immersed surface simply an \emph{immersed surface} from now on and call it an \emph{embedded surface} if it has no double points. If $p\in S$ is a double point of an immersed surface $S\subseteq W$ then by definition there is a neighbourhood $U$ of $p$ in $W$ such that $S\cap U$ is homeomorphic to $(\RR^2\times\{0\})\cup(\{0\}\times \RR^2)\subseteq\RR^4$. The submanifolds $S_1\subseteq S$ and $S_2\subseteq S$ identified with  $\RR^2\times\{0\}$ and $\{0\}\times \RR^2$ (respectively) under this homeomorphism are called \emph{local sheets} of $S$ near $p$. If $S\subseteq W$ is an immersed surface, it follows that $S$ has a neighbourhood homeomorphic to a self plumbing of a vector bundle over $S$; see~\cite[Remark 3.2]{powell20204dimensional}. Such a neighbourhood is called a \emph{tubular neighbourhood} and the bundle is called the \emph{normal vector bundle}.  All the surfaces in this section are compact and each component has boundary.  As a consequence, if the normal vector bundle is orientable then it is also trivializable. If a choice of trivialization has been made, we will say $S$ is \emph{framed}.

We say two (locally flat) immersed submanifolds are \emph{transverse} if they are locally transverse in the sense of smooth manifolds. Note that if connected embedded surfaces $N_1$ and $N_2$ meet transversely, this is equivalent to saying $N_1\cup N_2$ is immersed, so in particular it makes sense to talk about local sheets near transverse intersections of embedded surfaces. We say a proper submanifold $(N,\partial N)\subseteq (W,\partial W)$ is \emph{transverse} to the boundary if it is locally transverse to the boundary in the sense of smooth manifolds. Transversality is generic in the topological category as follows. Given locally flat proper submanifolds $N_1$, $N_2$ of a $4$-manifold $W$ that are transverse to the boundary $\partial W$, there is an isotopy of $W$, supported in any given neighbourhood of $N_1\cap N_2$ taking $N_1$ to a submanifold $N'_1$ that is transverse to $N_2$; see \cite{Quinn-EndsII,Quinn:transversality} and \cite[Section 9.5]{Freedman-Quinn:1990-1}.

\subsection{The relative Whitney trick}

Let $S$ be a generically immersed oriented surface in an oriented $4$-manifold, (possibly disconnected and possibly with corners). Assume $S$ is transverse to $\bdry W$. Then $\bdry S\subseteq W$ is embedded, so that $S\cap \bdry W\subseteq \bdry S$ is a union of embedded circles and arcs with endpoints at corners of $S$.

 Let $p$ be a double point of $S$ and write $S_1$, $S_2$ for two local sheets of the immersion near $p$. Assume both sheets belong to components of $S$ that have nonempty intersection with $\partial W$. For $i=1,2$, choose embedded arcs $\alpha_i\subseteq S$ running from some $q_i\in \bdry S\cap \bdry W$ to $p$; see Figure~\ref{fig: rel Whitney trick 2} for a labelled schematic.

\begin{figure}[h]
\begin{picture}(310,80)
%\put(310,80){*}
%\put(310,0){*}
%\put(0,80){*}
\put(0,5){\includegraphics[width=37.5mm]{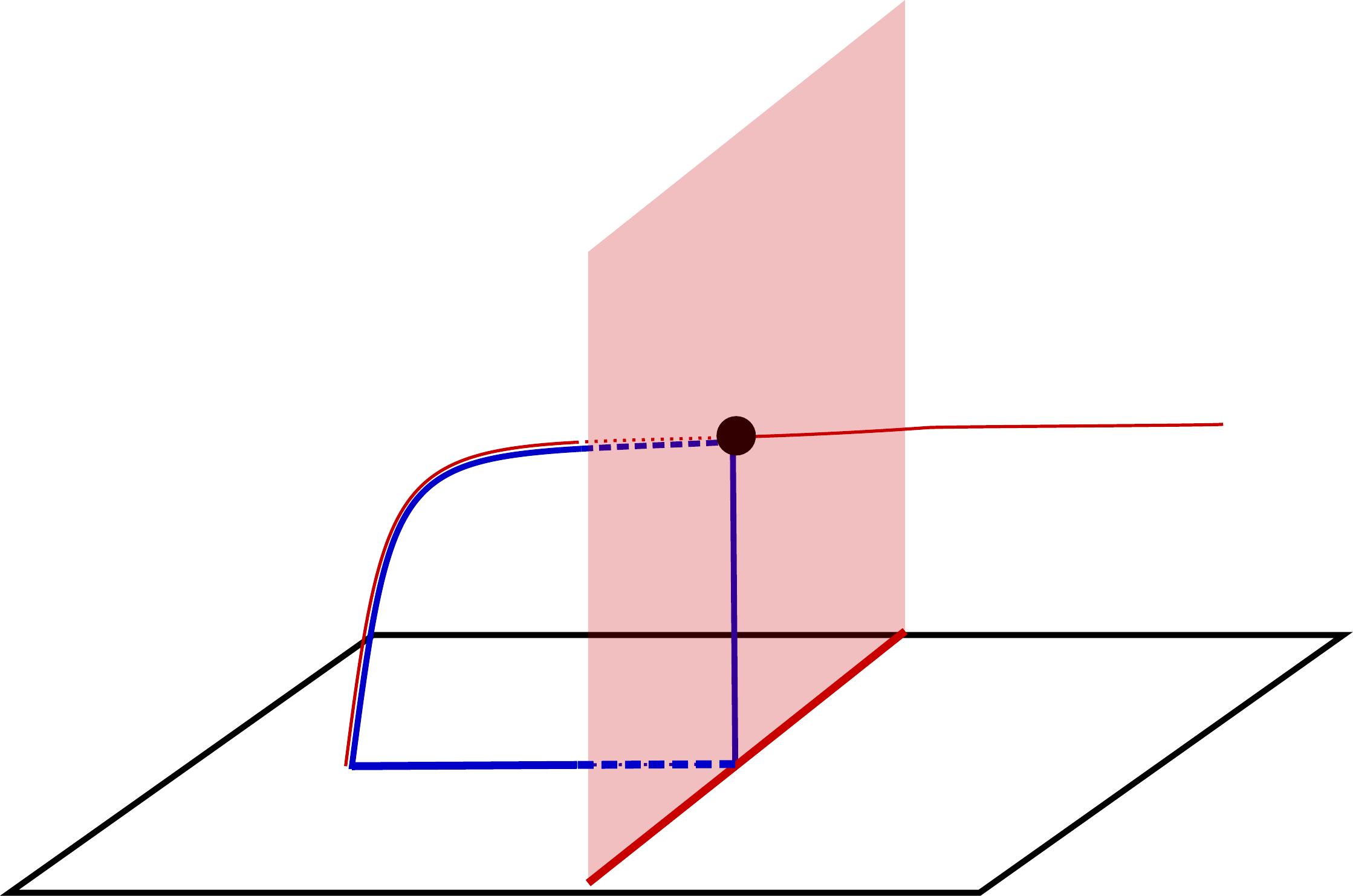}}

\put(70,12){$\bdry W$}
\put(50,68){$S_1$}
\put(82,45){$S_2$}
\put(55,45){$p$}

\put(55,9){$q_1$}
\put(20,12){$q_2$}
\put(59,30){$\alpha_1$}
\put(25,41){$\alpha_2$}
\put(34,10){$\alpha_3$}

\put(100,5){\includegraphics[width=37.5mm]{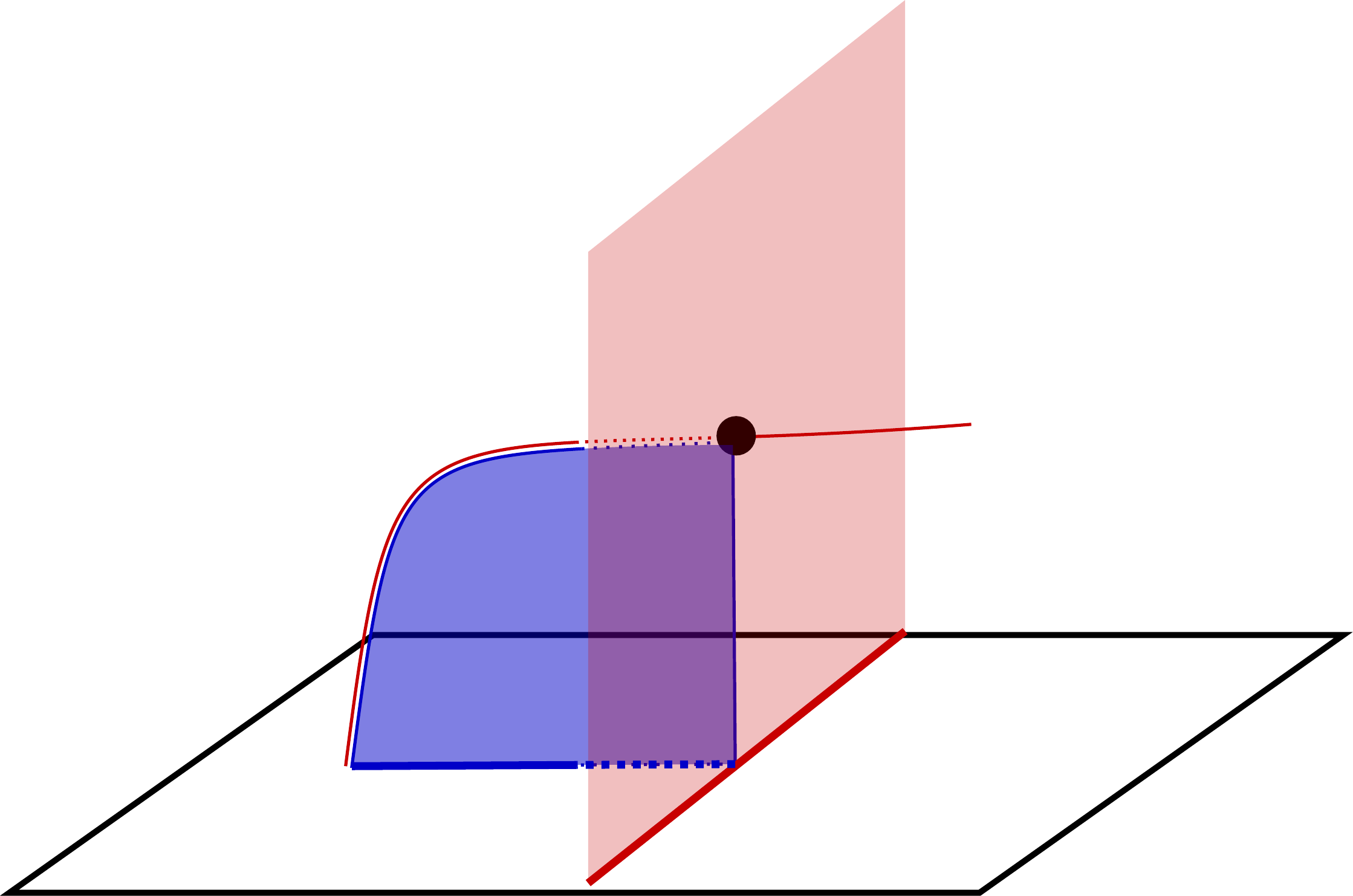}}
\put(155,45){$p$}
\put(135,30){$\Delta_p$}

\put(200,5){\includegraphics[width=37.5mm]{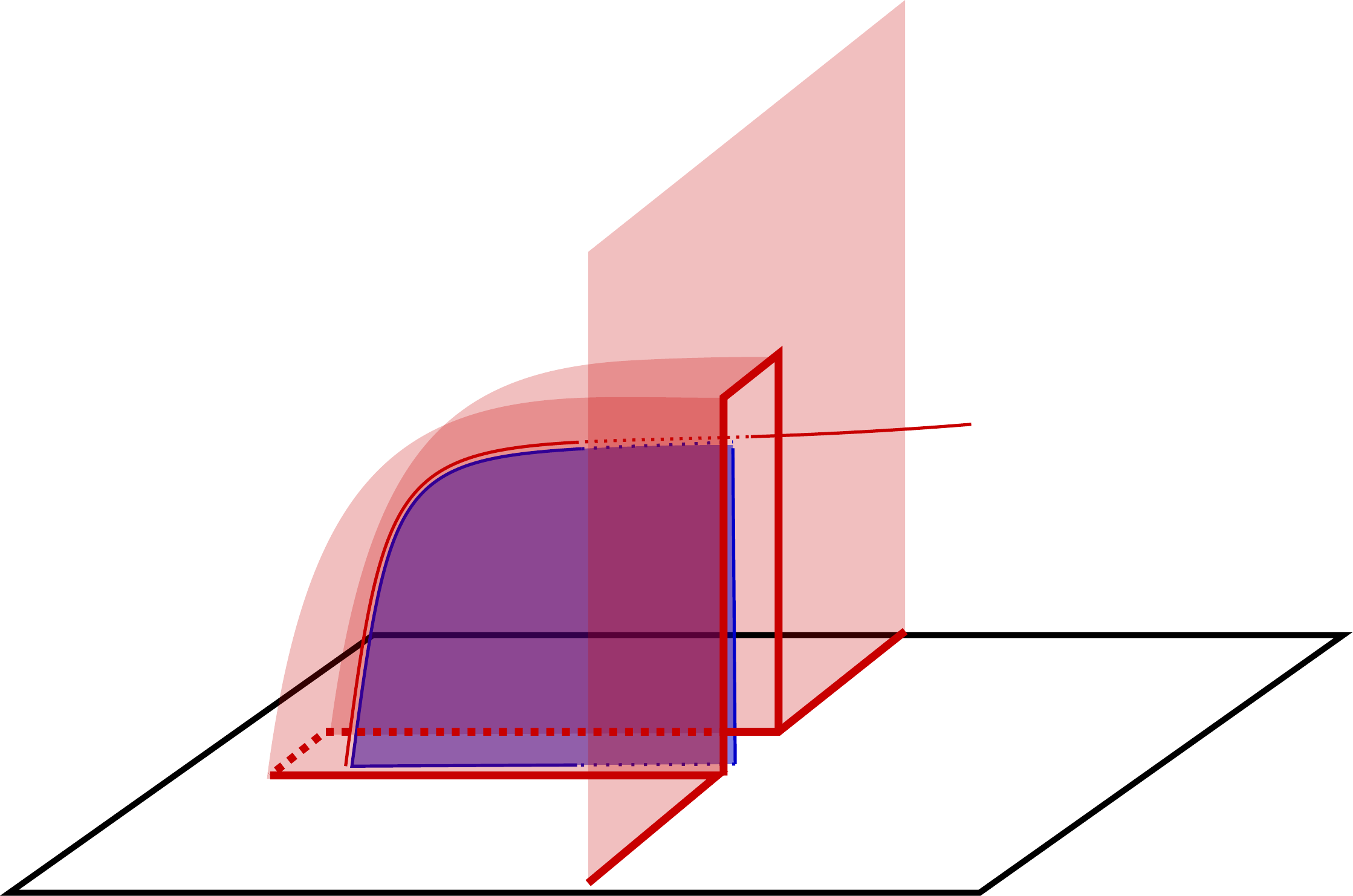}}
\end{picture}
\caption{Left to right:  A transverse point of intersection $p$ in $S_1\cap S_2$.  A relative Whitney disk, $\Delta_p$ for $p$.  The result of modifying $S_1$ by the relative Whitney trick.  }\label{fig: rel Whitney trick 2}
\end{figure} 

  We assume that $\alpha_i\cap(S_1\cup S_2)\subseteq S_i$, and that $\alpha_1$ and $\alpha_2$ are disjoint from each other and from all double points of $S$ other than their common endpoint at $p$.  Suppose that there exists an embedded arc $\alpha_3\subseteq \bdry W$ running from $q_2$ to $q_1$ with interior disjoint from $\bdry S$ such that the concatenation $\alpha_1*\alpha_2*\alpha_3$ is nullhomotopic in $W$.  Let $\Delta_p$ be an immersed disk in $W$ transverse to $S$ bounded by $\alpha_1*\alpha_2*\alpha_3$.
  
  \begin{definition} Let $S$ be an immersed oriented surface in an oriented 4-manifold $W$ and $p$ be a double point of $S$. Any disk $\Delta_p$ as above is called a \emph{relative Whitney disk} for $p$, and we call the corresponding $\alpha_3$ the \emph{relative Whitney arc} of $\Delta_p$.
  \end{definition}
  
  As we are working in the topological category, we take a moment to justify our later use of smooth concepts, such as tangent and normal vectors, by arranging the Whitney disk into a generic position near its boundary.
Let $U_1, U_2\subseteq S$ be closed regular neighbourhoods of $\alpha_1, \alpha_2\subseteq S$ respectively; each $U_i$ is thus homeomorphic to a closed disk, so that $\alpha_i$ is interior to $U_i$ except for its endpoint at $q_i$. We may assume, for some local sheets $S_1$ and $S_2$ of $p$, that $U_i\cap (S_1\cup S_2) = S_i$. By taking these neighbourhoods small enough we arrange that  $U_1\cap U_2 = \{p\}$ and $U_1$ and $U_2$ are disjoint from every double point of $S$ other than $p$.  As a consequence $U_1$ and $U_2$ are each embedded in $W$. Since $S$ is immersed, so is $U:=U_1\cup U_2$. Thus $U$ has a tubular neighbourhood $N_U\subseteq W$ homeomorphic to the result of plumbing together $U_1\times D^2$ and $U_2\times D^2$ and we may regard $U_1$ and $U_2$ as smoothly embedded disks in $N_U$ with the smooth structure pulled back from this plumbing. After changing $\Delta_p$ by a homotopy fixing its boundary we may assume that $\Delta_p\cap N_U\subseteq N_U$ is a smooth submanifold with corners. In the complement of the interior of $N_U$ we now see a properly immersed disk, $\Delta_p\sm \int(N_U)$.  Let $N\subseteq W\sm \int(N_U)$ be its tubular neighbourhood. We now have that $N_\Delta:=N_U\cup N$  is the result of gluing together $N_U$ and a tubular neighbourhood of $\Delta_p\smallsetminus \int(N_U)$ (a self plumbing of a disk) along a 3-ball in their shared boundary.  Being the result of gluing together two smooth 4-manifolds along a shared submanifold in their boundary, $N_\Delta$ is a smooth 4-manifold and $\Delta_p$, $U_1$ and $U_2$ are smoothly immersed surfaces in $N_\Delta$.

Introducing some notation, for points $a,b,c\in \RR^2$ denote by $\lineseg{a}{b}$ the straight line segment between $a$ and $b$, and by $\Delta_{abc}$ the triangle with vertices at $a$, $b$, and $c$. We will parametrize $\Delta_p$ by the triangle $\Delta_{xyz}\subseteq \RR^2$ of Figure~\ref{fig: triangle again}, sending the line segments $\lineseg{x}{y}$, $\lineseg{y}{z}$, and $\lineseg{z}{x}$ to $\alpha_1$, $\alpha_2$, and $\alpha_3$ respectively. By taking the neighbourhood $N_U$ small enough we may assume that $\Delta_p\cap N_U$ is parametrized by the region $R$ depicted in Figure~\ref{fig: triangle again}. We also construct a slightly larger immersed disk $\Delta_p^+\subset N_\Delta$ as follows. Let $w$ be the normal vector to $\alpha_2\subseteq N_\Delta$ given by the outward tangent to $\Delta_p\subseteq N_\Delta$. Define $\Delta_{p}^+$ by extending $\Delta_p$ in the $w$-direction along the length of the arc $\alpha_2$. We parametrize $\Delta_{p}^+$ by the larger triangle $\Delta^+_{xyz}:=\Delta_{xy^+z^+}\subseteq \RR^2$ of Figure~\ref{fig: triangle again}. Note that at $p$, the vector $w$ is tangent to $\alpha_1$ and to $S_1$.

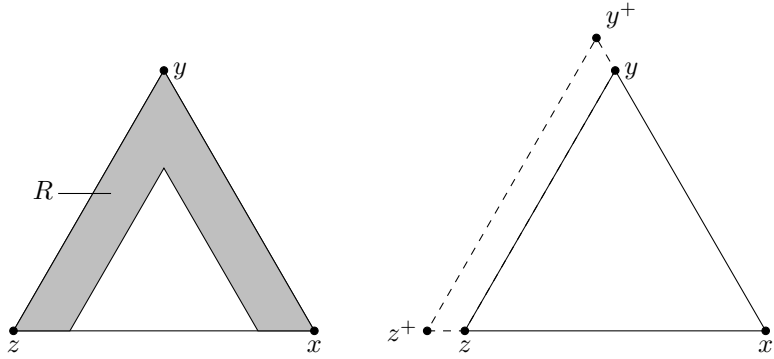
\begin{figure}[h]
\begin{tikzpicture}[scale=.5]
%\tikzset{
%    partial ellipse/.style args={#1:#2:#3}{
%        insert path={+ (#1:#3) arc (#1:#2:#3)}
%    }
%}

\node[] at (4,0) (x) {};
\node[] at ({4-1.5},0) (x') {};
\node[] at (0,{8*sin(60)}) (y) {};
\node[] at (-4,0) (z) {};
\node[] at ({0}, {(8-2*1.5)*sin(60)}) (y') {};
\node[] at ({-4+1.5},0) (z') {};

\draw[red](y)--(y');

\node[above] at ({-cos(60)}, {9*sin(60)}) (y+) {\phantom{$y^+$}};

\node[below] at (z) {$z$};
\node[below] at (z') {\phantom{$z'$}};
\node[below] at (x) {$x$};
\node[right] at (y) {$y$};
\node[below] at (z) {\phantom{$z$}};

\draw(x.center)--(y.center)--(z.center)--(x.center);
%\draw[fill=lightgray] (z.center)--(y.center)--(x.center)--(x'.center)--(y'.center){[rounded corners]--(z'.center)--}(z.center);

\draw[fill=lightgray] (z.center)--(y.center)--(x.center)--(x'.center)--(y'.center)--(z'.center)--(z.center);
%\draw[fill=lightgray] (z.center)--(y.center)--(x.center)--(x'.center){[rounded corners]--(y'.center)}--(z'.center)--(z.center);

\draw[fill=black] (x) circle (.1);
\draw[fill=black] (y) circle (.1);
\draw[fill=black] (z) circle (.1);
%\draw[fill=black] (y') circle (.1);
%\draw[fill=black] (z') circle (.1);

\node[left] at (-2.7,3.75) {$R$};
\put(-40,52){\line(10,0){20}}

%%%\end{tikzpicture}
%%%\begin{tikzpicture}[scale=.5]
%%%\tikzset{
%%%    partial ellipse/.style args={#1:#2:#3}{
%%%        insert path={+ (#1:#3) arc (#1:#2:#3)}
%%%    }
%%%}

\begin{scope}[shift={(12,0)}]
\node[] at (4,0) (x) {};
\node[] at (0,{8*sin(60)}) (y) {};
\node[] at (-4,0) (z) {};
\node[] at ({-cos(60)}, {9*sin(60)}) (y+) {};
\node[] at (-5,0) (z+) {};

\node[below] at (z) {$z$};
\node[left] at (z+) {$z^+$};
\node[below] at (x) {$x$};
\node[right] at (y) {$y$};
\node[above right] at (y+) {$y^+$};

\draw(x.center)--(y.center)--(z.center)--(x.center);
\draw[dashed] (z+.center)--(z.center)--(y.center)--(y+.center)--(z+.center);

\draw[fill=black] (x) circle (.1);
\draw[fill=black] (y) circle (.1);
\draw[fill=black] (z) circle (.1);
\draw[fill=black] (y+) circle (.1);
\draw[fill=black] (z+) circle (.1);

%\put(-50,50){$Q$}
%\put(-40,52){\line(10,0){20}}
\end{scope}

\end{tikzpicture}
\caption{Left: The triangle $\Delta_{xyz}$ used to parametrize $\Delta_p$, together with $R$, the preimage of $\Delta_p\cap N_U$. Right: A larger triangle $\Delta_{xyz}^+ := \Delta_{xy^+z^+}$.}
  \label{fig: triangle again}
\end{figure}

By restricting the normal bundle $E$ of $\Delta^+_p$ to $\alpha_1*\alpha_2$ we obtain a 2-plane bundle $E|_{\alpha_1*\alpha_2}$ over $\lineseg{x}{y}\cup\lineseg{y}{z}$. At each point in $\alpha_i$ let $v_i$ be the tangent direction in $S_i$ normal to $\alpha_i$, and let $u_i$ be a common normal to both $\Delta_p$ and $S_i$, chosen to vary continuously (i.e.~to form a section of $E|_{\alpha_i}$). Moreover, choose $u_i$ so that $u_1=v_2$ and  $u_2=v_1$ at $p$.  Together these give a framing of $E|_{\alpha_1*\alpha_2}$.  Since $\Delta^+_{xyz}$ deformation retracts to $\lineseg{x}{y}\cup\lineseg{y}{z}$, this framing extends to a framing of $E$.  By the tubular neighbourhood theorem we find an identification of a (possibly smaller) tubular neighbourhood of $\Delta^+_p$ with a plumbed disk, which we parametrize by an immersion $\Phi:\Delta^+_{xyz}\times \RR^2 \looparrowright N_\Delta$ such that:

\begin{itemize}[leftmargin=0.9cm]\setlength\itemsep{0em}
\item $\Phi\left(\lineseg{x}{y}\times\{(0,0)\}\right)=\alpha_1$, $\Phi|_{\lineseg{x}{y\textcolor{blue}{^+}}\times \RR\times \{0\}}$ is an embedding with image in $U_1$, 

\item $\Phi\left(\lineseg{x}{y^+}\times\{(0,0)\}\right)\subseteq U_1$ is $\alpha_1$ extended slightly on $S_1$ along the direction tangent to $\alpha_1$ at $p$,

\item $\Phi(\lineseg{y}{z}\times\{(0,0)\})=\alpha_2$, $\Phi|_{\lineseg{y}{z}\times \{0\}\times \RR}$ is an embedding with image in $U_2$,

\item otherwise $\Phi\left(\left(\lineseg{x}{y^+}\cup\lineseg{y}{z}\right)\times \RR^2\right)$ is disjoint from $S$, and

\item $\Phi(\lineseg{z}{x}\times\{(0,0)\}) = \alpha_3$ and $\Phi\left(\lineseg{z}{x}\times \RR^2\right)\subseteq \bdry W$ is a tubular neighbourhood of $\alpha_3$.
\end{itemize}

We now take $U_1$, remove a neighbourhood of $\alpha_1$, and replace it with pushed-off copies of $\Delta^+_p$ together with a thickened $\alpha_2$, pushed up in the outwards facing tangent direction to $\Delta_p$. Precisely, we form the following:

\begin{multline*}
U_1':=\left(U_1\smallsetminus\Phi\left(\lineseg{x}{y^+}\times [-1,1]\times \{0\}\right)\right) \cup
\Phi\left(\Delta^+_{xyz}\times \{-1,1\}\times\{0\}\right) 
\\\cup \Phi\left(\lineseg{y^+}{z^+}\times[-1,1]\times\{0\}\right).
\end{multline*}

We will refer to the act of modifying $S$ by replacing $U_1$ by $U_1'$ the \emph{relative Whitney trick} along $\Delta_p$.

\begin{remark}\label{rmk: affect of RWT}
Here are some observations.
\begin{enumerate}[leftmargin=0.9cm, font=\upshape]\setlength\itemsep{0em}
\item $U_1'$ and $U_2$ are disjoint.

\item The image $\Phi(\Delta^+_{xyz}\times[-1,1]\times\{0\})$ parametrizes a homotopy from $U_1$ to $U_1'$. If $\Delta_p$ is not embedded then $U_1'$ is not embedded.

\item The homotopy describing the relative Whitney trick changes $\bdry S$ by a \emph{finger move} as depicted in Figure~\ref{fig: rel Whitney trick on bdry again}. In general, a finger move between embedded arcs $A_1$ and $A_2$ in a $3$-manifold is band-sum operation from $A_1$ to a meridional circle of $A_2$. This operation is specified by a choice of embedded arc from $A_1$ to $A_2$, together with a choice of framing for that arc relative to a fixed framing on the boundary of the arc. In our case $\alpha_3$ is this arc from $A_1$ to $A_2$, and the framing is determined by $\Phi$. For the purposes of this paper, we will not need to keep track of this specific framing, but we note that for future applications it might be desirable to do so.

\item If $F$ is an immersed surface in $W$ (for example a subsurface of $S$), meeting the interior of $\Delta_p$ transversely $n$ times, then the relative Whitney trick adds $2n$ points of intersection to $U_1'\cap F$. Moreover, for every point of self-intersection of $\Delta_p$, there are four new points of self intersection are added to $U_1'$ by performing the relative Whitney trick.

\item \label{item: Rel Whit Trick Smooth} If $W$ is a smooth 4-manifold in which $S$ and $\Delta_p$ are smoothly immersed, then the result of modifying $S$ by the Whitney trick using $\Delta_p$ is still smoothly immersed (after smoothing corners).

\end{enumerate}
\end{remark}

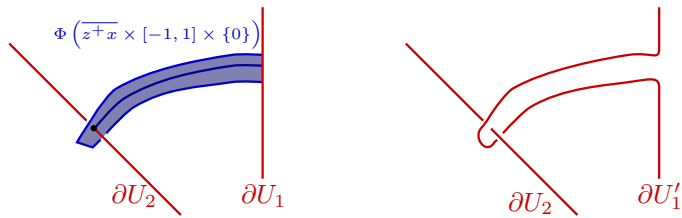
\begin{figure}[h]
\begin{picture}(300,95)
\put(0,5){\includegraphics[height=27.75mm]{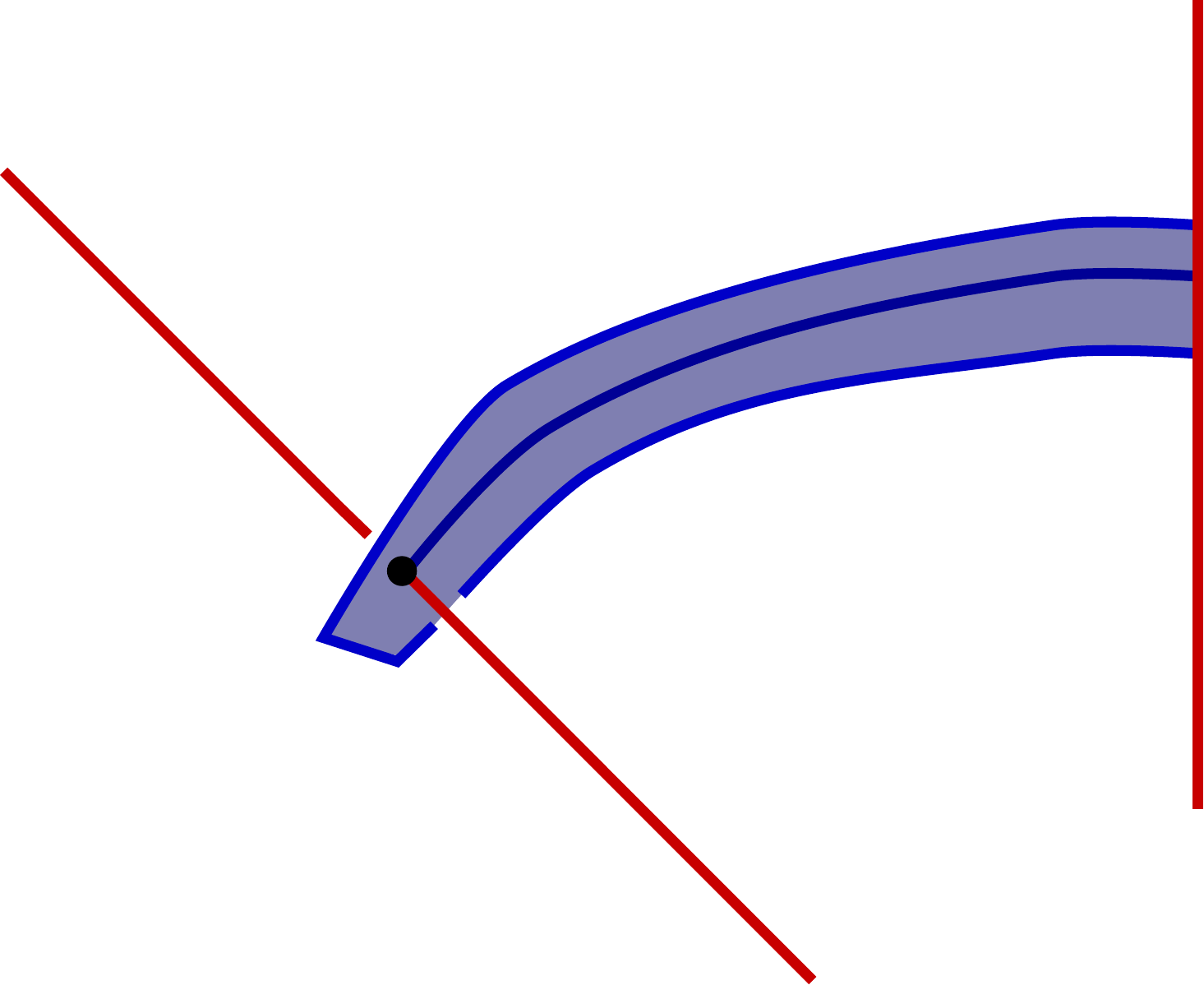}}
\put(39,10){\textcolor{diagramred}{$\bdry U_2$}}
\put(88,10){\textcolor{diagramred}{$\bdry U_1$}}
\put(17,72){\textcolor{diagramblue}{\tiny{$\Phi\left(\lineseg{z^+}{x}\times[-1,1]\times\{0\}\right)$}}}
%\put(-47,37){\textcolor{diagramgreen}{\tiny{$\Phi\left(\lineseg{z}{z'}\times\{0\}\times[-1,1]\right)$}}}

%\put(150,5){\includegraphics[height=27.75mm]{relWhitMoveOnBdry2.pdf}}
\put(150,5){\includegraphics[height=27.75mm]{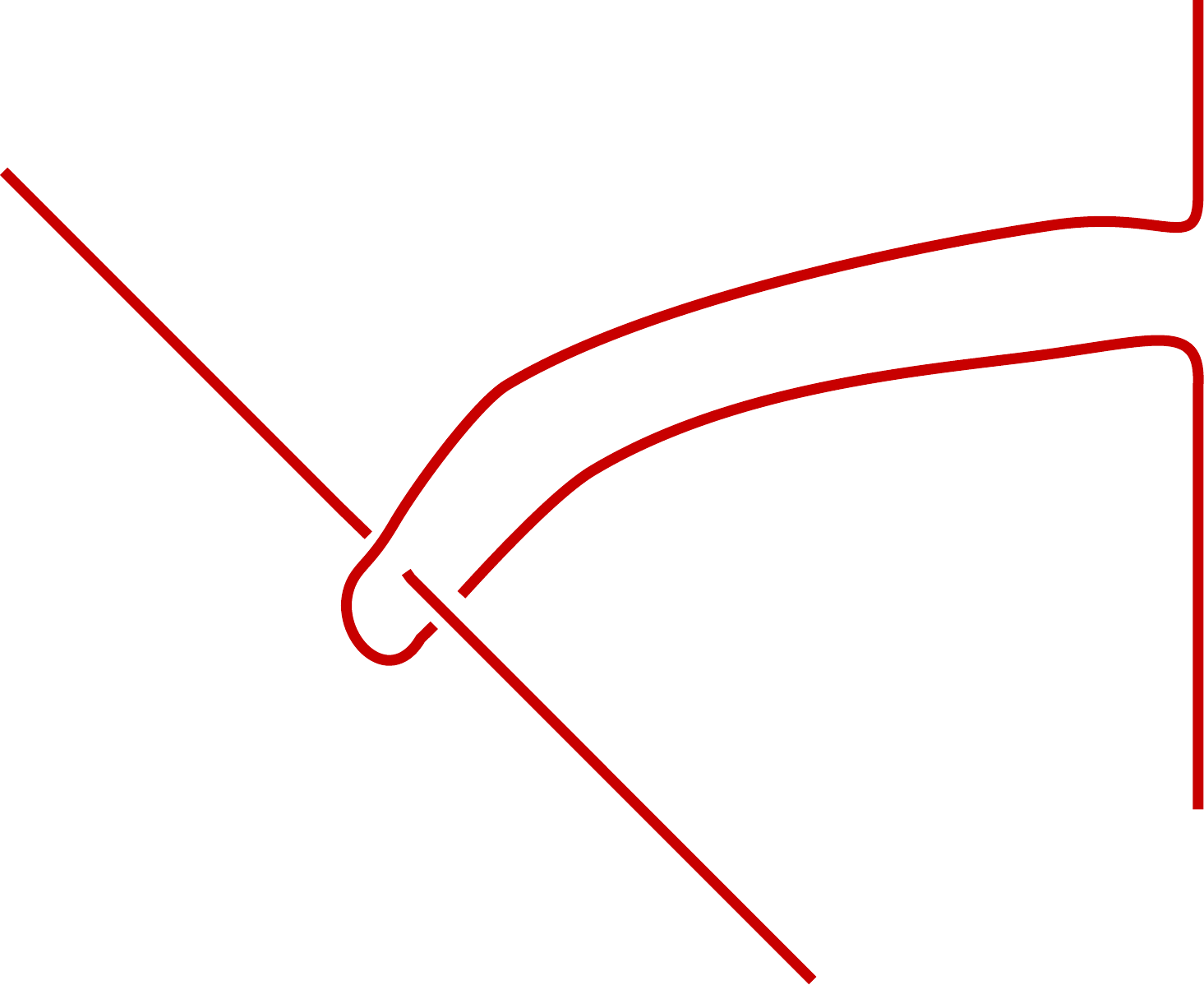}}
\put(189,7){\textcolor{diagramred}{$\bdry U_2$}}
\put(238,10){\textcolor{diagramred}{$\bdry U_1'$}}

\end{picture}
\caption{The relative Whitney trick affects
$\bdry S \cap \bdry W$ by a finger move.}
\label{fig: rel Whitney trick on bdry again}
\end{figure}

\section{Separating an immersed disk collection}\label{sec:separating}

Let $L$ be a link in a 3-manifold, whose components bound immersed disks in a bounded 4-manifold. In this section, we explain how to use the relative Whitney trick to homotope away all intersections between these disks.  

\begin{lemma}\label{lem:warmup} Let $W$ be a 4-manifold, $Y$ be a connected 3-manifold in $\bdry W$, and $L=L_1\cup L_2$ be a link in $Y$ whose components are nullhomotopic in $W$. If the inclusion induced map $\pi_1(Y)\to\pi_1(W)$ is surjective, then there exists a link $J=J_1\cup J_2$ in $Y$ which is freely homotopic to $L$ and bounds disjoint immersed disks in $W$. If $W$ is smooth, then these disks may be smoothly immersed. 

Moreover, if {$D_1$ and $D_2$}
 are transverse immersed disks in $W$ bounded by $L$, then we may choose a homotopy from $L$ to $J$ which restricts to an isotopy on each component of $L$ and changes a crossing between $L_1$ and $L_2$ exactly once for each point in $D_1 \cap D_2$.
\end{lemma}

\begin{proof}
Let $D_1$ and $D_2$ be transverse immersed disks bounded by $L$.  We will show how to reduce the number $|D_1\cap D_2|$ by one, via the relative Whitney trick.    Notice that each point in $D_1\cap D_2$ can be thought of as a double point in $D_1\cup D_2$, and so it makes sense to apply the relative Whitney trick as outlined in Section~\ref{sect: rel Whit trick}.  As in Figure~\ref{fig: rel Whitney trick on bdry again} this relative Whitney trick 
has the effect of changing a single crossing between $L_1=\bdry D_1$ and $L_2=\bdry D_2$.

First, we alter $D_1$ and $D_2$ by a homotopy which is constant on the boundary so that $D_1$ and $D_2$ intersect transversely.  Let $p\in D_1\cap D_2$, $q_1\in L_1$, $q_2\in L_2$, $\alpha_1$ be an embedded arc in $D_1$ running from $q_1$ to $p$, and $\alpha_2$ be an embedded arc in $D_2$ running from $p$ to $q_2$.  We may assume that $\alpha_1$ and $\alpha_2$ are disjoint from all double points of $D_1$ and $D_2$ and miss all points in $D_1\cap D_2$ other than their common endpoint at $p$.  
Since $\pi_1(Y)\to \pi_1(W)$ is surjective, there is an  embedded arc $\alpha_3$ in $Y$ running from $q_2$ to $q_1$ so that the concatenation $\alpha_1*\alpha_2*\alpha_3$ is nullhomotopic in $W$.  Thus, there is a relative Whitney disk $\Delta_p$ for $p$. 

Performing the relative Whitney trick using $\Delta_p$ to modify $D_1$ will add two points to $D_1\cap D_2$ for each point in $\Delta_p\cap D_2$.  As the objective is to reduce the number $|D_1\cap D_2|$, our next goal must be the removal of all points in $\Delta_p\cap D_2$. Let $r\in \Delta_p\cap D_2$. Pick an embedded arc $\beta$ in $\Delta_p$ from $r$ to a point $t$ interior to $\alpha_2\subseteq D_2$ and disjoint from all double points and points of intersection. Perform a $4$-dimensional \emph{finger move} on $D_2$ along $\beta$ (cf.~\cite[\textsection 1.5]{Freedman-Quinn:1990-1}); a visualization of how this move changes $D_2$ appears in Figure~\ref{fig:FingerMove}.
We will refer to the modified $D_2$ by the same name. The cost of the finger move is to add two new points of self-intersection to $D_2$, but this will not concern us. Repeat this process at each point in $\Delta_p\cap D_2$.  A similar procedure may be used to remove points in $D_1\cap \Delta_p$, although this is not required in the proof of the lemma.

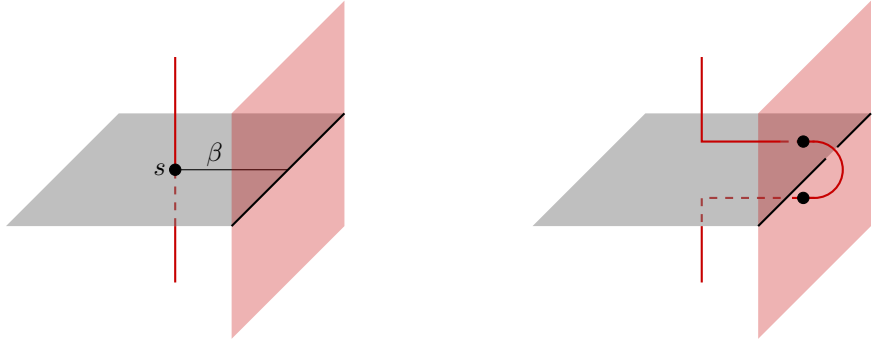
\begin{figure}
\begin{tikzpicture}

\begin{scope}[shift={(-2,0)}, scale=1.5]
\draw[diagramred, thick] (1.5,0)--(1.5,-.5);
\draw[diagramred, thick, dashed] (1.5,0)--(1.5,.5);
\draw[opacity=.5, fill=gray, draw=none] (0,0)--(2,0)--(3,1)--(1,1)--(0,0);
\draw[diagramred, thick] (1.5,1.5)--(1.5,.5);
\draw (1.5,.5)--(2.5,.5);

\draw[opacity=.3, fill=diagramred, thick, draw=none] (2,1)--(3,2)--(3,0)--(2,-1)--(2,1);
\draw[black, thick] (2,0)--(3,1);

\draw[fill=black] (1.5,.5) circle (.05);
\node[left] at  (1.5,.5) {$s$};
\node[above] at  (1.85,.45) {$\beta$};
\end{scope}

\begin{scope}[shift={(5,0)}, scale=1.5]

\draw[diagramred,thick,dashed](2.2,.75)--(2.5,.75);
\draw[diagramred,thick,dashed](2.5,.25)--(2.3,.25);
\draw[diagramred, thick, dashed] (2.2,.25)--(1.5,.25)--(1.5,0);
\draw[diagramred, thick] (1.5,0)--(1.5,-.5);
\draw[opacity=.5, fill=gray, draw=none] (0,0)--(2,0)--(3,1)--(1,1)--(0,0);
\draw[diagramred, thick] (1.5,1.5)--(1.5,.75)--(2.2,.75);% arc (90:-90:.25)--(2.3,.25);
\draw[opacity=.3, fill=diagramred, thick, draw=none] (2,1)--(3,2)--(3,0)--(2,-1)--(2,1);

\draw[diagramred,thick](2.4,.75)--(2.5,.75) arc (90:-90:.25)--(2.4,.25);
\draw[fill=black] (2.4,.25) circle (.05);
\draw[fill=black] (2.4,.75) circle (.05);

\draw[black, thick] (2,0)--(2.6,.6);(3,1);
\draw[black, thick] (2.7,.7)--(3,1);(3,1);

\end{scope}
\end{tikzpicture}
\caption{Left: A point $s$ in the intersection of $\Delta_p$ (gray) with $D_2$ (red) together with an embedded arc $\beta$ in $\Delta_p$ from $s$ to $\alpha_2\subseteq \Delta_2$.  Right: A homotopy of $D_2$ removes the point of intersection at $\phi(s)$ and introduces two new points of self intersection for $D_2$.}
\label{fig:FingerMove}
\end{figure}

Now perform the relative Whitney trick using $\Delta_p$ to modify $D_1$. As we have $D_2\cap \Delta_p=\emptyset$, the relative Whitney trick reduces $|D_1\cap D_2|$ by one. By iterating the procedure above, we achieve that $|D_1\cap D_2|=0$.  Since each application of the relative Whitney trick reduces $|D_1\cap D_2|$ by one and changes one crossing between $L_1$ and $L_2$, the last claim of the lemma follows.

For the statement regarding smoothly immersed disks, if $W$ is smooth, then we can arrange that $D_1, D_2$ as well as each $\Delta_p$ are smoothly immersed.  As a consequence of Remark~\ref{rmk: affect of RWT} \pref{item: Rel Whit Trick Smooth}, this would result in the disjoint disks produced being smoothly immersed.
\end{proof}

When one tries to extend Lemma \ref{lem:warmup} to links with more components, a new complication arises. Assume the hypotheses of the lemma, but now assume that $D_1,\dots, D_n$ is a collection of $n$ transverse immersed disks. Consider two of these immersed disks $D_i$ and $D_j$. We wish to separate them by removing intersection points with the relative Whitney trick. Let $p\in D_i\cap D_j$ and $\Delta_p$ be a relative Whitney disk for $p$. If $\Delta_p$ intersects a disk $D_k$ for some $k\notin \{i,j\}$, then performing the relative Whitney trick using $\Delta_p$ 
to modify $D_i$ will produce new points in $D_i\cap D_k$. Hence, we would need a way to arrange that the relative Whitney disk $\Delta_p$ is disjoint from $D_k$ for all $k\notin \{i,j\}$. Also, note that if $k\in \{i,j\}$, as in the proof of Lemma~\ref{lem:warmup}, we may perform a finger move on $D_k$ to remove the intersection points in $\Delta_p \cap D_k$  for the price of introducing more self intersections.

Suppose $k\notin \{i,j\}$ and let $r\in \Delta_p\cap D_k$. Our strategy to remove $r$ is to modify the disk $\Delta_p$ itself by a relative Whitney trick using some relative Whitney disk~$\Delta_r$ for $r$. Of course, the use of $\Delta_r$ may add new intersection points to $\Delta_p$, according to what intersects $\Delta_r$, so we pause to consider this. Intersection points in $\Delta_r\cap D_i$ and $\Delta_r\cap D_j$ are not a problem because performing the relative Whitney trick using $\Delta_r$ to modify $\Delta_p$ only adds points to $\Delta_p\cap D_i$ and $\Delta_p\cap D_j$, and these points can be removed by performing finger moves as in the proof of Lemma~\ref{lem:warmup}.
In fact, an intersection point in $\Delta_r\cap D_k$ is also not a problem, as we can now perform a finger move on $D_k$ to remove it, at the cost of adding two self-intersections to $D_k$. So intersections between $\Delta_r$ and each of $D_i$, $D_j$ and $D_k$ are all unproblematic for us, whereas on $\Delta_p$ we could only deal with the first two types. The above argument allows us to extend Lemma~\ref{lem:warmup} to 3-component links. For links with more than 3-components there may be further intersection types in $\Delta_r$ that we cannot yet deal with. This suggests an induction: adding relative Whitney disks at each stage, and increasing the number of intersection types we know how to remove with finger moves. Eventually we know how to deal with all intersection types. Then a series of finger moves and relative Whitney tricks will remove the intersection point $r$, with the only cost a possible increase in self-intersection for the disks. We now give more detailed proof. Note that the following statement is more general than Proposition~\ref{prop: disk sep in intro}.

\begin{proposition}\label{prop:get immersed disks}Let $W$ be a 4-manifold, $Y$ be a connected 3-manifold in $\bdry W$, and $L=L_1\cup\dots \cup L_n$ be a link in $Y$ whose components are nullhomotopic in $W$. If the inclusion induced map $\pi_1(Y)\to\pi_1(W)$ is surjective, then there exists a link $J=J_1\cup\dots \cup J_n$ in $Y$ which is freely homotopic to $L$ and bounds disjoint immersed disks in $W$. If $W$ is smooth, then these disks may be smoothly immersed.

Moreover, if {$D_1, \dots , D_n$} are transverse immersed disks in $W$ bounded by $L$, then we may choose a homotopy from $L$ to $J$ which restricts to an isotopy on each component of $L$ and for any $i\neq j$ changes a crossing between the $i^{\text{th}}$ and $j^{\text{th}}$ components exactly once for each intersection point in $D_i\cap D_j$.\end{proposition}

\begin{proof}
Let {$D_1, \dots , D_n$} be transverse immersed disks bounded by $L$. Consider two of these immersed disks $D_i$ and $D_j$ and let $p\in D_i\cap D_j$.  We will modify $D_1 \cup \dots \cup D_n$  by a homotopy which is constant on the boundary and which does not change $|D_k\cap D_\ell|$ for any $k\neq \ell$.  Afterwards we will produce a relative Whitney disk $\Delta_p$ associated with $p$ so that $\Delta_p\cap D_k=\emptyset$ for all $k\neq i$. Once we have accomplished this,  the relative Whitney trick using $\Delta_p$ to modify $D_i$ preserves $|D_k\cap D_\ell|$ for all $\{k,\ell\}\neq \{i,j\}$ and will reduce this number by one if $\{k,\ell\}=\{i,j\}$. As in Lemma~\ref{lem:warmup}, it also affects $L$ by a homotopy obtained by changing a crossing between the $i^{\text{th}}$ and the $j^{\text{th}}$ component.  

We will call the point $p\in D_{i}\cap D_{j}$ an \emph{order zero} intersection point.  Choose a relative Whitney disk $\Delta_p$ associated with $p$ and call it an \emph{order one} relative Whitney disk. As in the proof of Lemma~\ref{lem:warmup}, a relative Whitney disk exists since  the induced map $\pi_1(Y)\to\pi_1(W)$ is surjective. In fact, for any double point, we can find an associated relative Whitney disk.  We define the set of \emph{acceptable numbers} for $\Delta_p$ to be $A_p = \{i,j\}$.  An intersection point $r\in \Delta_p\cap D_k$ is called \emph{unacceptable} for $\Delta_p$ if $k\not\in A_p$. We now make an inductive definition. Let $m\in \N$, suppose that $\Delta_q$ is an order $m$ relative Whitney disk with set of acceptable numbers $A_q\subseteq\{1,\dots,n\}$ and that $r\in \Delta_{q}\cap D_{k}$ is some unacceptable intersection point for $\Delta_q$.  We will call the point  $r$ an \emph{order $m$} intersection point.  An associated relative Whitney disk $\Delta_r$ is called an \emph{order $m+1$} relative Whitney disk.  The set of \emph{acceptable numbers} for $\Delta_r$ is defined to be $A_r := A_q\cup \{k\}$.  It follows from induction that if $\Delta_q$ is an order $m$ relative Whitney disk, then $|A_q| = m+1$.  In particular, if $\Delta_q$ is an order $n-1$ relative Whitney disk, then $A_q=\{1,\dots, n\}$ and every intersection point $r\in \Delta_{q}\cap D_{k}$ is acceptable.

Let $\mathcal{D}_1=\{\Delta_p\}$. We make an inductive construction. Let $m\in \N$ and suppose $\mathcal{D}_{m}$ is a set of order $m$ relative Whitney disks. For each $\Delta_q\in\mathcal{D}_{m}$ and for each unacceptable order $m$ intersection $r\in \Delta_q$, choose an order $m+1$ relative Whitney disk $\Delta_r$. Write $\mathcal{D}_{m+1}$ for the set consisting of a single choice of $\Delta_r$ for each such $q$ and $r$. Write $\mathcal{D}:=\mathcal{D}_0\cup\mathcal{D}_1\cup\dots\cup\mathcal{D}_{n-1}$. We now organize $\D$ into a tree with root $\Delta_p$ by declaring that for any $\Delta_q\in \mathcal{D}_m$ and any unacceptable intersection point $r\in \Delta_q\cap D_k$ the relative Whitney disk $\Delta_r$ is a descendent of $\Delta_q$. Notice that a vertex $\Delta_q$ on this graph is a leaf if and only if it has no unacceptable intersections.  Consequentially, any order $n-1$ relative Whitney disk in $\D$ is a leaf.

Suppose that  $\Delta_r\in\mathcal{D}_m$ is a leaf of order $m>1$. We will homotope $D_1\cup\dots \cup D_n$ while preserving $D_k\cap D_\ell$ for all $k\neq \ell$, preserving $\D$, and without introducing any new unacceptable intersection points.  The end result will have that the interior of $\Delta_r$ is disjoint from $D_1\cup\dots \cup D_n$. 

 Proceeding, since $\Delta_r$ is order $m>1$, it is a descendent of some $\Delta_q$.  Thus for some $k$, we have that $r\in \Delta_q \cap D_k$ where $r$ is an unacceptable intersection point of order $m-1$, and $A_r = A_q\cup \{k\}$. Suppose $\Delta_{r} \cap D_{\ell}$ is nonempty for some $\ell$. If $\ell = k$, then for each point perform a finger move on $D_{k}$, along an embedded arc in $\Delta_r$, to remove this intersection point, with the cost of producing two new self-intersections in $D_{k}$. If $\ell \neq k$, then we perform a finger move on $D_\ell$, along an embedded arc in $\Delta_r$, to remove this point of intersection, with the cost of producing two new points in $\Delta_{q} \cap D_{\ell}$.  Since~$\Delta_r$ is a leaf, $\ell$ must be in $A_r = A_q\cup\{k\}$.  As $\ell\neq k$, it must be that $\ell\in A_q$, and so these two new points of intersection are acceptible.  We have now arranged that the interior of $\Delta_r$ is disjoint from $D_1\cup\dots \cup D_n$.  Importantly, we have created no new unacceptable intersections.

We now modify $\Delta_q$ by the relative Whitney trick using $\Delta_r$.  This reduces the number of unacceptable intersections in~$\Delta_q$ by one and eliminates the leaf at $\Delta_r$. The possible cost of this procedure is to produce new self-intersections in $D_k$ and new acceptable intersections in $\Delta_q$.  The affect on the boundary of this move is to change the relative Whitney arc associated with $\Delta_q$ by a homotopy passing it through a component of $L$.  In particular, the original link $L$ is preserved.  

Iterate the modification of the previous three {paragraphs} until the tree $\D$ has no leaves of order $m>1$.  This means that $\D=\{\Delta_p\}$ and so $\Delta_p$ is a leaf.  Thus $\Delta_p\cap D_k=\emptyset$ for all $k\notin \{i,j\}$.  By performing finger moves as in Lemma~\ref{lem:warmup}, we arrange that $\Delta_p$ is disjoint from $D_j$ at a cost of adding self intersections to $D_j$.  Modifying $D_i$ by the relative Whitney trick using $\Delta_p$ reduces~$|D_i\cap D_j|$ by exactly one at the possible expense of increasing the self-intersections of $D_i$.  On the boundary this relative Whitney move affects a single crossing change between $L_i$ and $L_j$ as claimed.  

{Again, the same argument from the proof of Lemma~\ref{lem:warmup} gives the statement regarding smoothly immersed disks.}
\end{proof}

\begin{remark}
The idea of finding increasingly high order relative Whitney disks used in the proof of Proposition~\ref{prop:get immersed disks} is reminiscent of the concept of a Whitney tower; see for example \cite{CST2012, MR3770161}.  It motivates and is an example of what we call a relative Whitney tower, a concept whose formal definition appears in Section~\ref{sect: towers}.  
\end{remark}

\section{Every link in a homology sphere is freely homotopic to a slice link}\label{sec:homotopictoslice}

Cha-Kim-Powell \cite{Cha-Kim-Powell:2020-1} have obtained conditions that ensure a link in $S^3$ is freely slice. We now describe a generalization of these conditions for links in a general homology sphere. Links satisfying the generalized conditions are also freely slice, as we confirm in Appendix \ref{appendix}. We then describe how the disk separation results from Section \ref{sec:separating} are used to modify any link in a homology sphere by a homotopy to satisfy the generalized Cha-Kim-Powell conditions. This will prove Theorem~\ref{thm:main}, the main theorem of the paper.  We first give some definitions. Recall that a link $L$ in a homology sphere $Y$ is a \emph{boundary link} if it bounds pairwise disjoint Seifert surfaces in $Y$ and a collection of such surfaces is called a \emph{boundary link Seifert surface} for $L$. Lastly, a link $L$ is \emph{freely slice} if it bounds a collection of disjoint locally flat disks $D$ in the contractible topological 4-manifold $X$ bounded by $Y$ such that $ \pi_1(X\smallsetminus D)$ is a free group generated by the meridians of $L$.

\begin{definition}\label{def:goodbasis} Given a boundary link $L$ together with a boundary link Seifert surface $F$, a \emph{good basis} is a collection $\{a_i,b_i\}_{1\leq i \leq g}$ of simple closed curves on $F$, whose geometric intersections are symplectic, that represent a basis for $H_1(F;\Z)$, and such that the Seifert matrix of $F$ with respect to this basis is reducible by a sequence of elementary $S$-reductions to the \emph{null matrix}

\setlength\dashlinedash{0.2pt}
\setlength\dashlinegap{1.5pt}
\setlength\arrayrulewidth{0.3pt}
\[
\left[
\begin{array}{c:c:c:c:c}
\begin{array}{cc}0&\epsilon_1
\\1-\epsilon_1&0\end{array} &
\begin{array}{cc}0&*
\\0&*\end{array} &
\begin{array}{cc}0&*
\\0&*\end{array} 
&\dots&
\begin{array}{cc}0&*
\\0&*\end{array} 
\\
\hdashline
\begin{array}{cc}0&0
\\**&*\end{array} &
\begin{array}{cc}0&\epsilon_2
\\1-\epsilon_2&0\end{array} &
\begin{array}{cc}0&*
\\0&*\end{array}
&\dots&
\begin{array}{cc}0&*
\\0&*\end{array} 
 \\
\hdashline
\begin{array}{cc}0&0
\\**&*\end{array} &
\begin{array}{cc}0&0
\\**&*\end{array} &
\begin{array}{cc}0&\epsilon_3
\\1-\epsilon_3&0\end{array}
&\dots&
\begin{array}{cc}0&*
\\0&*\end{array} 
 \\
\hdashline
\vdots&\vdots&\vdots&\ddots&\vdots\\
\hdashline
\begin{array}{cc}0&0
\\**&*\end{array} &
\begin{array}{cc}0&0
\\**&*\end{array} &
\begin{array}{cc}0&0
\\**&*\end{array} 
&\dots&
\begin{array}{cc}0&\epsilon_g
\\1-\epsilon_g&0
\end{array}
\end{array}
\right]
\]
where $\epsilon_i\in\{0,1\}$ for each $i$ and each $\ast$ represents some integer.
\end{definition}

\begin{definition}\label{def:disky}
Let $L$ be a boundary link in a homology sphere $Y$, $F$ be a boundary link Seifert surface for $L$, and  $\{a_i,b_i\}_{1\leq i \leq g}$ be a good basis for $L$ on $F$. For each $i$, let $b_i'$ be the result of pushing $b_i$ off of $F$ such that it has zero linking with $a_i$, and $(b_i')^+$ be a zero linking parallel copy of $b_i'$. 

We say that $\{a_i,b_i\}_{1\leq i \leq g}$ is a \emph{good disky basis} for $L$ if there exist immersed disks $$\left\{\Delta_j^+,\Delta_i \mid 1\leq j \leq 2g, 1\leq i \leq g\right\}$$ in the contractible 4-manifold bounded by $Y$, such that $\bdry \Delta_{i}^+ = a_i$, $\bdry \Delta_{g+i}^{+} = (b_i')^+$, and $\bdry \Delta_i = b_i'$ for each $i$ and all disks are pairwise disjoint except possibly for intersections among $\{\Delta_j^+\}_{1\leq j \leq 2g}$.
\end{definition}

\begin{remark}
Cha-Kim-Powell use the condition of `homotopy trivial$^+$' in addition to being a good basis. As in the proof of \cite[Proposition 4.3]{Cha-Kim-Powell:2020-1} it follows from 
{\cite[Remark 3.2 (1) and Lemma 3.3]{Cha-Kim-Powell:2020-1}}
 that a homotopy trivial$^+$ good basis for a boundary link in $S^3$ satisfies the conditions of Definition~\ref{def:disky}.  
\end{remark}

We can now state a version of the Cha-Kim-Powell theorem for links in general homology spheres.

\begin{theorem}\label{thm:CKP} A boundary link in a homology sphere with a good disky basis is freely slice.
\end{theorem}

The proof of Theorem~\ref{thm:CKP} is essentially the same as Cha-Kim-Powell's proof of \cite[Theorem A]{Cha-Kim-Powell:2020-1}. A sketch of this argument, together with the minor adjustments required to confirm their argument transfers over to general homology spheres, are found in Appendix~\ref{appendix}. The current section will proceed assuming that Theorem~\ref{thm:CKP} is proved.

Next we show how a link in a 3-manifold can be modified by a homotopy to a boundary link with a good disky basis.   

\begin{proposition}\label{prop:get immersed disks +}
Let $W$ be a simply connected 4-manifold, $Y$ be a connected 3-manifold in $\bdry W$, and $L=L_1\cup\dots\cup L_n$ be a link in $Y$ whose components are nullhomologous in $Y$.  Then $L$ is freely homotopic to a link $J = J_1\cup\dots \cup J_n$ so that there exists a collection of disjoint immersed disks $\{D_i, D_i^+\}_{1\leq i \leq n}$ in $W$ such that $\bdry D_i = J_i$ and $\bdry  {D^+_i} = (J_i)^+$, where $(J_i)^+$ is a zero linking parallel copy of $J_i$.  \end{proposition}
\begin{proof}
Pick $n$ distinct points $p_1,\dots, p_n$ in $Y$.  As each $L_i$ is nullhomologous in $Y$, we see that $L_i$ is freely homotopic to a product of commutators $ \prod_{j=1}^{g_i} [\alpha_{i,j}, \beta_{i,j}]$ for some $\alpha_{i,j}, \beta_{i,j}\in \pi_1(Y,p_i).$ By Proposition~\ref{prop:get immersed disks}, there is a link 
$$\mathcal{L}:=\Cup_{i=1}^n\Cup_{j=1}^{g_i} a_{i,j}\cup b_{i,j}$$
such that for each $i,j$, the components $a_{i,j}$ and $b_{i,j}$ are freely homotopic to $\alpha_{i,j}$ and $\beta_{i,j}$ respectively, and there are disjoint immersed disks 
$$\left\{\Delta_{a_{i,j}}, \Delta_{b_{i,j}} \mid 1\leq i \leq n, 1\leq j \leq g_i\right\}$$ in $W$ bounded by $\mathcal{L}$.  Thus, for each $i,j$, there exist embedded arcs $c_{i,j}$ and $d_{i,j}$ from $p_i$ to a point on $a_{i,j}$ and $b_{i,j}$ respectively, so that, as elements in $\pi_1(Y,p_i)$
 $$\alpha_{i,j} = c_{i,j}*a_{i,j}*c_{i,j}^{-1} \quad\text{and}\quad \beta_{i,j} = d_{i,j}*b_{i,j}*d_{i,j}^{-1}.$$  As in Figure~\ref{fig:commutator to surface}, we construct a collection of disjoint surfaces $F_1,\dots, F_n$ in $Y$ so that for each $i$,

\begin{itemize}[leftmargin=0.9cm]\setlength\itemsep{0em}
\item $F_i$ is a Seifert surface for some knot, denoted by $J_i$, which is freely homotopic to $L_i$,
\item $F_i$ has a symplectic basis $\{A_{i,j}, B_{i,j}\}_{1\leq j \leq g_i}$ so that 
as curves in $Y$, $A_{i,j} = a_{i,j}$ and $B_{i,j}' = b_{i,j}$ for each $j$. Here, the curve $B_{i,j}'$ is the positive pushoff of $B_{i,j}$ with respect to $F_i$.
\item The framing of $a_{i,j} = A_{i,j}$ induced by $F_i$ extends over $\Delta_{a_{i,j}}$   
and the framing of $b_{i,j} = B_{i,j}'$ induced by $F_i^+$, the result of pushing $F_i$ off itself, extends over $\Delta_{b_{i,j}}$.  (This can be arranged by adding twists to the bands of $F_i$.)
\end{itemize}
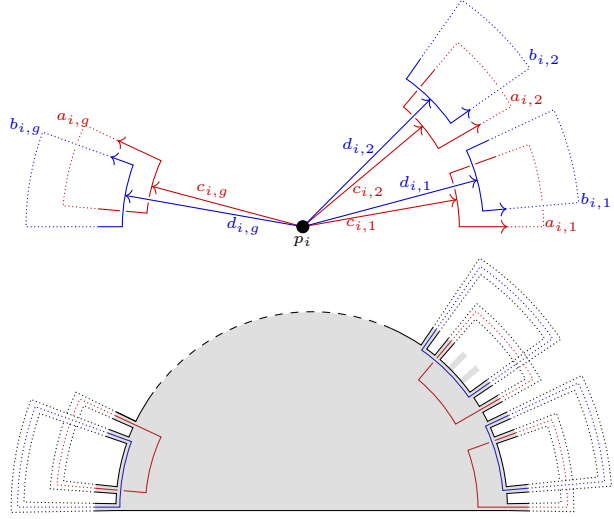
\begin{figure}
\begin{tikzpicture}[scale=1.6]
    \draw[diagramred,<-] (1.7,0)--(1.3,0) arc (0:20:1.3)--({1.7*cos(20)},{1.7*sin(20)});
    \draw[diagramred, densely dotted] ({1.7*cos(20)},{1.7*sin(20)})--({2*cos(20)},{2*sin(20)}) arc (20:6:2);
    \draw[diagramred, densely dotted] ({2*cos(4)},{2*sin(4)}) arc (4:0:2)--(1.7,0);
    \draw[draw=white, fill=white] ({1.5*cos(20)},{1.5*sin(20)}) circle (.02);
    \draw[draw=white, fill=white] ({1.3*cos(15)},{1.3*sin(15)}) circle (.02);
    \draw[draw=white, fill=white] ({2*cos(5)},{2*sin(5)}) circle (.02);
    \draw[diagramblue,<-] ({1.7*cos(5)},{1.7*sin(5)})--({1.5*cos(5)},{1.5*sin(5)}) arc (5:25:1.5)--({1.7*cos(25)},{1.7*sin(25)});
    \draw[diagramblue, densely dotted] ({1.7*cos(25)},{1.7*sin(25)})--({2.3*cos(25)},{2.3*sin(25)}) arc (25:5:2.3)--({1.7*cos(5)},{1.7*sin(5)});    
    \draw[diagramred, ->](0,0)--({1.3*cos(10)},{1.3*sin(10)});
    \draw[diagramblue, ->](0,0)--({1.5*cos(15)},{1.5*sin(15)});

%alpha_2, \beta_2
    \draw[diagramred,<-] ({1.7*cos(30)},{1.7*sin(30)})--({1.3*cos(30)},{1.3*sin(30)}) arc (30:50:1.3)--({1.7*cos(50)},{1.7*sin(50)});
    \draw[diagramred, densely dotted] ({1.7*cos(20+30)},{1.7*sin(20+30)})--({2*cos(20+30)},{2*sin(20+30)}) arc (20+30:6+30:2);
    \draw[diagramred, densely dotted] ({2*cos(4+30)},{2*sin(4+30)}) arc (4+30:0+30:2)--({1.7*cos(30)},{1.7*sin(30)});
    \draw[draw=white, fill=white] ({1.5*cos(50)},{1.5*sin(50)}) circle (.02);
    \draw[draw=white, fill=white] ({1.3*cos(45)},{1.3*sin(45)}) circle (.02);
    \draw[draw=white, fill=white] ({2*cos(30+5)},{2*sin(30+5)}) circle (.02);
    \draw[diagramblue,<-] ({1.7*cos(35)},{1.7*sin(35)})--({1.5*cos(35)},{1.5*sin(35)}) arc (35:55:1.5)--({1.7*cos(55)},{1.7*sin(55)});
    \draw[diagramblue, densely dotted] ({1.7*cos(55)},{1.7*sin(55)})--({2.3*cos(55)},{2.3*sin(55)}) arc (55:35:2.3)--({1.7*cos(35)},{1.7*sin(35)});    
    \draw[diagramred, ->](0,0)--({1.3*cos(40)},{1.3*sin(40)});
    \draw[diagramblue, ->](0,0)--({1.5*cos(45)},{1.5*sin(45)});

%alpha_g, \beta_g\

    \draw[diagramred,<-] ({1.7*cos(155)},{1.7*sin(155)})--({1.3*cos(155)},{1.3*sin(155)}) arc (155:175:1.3)--({1.7*cos(175)},{1.7*sin(175)});
    \draw[diagramred, densely dotted] ({1.7*cos(20+155)},{1.7*sin(20+155)})--({2*cos(20+155)},{2*sin(20+155)}) arc (20+155:6+155:2);
    \draw[diagramred, densely dotted] ({2*cos(4+155)},{2*sin(4+155)}) arc (4+155:0+155:2)--({1.7*cos(155)},{1.7*sin(155)});
    \draw[draw=white, fill=white] ({1.5*cos(20+155)},{1.5*sin(20+155)}) circle (.02);
    \draw[draw=white, fill=white] ({1.3*cos(45+125)},{1.3*sin(45+125)}) circle (.02);
    \draw[draw=white, fill=white] ({2*cos(160)},{2*sin(160)}) circle (.02);
    \draw[diagramblue,<-] ({1.7*cos(160)},{1.7*sin(160)})--({1.5*cos(160)},{1.5*sin(160)}) arc (160:180:1.5)--({1.7*cos(180)},{1.7*sin(180)});
    \draw[diagramblue, densely dotted] ({1.7*cos(180)},{1.7*sin(180)})--({2.3*cos(180)},{2.3*sin(180)}) arc (180:160:2.3)--({1.7*cos(160)},{1.7*sin(160)});    
    \draw[diagramred, ->](0,0)--({1.3*cos(165)},{1.3*sin(165)});
    \draw[diagramblue, ->](0,0)--({1.5*cos(170)},{1.5*sin(170)});

\draw [fill=black] (0,0) circle (.05);
\node[below] at (0,0) {\tiny{$ p_i$}};
\node[diagramred] at ({0.5*cos(10)},{0.5*sin(10)-.08}) {\tiny{$c_{i,1}$}};
\node[diagramblue] at ({1*cos(15)-.1*sin(15)},{1*sin(15)+.1*cos(15}) {\tiny{$d_{i,1}$}};
\node[diagramred] at ({0.6*cos(40)+.13*sin(40)},{0.6*sin(40)-.13*cos(40)}) {\tiny{$c_{i,2}$}};
\node[diagramblue] at ({0.8*cos(45)-.14*sin(45)},{0.8*sin(45)+.14*cos(45)}) {\tiny{$d_{i,2}$}};
\node[diagramblue] at ({-0.5*cos(10)},{0.5*sin(10)-.08}) {\tiny{$d_{i,g}$}};
\node[diagramred] at ({-0.8*cos(15)+.09*sin(15)},{0.8*sin(15)+.09}) {\tiny{$c_{i,g}$}};

\node[diagramred] at ({2.15},{0}) {\tiny{$a_{i,1}$}};
\node[diagramblue] at ({2.45*cos(5)},{2.45*sin(5)}) {\tiny{$b_{i,1}$}};

\node[diagramred] at ({2.15*cos(30)},{2.1*sin(30)}) {\tiny{$a_{i,2}$}};
\node[diagramblue] at ({2.45*cos(35)},{2.45*sin(35)}) {\tiny{$b_{i,2}$}};

\node[diagramred] at ({2.1*cos(155)},{2.1*sin(155)}) {\tiny{$a_{i,g}$}};
\node[diagramblue] at ({2.45*cos(160)},{2.45*sin(160)}) {\tiny{$b_{i,g}$}};
\end{tikzpicture}
\begin{tikzpicture}[scale=1.7]
%alpha_1, \beta_1
    \draw[diagramred] (1.7,0)--(1.3,0) arc (0:20:1.3)--({1.7*cos(20)},{1.7*sin(20)});
    \draw[diagramred, densely dotted] ({1.7*cos(20)},{1.7*sin(20)})--({2*cos(20)},{2*sin(20)}) arc (20:6:2);
    \draw[diagramred, densely dotted] ({2*cos(4)},{2*sin(4)}) arc (4:0:2)--(1.7,0);
    \draw[draw=white, fill=white] ({1.5*cos(20)},{1.5*sin(20)}) circle (.02);
%    \draw[draw=white, fill=white] ({1.3*cos(15)},{1.3*sin(15)}) circle (.02);
%    \draw[draw=white, fill=white] ({2*cos(5)},{2*sin(5)}) circle (.02);
    \draw[diagramblue] ({1.7*cos(5)},{1.7*sin(5)})--({1.5*cos(5)},{1.5*sin(5)}) arc (5:25:1.5)--({1.7*cos(25)},{1.7*sin(25)});
    \draw[diagramblue, densely dotted] ({1.7*cos(25)},{1.7*sin(25)})--({2.3*cos(25)},{2.3*sin(25)}) arc (25:5:2.3)--({1.7*cos(5)},{1.7*sin(5)});    
%    \draw[diagramred](0,0)--({1.3*cos(10)},{1.3*sin(10)});
%    \draw[diagramblue](0,0)--({1.5*cos(15)},{1.5*sin(15)});

%alpha_2, \beta_2
    \draw[diagramred] ({1.7*cos(30)},{1.7*sin(30)})--({1.3*cos(30)},{1.3*sin(30)}) arc (30:50:1.3)--({1.7*cos(50)},{1.7*sin(50)});
    \draw[diagramred, densely dotted] ({1.7*cos(20+30)},{1.7*sin(20+30)})--({2*cos(20+30)},{2*sin(20+30)}) arc (20+30:6+30:2);
    \draw[diagramred, densely dotted] ({2*cos(4+30)},{2*sin(4+30)}) arc (4+30:0+30:2)--({1.7*cos(30)},{1.7*sin(30)});
    \draw[draw=white, fill=white] ({1.5*cos(50)},{1.5*sin(50)}) circle (.02);
%    \draw[draw=white, fill=white] ({1.3*cos(45)},{1.3*sin(45)}) circle (.02);
%    \draw[draw=white, fill=white] ({2*cos(30+5)},{2*sin(30+5)}) circle (.02);
    \draw[diagramblue] ({1.7*cos(35)},{1.7*sin(35)})--({1.5*cos(35)},{1.5*sin(35)}) arc (35:55:1.5)--({1.7*cos(55)},{1.7*sin(55)});
    \draw[diagramblue, densely dotted] ({1.7*cos(55)},{1.7*sin(55)})--({2.3*cos(55)},{2.3*sin(55)}) arc (55:35:2.3)--({1.7*cos(35)},{1.7*sin(35)});    
%    \draw[diagramred](0,0)--({1.3*cos(40)},{1.3*sin(40)});
%    \draw[diagramblue](0,0)--({1.5*cos(45)},{1.5*sin(45)});

%alpha_g, \beta_g\ 
\draw[diagramred] ({1.7*cos(155)},{1.7*sin(155)})--({1.3*cos(155)},{1.3*sin(155)}) arc (155:175:1.3)--({1.7*cos(175)},{1.7*sin(175)});
    \draw[diagramred, densely dotted] ({1.7*cos(20+155)},{1.7*sin(20+155)})--({2*cos(20+155)},{2*sin(20+155)}) arc (20+155:6+155:2);
    \draw[diagramred, densely dotted] ({2*cos(4+155)},{2*sin(4+155)}) arc (4+155:0+155:2)--({1.7*cos(155)},{1.7*sin(155)});
    \draw[draw=white, fill=white] ({1.5*cos(20+155)},{1.5*sin(20+155)}) circle (.02);
%    \draw[draw=white, fill=white] ({1.3*cos(45+125)},{1.3*sin(45+125)}) circle (.02);
%    \draw[draw=white, fill=white] ({2*cos(160)},{2*sin(160)}) circle (.02);
    \draw[diagramblue] ({1.7*cos(160)},{1.7*sin(160)})--({1.5*cos(160)},{1.5*sin(160)}) arc (160:180:1.5)--({1.7*cos(180)},{1.7*sin(180)});
    \draw[diagramblue, densely dotted] ({1.7*cos(180)},{1.7*sin(180)})--({2.3*cos(180)},{2.3*sin(180)}) arc (180:160:2.3)--({1.7*cos(160)},{1.7*sin(160)});

    \draw[white, fill=gray, opacity=.25]  
    ({1.7*cos(181)},{1.7*sin(181)})--({1.53*cos(181)},{1.53*sin(181)})--
    ({1.53*cos(-1)},{1.53*sin(-1)})--({1.7*cos(-1)},{1.7*sin(-1)})--({1.7*cos(1)},{1.7*sin(1)})--({1.53*cos(1)},{1.53*sin(1)}) arc(1:4:1.53)
    --({1.7*cos(4)},{1.7*sin(4)})--({1.7*cos(6)},{1.7*sin(6)})--({1.53*cos(6)},{1.53*sin(6)}) arc (6:19:1.53)
    --({1.7*cos(19)},{1.7*sin(19)})--({1.7*cos(21)},{1.7*sin(21)})--({1.53*cos(21)},{1.53*sin(21)}) arc (21:24:1.53)
    --({1.7*cos(24)},{1.7*sin(24)})--({1.7*cos(26)},{1.7*sin(26)})--({1.53*cos(26)},{1.53*sin(26)}) arc (26:29:1.53)
    --({1.7*cos(29)},{1.7*sin(29)})--({1.7*cos(31)},{1.7*sin(31)})--({1.53*cos(31)},{1.53*sin(31)})arc (31:34:1.53)
    --({1.7*cos(34)},{1.7*sin(34)})--({1.7*cos(36)},{1.7*sin(36)})--({1.53*cos(36)},{1.53*sin(36)})arc (36:39:1.53)
    --({1.7*cos(39)},{1.7*sin(39)})--({1.7*cos(41)},{1.7*sin(41)})--({1.53*cos(41)},{1.53*sin(41)})arc (41:44:1.53)
    --({1.7*cos(44)},{1.7*sin(44)})--({1.7*cos(46)},{1.7*sin(46)})--({1.53*cos(46)},{1.53*sin(46)})arc (46:49:1.53)
    --({1.7*cos(49)},{1.7*sin(49)})--({1.7*cos(51)},{1.7*sin(51)})--({1.53*cos(51)},{1.53*sin(51)})arc (51:54:1.53)
    --({1.7*cos(54)},{1.7*sin(54)})--({1.7*cos(56)},{1.7*sin(56)})--({1.53*cos(56)},{1.53*sin(56)})arc (56:154:1.53)
    --({1.7*cos(154)},{1.7*sin(154)})--({1.7*cos(156)},{1.7*sin(156)})--({1.53*cos(156)},{1.53*sin(156)})arc (156:159:1.53)
    --({1.7*cos(159)},{1.7*sin(159)})--({1.7*cos(161)},{1.7*sin(161)})--({1.53*cos(161)},{1.53*sin(161)})arc (161:174:1.53)
    --({1.7*cos(174)},{1.7*sin(174)})--({1.7*cos(176)},{1.7*sin(176)})--({1.53*cos(176)},{1.53*sin(176)})arc (176:179:1.53)
    --({1.7*cos(179)},{1.7*sin(179)})--({1.7*cos(181)},{1.7*sin(181)})--({1.53*cos(181)},{1.53*sin(181)});

%Draw L near alpha_1, beta_1
    \draw[] ({1.7*cos(-1)},{1.7*sin(-1)}) -- ({1.53*cos(-1)},{1.53*sin(-1)});
    \draw[densely dotted] ({1.7*cos(-1)},{1.7*sin(-1)})  -- ({2.05*cos(-1)},{2.05*sin(-1)}) arc (-1:4:2.05);
\draw[densely dotted] ({2.05*cos(6)},{2.05*sin(6)}) arc (6:21:2.05)--({1.7*cos(21)},{1.7*sin(21)});
    \draw[] ({1.7*cos(21)},{1.7*sin(21)}) -- ({1.53*cos(21)},{1.53*sin(21)}) arc (21:24:1.53) -- ({1.7*cos(24)},{1.7*sin(24)});
    \draw[densely dotted] ({1.7*cos(24)},{1.7*sin(24)})  -- ({2.25*cos(24)},{2.25*sin(24)}) arc (24:6:2.25)--({1.7*cos(6)},{1.7*sin(6)});
    \draw[] ({1.7*cos(19)},{1.7*sin(19)}) -- ({1.53*cos(19)},{1.53*sin(19)}) arc(19:6:1.53) -- ({1.7*cos(6)},{1.7*sin(6)});
    \draw[densely dotted] ({1.7*cos(19)},{1.7*sin(19)})  -- ({1.95*cos(19)},{1.95*sin(19)}) arc (19:6:1.95);
    \draw[densely dotted]  ({1.95*cos(4)},{1.95*sin(4)}) arc (4:1:1.95)--({1.7*cos(1)},{1.7*sin(1)});
    \draw[] ({1.7*cos(4)},{1.7*sin(4)}) -- ({1.53*cos(4)},{1.53*sin(4)}) arc(4:1:1.53) -- ({1.7*cos(1)},{1.7*sin(1)});
    \draw[densely dotted] ({1.7*cos(4)},{1.7*sin(4)})--({2.35*cos(4)},{2.35*sin(4)}) arc (4:26:2.35)--({1.7*cos(26)},{1.7*sin(26)});

    \draw[] ({1.7*cos(26)},{1.7*sin(26)}) -- ({1.53*cos(26)},{1.53*sin(26)}) arc (26:29:1.53) --({1.7*cos(30-1)},{1.7*sin(30-1)});
    
%Draw L near alpha_2, beta_2
    \draw[densely dotted] ({1.7*cos(30-1)},{1.7*sin(30-1)})  -- ({2.05*cos(30-1)},{2.05*sin(30-1)}) arc (30-1:30+21:2.05)--({2.05*cos(30+21)},{2.05*sin(30+21)})--({1.7*cos(30+21)},{1.7*sin(30+21)});
    \draw[] ({1.7*cos(30+21)},{1.7*sin(30+21)}) -- ({1.53*cos(30+21)},{1.53*sin(30+21)}) arc (30+21:30+24:1.53) -- ({1.7*cos(30+24)},{1.7*sin(30+24)});
\draw[densely dotted] ({1.7*cos(30+24)},{1.7*sin(30+24)})  -- ({2.25*cos(30+24)},{2.25*sin(30+24)}) arc (30+24:30+6:2.25)--({1.7*cos(30+6)},{1.7*sin(30+6)});
    \draw[] ({1.7*cos(30+19)},{1.7*sin(30+19)}) -- ({1.53*cos(30+19)},{1.53*sin(30+19)}) arc(30+19:30+6:1.53) -- ({1.7*cos(30+6)},{1.7*sin(30+6)});
    \draw[densely dotted] ({1.7*cos(30+19)},{1.7*sin(30+19)})  -- ({1.95*cos(30+19)},{1.95*sin(30+19)}) arc (30+19:30+1:1.95)--({1.7*cos(30+1)},{1.7*sin(30+1)});
    \draw[] ({1.7*cos(30+4)},{1.7*sin(30+4)}) -- ({1.53*cos(30+4)},{1.53*sin(30+4)}) arc(30+4:30+1:1.53) -- ({1.7*cos(30+1)},{1.7*sin(30+1)});
    \draw[densely dotted] ({1.7*cos(30+4)},{1.7*sin(30+4)})--({2.35*cos(30+4)},{2.35*sin(30+4)}) arc (30+4:30+26:2.35)--({1.7*cos(30+26)},{1.7*sin(30+26)});
    \draw[] ({1.7*cos(30+26)},{1.7*sin(30+26)}) -- ({1.53*cos(30+26)},{1.53*sin(30+26)}) arc (56:66:1.53);
    
%    \draw[] ({1.53*cos(30+26)},{1.53*sin(30+26)});
    \draw[dashed] ({1.53*cos(66)},{1.53*sin(66)}) arc (66:180-36:1.53);
    \draw[] ({1.53*cos(180-36)},{1.53*sin(180-36)}) arc (180-36:180-26:1.53)--({1.7*cos(180-26)},{1.7*sin(180-26)});

%Draw L near alpha_g, beta_g
    \draw[] ({1.7*cos(180+1)},{1.7*sin(180+1)}) -- ({1.53*cos(180+1)},{1.53*sin(180+1)});
    \draw[densely dotted] ({1.7*cos(180+1)},{1.7*sin(180+1)})  -- ({2.35*cos(180+1)},{2.35*sin(180+1)}) arc (180+1:180-21:2.35)--({2.35*cos(180-21)},{2.35*sin(180-21)})--({1.7*cos(180-21)},{1.7*sin(180-21)});
    \draw[] ({1.7*cos(180-21)},{1.7*sin(180-21)}) -- ({1.53*cos(180-21)},{1.53*sin(180-21)}) arc (180-21:180-24:1.53) -- ({1.7*cos(180-24)},{1.7*sin(180-24)});
\draw[densely dotted] ({1.7*cos(180-24)},{1.7*sin(180-24)})  -- ({1.95*cos(180-24)},{1.95*sin(180-24)}) arc (180-24:180-21:1.95);
\draw[densely dotted] ({1.95*cos(180-19)},{1.95*sin(180-19)}) arc (180-19:180-6:1.95)--({1.7*cos(180-6)},{1.7*sin(180-6)});
    \draw[] ({1.7*cos(180-19)},{1.7*sin(180-19)}) -- ({1.53*cos(180-19)},{1.53*sin(180-19)}) arc(180-19:180-6:1.53) -- ({1.7*cos(180-6)},{1.7*sin(180-6)});
    \draw[densely dotted] ({1.7*cos(180-19)},{1.7*sin(180-19)})  -- ({2.25*cos(180-19)},{2.25*sin(180-19)}) arc (180-19:180-1:2.25)--({1.7*cos(180-1)},{1.7*sin(180-1)});
    \draw[] ({1.7*cos(180-4)},{1.7*sin(180-4)}) -- ({1.53*cos(180-4)},{1.53*sin(180-4)}) arc(180-4:180-1:1.53) -- ({1.7*cos(180-1)},{1.7*sin(180-1)});
    \draw[densely dotted] ({1.7*cos(180-4)},{1.7*sin(180-4)})--({2.05*cos(180-4)},{2.05*sin(180-4)}) arc (180-4:180-19:2.05);
    \draw[densely dotted] ({2.05*cos(180-21)},{2.05*sin(180-21)}) arc (180-21:180-26:2.05)--({1.7*cos(180-26)},{1.7*sin(180-26)});
    \draw[] ({1.7*cos(180-26)},{1.7*sin(180-26)}) -- ({1.53*cos(180-26)},{1.53*sin(180-26)});
    
    \draw ({1.53*cos(180+1)},{1.53*sin(180+1)})--({1.53*cos(-1)},{1.53*sin(-1)});
            
\end{tikzpicture}

\caption{Above: The elements $\alpha_{i,1}, \beta_{i,1},\dots, \alpha_{i,g}, \beta_{i,g}\in \pi_1(Y, p_i)$ are given by conjugating the components of a link $a_{i,1}, b_{i,1},\dots, a_{i,g}, b_{i,g}$ whose components bound disjoint immersed disks in a contractible 4-manifold by embedded arcs $c_{i,1}, d_{i,1}, \dots c_{i,g}, d_{i,g}$.  Below: A knot $L_i$ which is freely homotopic to $\prod_{j=1}^{g_i} [\alpha_{i,j}, \beta_{i,j}]$ which bounds a Seifert surface $S_i$ so that $a_{i,1},\dots, a_{i,g}$ sit on $S_i$ and $b_{i,1},\dots, b_{i,g}$ sit on the normal pushoff $S_i'$.} 
\label{fig:commutator to surface}
\end{figure}

Finally, for each $i$, the disk $D_i$ required by the theorem is produced by starting with $F_{i}$ and performing ambient surgery using the disks $\{ \Delta_{a_{i,j}}\}_{1\leq j \leq g_i}$. The disk $D_i^+$ is produced similarly by starting with $F_i^+$ and
  performing ambient surgery using $\{ \Delta_{b_{i,j}}\}_{1\leq j \leq g_i}$.  This completes the proof.
  \end{proof}

We are ready to prove Theorem~\ref{thm:main}, which we restate here.

\begin{reptheorem}{thm:main}
Every link in a homology sphere is freely homotopic to a freely slice link.
\end{reptheorem}

\begin{proof}

Let $L$ be a link in a homology sphere $Y$, and $X$ be the unique contractible 4-manifold bounded by $Y$.  As in the proof of Proposition~\ref{prop:get immersed disks +}, pick $n$ distinct points $p_1,\dots, p_n$ in $Y$.  Since $Y$ is a homology sphere, each $L_i$ is freely homotopic to a product of commutators $ \prod_{j=1}^{g_i} [\alpha_{i,j}, \beta_{i,j}]$ for some $\alpha_{i,j}, \beta_{i,j}\in \pi_1(Y,p_i).$ So far the proof is similar to the proof of Proposition~\ref{prop:get immersed disks +}, but now we replace the reference to Proposition~\ref{prop:get immersed disks} with the reference to Proposition~\ref{prop:get immersed disks +}, so that we obtain a link $$\mathcal{L}:= \Cup_{i=1}^n\Cup_{j=1}^{g_i} a_{i,j}\cup b_{i,j}$$ such that for each $i,j$, the components $a_{i,j}$ are $b_{i,j}$ are freely homotopic to $\alpha_{i,j}$ and $\beta_{i,j}$ respectively. Moreover, there are disjoint immersed disks 
$$\left\{\Delta_{a_{i,j}}, \Delta_{b_{i,j}},\Delta_{a_{i,j}}^+, \Delta_{b_{i,j}}^+ \mid 1\leq i \leq n, 1\leq j \leq g_i\right\}$$
in $X$ bounded by $\mathcal{L} \cup \mathcal{L}^+$ where $\mathcal{L}^+$ is a zero linking parallel copy of $\mathcal{L}.$

As in the proof of Proposition~\ref{prop:get immersed disks +}, we obtain a boundary link $J=J_1\cup\dots\cup J_n$ which is freely homotopic to $L$,  a boundary link Seifert surface $F =F_1 \cup \cdots \cup F_n$ for $J$, and a symplectic basis $\{A_{i,j}, B_{i,j}\}_{1\leq i \leq n, 1\leq j \leq g_i}$ on $F$ such that $A_{i,j} = a_{i,j}$ and $B_{i,j}' = b_{i,j}$ for each $i,j$. Again, here $B_{i,j}'$ is the positive pushoff of $B_{i,j}$ with respect to $F_i$, and we may assume that the Seifert framings on $A_{i,j}$ and $B_{i,j}$ are the zero framings.

We claim that $J$ satisfies all of the assumptions of Theorem~\ref{thm:CKP}.  Since linking numbers can be computed in terms of intersections of bounded disks, we see that the Seifert matrix for $F$ with respect to the basis $\{A_{i,j}, B_{i,j}\}_{1\leq i \leq n, 1\leq j \leq g_i}$ has the form\setlength\dashlinedash{0.2pt}
\setlength\dashlinegap{1.5pt}
\setlength\arrayrulewidth{0.3pt}
\[
\left[
\begin{array}{c:c:c:c:c}
\begin{array}{cc}0&\epsilon_1
\\1-\epsilon_1&0\end{array} &
\begin{array}{cc}0&*
\\0&*\end{array} &
\begin{array}{cc}0&*
\\0&*\end{array} 
&\dots&
\begin{array}{cc}0&*
\\0&*\end{array} 
\\
\hdashline
\begin{array}{cc}0&0
\\**&*\end{array} &
\begin{array}{cc}0&\epsilon_2
\\1-\epsilon_2&0\end{array} &
\begin{array}{cc}0&*
\\0&*\end{array}
&\dots&
\begin{array}{cc}0&*
\\0&*\end{array} 
 \\
\hdashline
\begin{array}{cc}0&0
\\**&*\end{array} &
\begin{array}{cc}0&0
\\**&*\end{array} &
\begin{array}{cc}0&\epsilon_3
\\1-\epsilon_3&0\end{array}
&\dots&
\begin{array}{cc}0&*
\\0&*\end{array} 
 \\
\hdashline
\vdots&\vdots&\vdots&\ddots&\vdots\\
\hdashline
\begin{array}{cc}0&0
\\**&*\end{array} &
\begin{array}{cc}0&0
\\**&*\end{array} &
\begin{array}{cc}0&0
\\**&*\end{array} 
&\dots&
\begin{array}{cc}0&\epsilon_g
\\1-\epsilon_g&0
\end{array}
\end{array}
\right]
\]
with each $*$ and each $\epsilon_i$ equal to zero. We have now produced a good basis.  The existence of the disjoint disks 
$\{\Delta_{a_{i,j}}, \Delta_{b_{i,j}}, \Delta_{b_{i,j}}^+\}_{1\leq i \leq n, 1\leq j \leq g_i}$ implies that $\{a_{i,j}, b_{i,j}\}_{1\leq i \leq n, 1\leq j \leq g_i}$ is a good disky basis. Thus, by Theorem~\ref{thm:CKP} we conclude that $J$ is freely slice.\end{proof}

\section{Whitney tower concordance and links in homology spheres}\label{sect: towers}

In this section, we will explain how to use the relative Whitney trick to construct Whitney towers.  We will begin by recalling the definition of a (non-relative) Whitney disk, and Whitney tower; see also~\cite[Section 2.1]{MR3770161}, for example.

Let $S$ be an immersed oriented surface in a 4-manifold $W$ with double points $p$ and $q$ of opposite signs.  Let $\alpha_1$ and $\alpha_2$ be embedded arcs in $S$, running from $p$ to $q$ and $q$ to $p$ respectively. Assume that $\alpha_1$ and $\alpha_2$ meet the double point set of $S$ only at $\{p,q\}$, that $\alpha_1\cap\alpha_2=\{p,q\}$, and that near both $p$ and $q$ the arcs are in different local sheets. Let $\Delta$ be an immersed disk in $W$ bounded by the circle $\gamma:=\alpha_1\ast\alpha_2$, and with interior transverse to $S$. The normal bundle of $\Delta$ has a unique trivialisation; its restriction to $\gamma$ determines a choice of framing for the trivial $2$-plane bundle over $\gamma$. At each point on $\alpha_i$ let $v_i$ be the tangent direction in $S$ normal to $\alpha_i$, and let $u_i$ be a common normal to both $\Delta$ and $S$, chosen to be a section of the normal bundle of $\Delta$ restricted to $\alpha_i$. These can moreover be chosen so that at $p$ and $q$, we have $u_1=v_2$ and $u_2=v_1$. These combine to determine a second framing for the trivial $2$-plane bundle over $\gamma$. If the two framings described above agree, then $\Delta$ is a called a \emph{Whitney disk} pairing the double points at $p$ and $q$.

A \emph{Whitney tower} is a special type of union of immersed surfaces.  The precise definition is recursive.  A union of properly immersed oriented surfaces in a 4-manifold $W$ which are transverse to each other is a Whitney tower.   Let $T$ be a Whitney tower and $\Delta$ be a Whitney disk pairing two intersections of opposite signs between 
surfaces in $T$.  Suppose also that  $\Delta$ is disjoint from the boundary of every surface in $T$. Then $T\cup \Delta$ is a Whitney tower.  

The various immersed surfaces which make up a Whitney tower have an associated \emph{order}.  The initial surfaces in a Whitney tower $T$ are called \emph{order 0} surfaces in $T$.  A point in the intersection of an order $k$ and an order $\ell$ surface in $T$ is called an \emph{order $k+\ell$} intersection.  A Whitney disk pairing two order $k$ intersections is called an \emph{order $k+1$} Whitney disk.  If all intersection points of order less than $k$ are paired by Whitney disks, then $T$ is called an \emph{order $k$} Whitney tower. 

Given an intersection point $p$ in a Whitney tower $T$, it may be that in the 4-manifold $W$ there is no Whitney disk pairing $p$ with another intersection point in $T$; as a consequence, there is a filtration of link concordance, as we now describe. Suppose that $W$ is a 4-manifold with $\bdry W = \bdry_+ W \cup -\bdry_- W$.  Two $n$-component links $L\subseteq \bdry_+ W$ and $J\subseteq \bdry_- W$ are \emph{order $k$ Whitney tower concordant in $W$} if there is an order $k$ Whitney tower $T$ in $W$ so that the order $0$ surfaces of $T$ are $n$ immersed annuli $A_1,\dots, A_n$ with $\bdry A_i = L_i\cup-J_i$.  

\begin{definition}\label{def:Whitneyconcordant} If $L$ and $J$ are links in homology spheres that are order $k$ Whitney tower concordant in a simply connected homology cobordism between the homology spheres, then we say that $L$ and $J$ are \emph{order $k$ Whitney tower concordant} and write $L\simeq_k J$. \end{definition}

\begin{remark}
Particularly for links in $S^3$, the equivalence relation from Definition~\ref{def:Whitneyconcordant} has been the subject of deep study and is known to be highly nontrivial. The reader is directed to \cite{MR3770161, CST2012, CST2014}, for example, for further background and results.\end{remark}

The main goal of this section is to prove Theorem~\ref{thm:WhitneyTowerCobordism}. For convenience, we recall the statement.

\begin{reptheorem}{thm:WhitneyTowerCobordism}
If $L$ is link in a homology sphere and $k$ is a nonnegative integer, then there is a link $J$ in $S^3$ such that~$L\simeq_k J$.
\end{reptheorem}

In order to construct the link $J$ in Theorem~\ref{thm:WhitneyTowerCobordism}, as well as the needed Whitney tower concordance, we  extend the idea of a relative Whitney disk to an object analogous to a Whitney tower with relative Whitney disks in place of Whitney disks.

A \emph{relative Whitney tower} is recursively defined as follows.  A union of properly immersed oriented surfaces in a 4-manifold $W$ which are transverse to each other is a relative Whitney tower.  
Let $T$ be a relative Whitney tower and $\Delta$ be a relative Whitney disk associated with a double point in $T$.  Suppose that $\Delta$ is disjoint from the boundary of every surface in $T$ other than the endpoints of its relative Whitney arc.
 Then $T\cup \Delta$ is a relative Whitney tower.

Similarly to Whitney towers, relative Whitney towers have an associated order.  The initial surfaces in a Whitney tower $T$ are called \emph{order 0 surfaces} of $T$.  A point in the intersection of an order $k$ and an order $\ell$ surface in $T$ is called an \emph{order $k+\ell$ intersection}.  A relative Whitney disk associated to an order $k$ intersection is called an \emph{order $k+1$ relative Whitney disk}.  If all intersection points of order less than $k$ have relative Whitney disks in $T$, then $T$ is called an \emph{order $k$ relative Whitney tower.}

\begin{remark}The proof of Proposition~\ref{prop:get immersed disks} involved the construction of an object which is a relative Whitney tower.\end{remark}

The proof of Theorem~\ref{thm:WhitneyTowerCobordism} will require two lemmas which we hope also provide evidence that relative Whitney towers are interesting and useful.  First, in Lemma~\ref{lem: rel Whit tow exist} we show that they exist much more readily than Whitney towers.  Secondly, in Lemma~\ref{lem: rel Whit tow to Whit tow} we explain how a relative Whitney tower can be modified to produce a Whitney tower.

\begin{lemma}\label{lem: rel Whit tow exist}
Let $W$ be a 4-manifold and $S = S_1\cup\dots\cup S_n$ be a  union of properly immersed connected oriented surfaces in $W$. If $Y$ is a connected submanifold of $\bdry W$ such that $\pi_1(Y)\to \pi_1(W)$ is surjective, and $\bdry S_i\cap Y\neq \emptyset$ for each $i$,  then for any nonnegative integer $k$, there is an order $k$ relative Whitney tower $T$ whose order 0 surfaces are precisely $S$ and for which all relative Whitney arcs are contained in $Y$.
\end{lemma}

\begin{lemma}\label{lem: rel Whit tow to Whit tow}
Let $W$ be a 4-manifold and $T$ be an order $k$ relative Whitney tower.  If $Y$ is a connected submanifold of $\bdry W$ and contains all relative Whitney arcs of $T$, then there exists an order $k$ Whitney tower $T'$ such that the order 0 surfaces of $T$ and $T'$ differ by a homotopy which is constant outside of a small neighbourhood of $Y$. \end{lemma}

Armed with these lemmas we can prove Theorem~\ref{thm:WhitneyTowerCobordism}.

\begin{proof}[Proof of Theorem~\ref{thm:WhitneyTowerCobordism} $($assuming Lemmas~\ref{lem: rel Whit tow exist} and~\ref{lem: rel Whit tow to Whit tow}$)$] 
Let $L$ be a link in a homology sphere $Y$, let $k$ be a nonnegative integer, and let $X$ be the contractible 4-manifold bounded by $Y$. Let $W$ be a 4-manifold obtained by removing an open 4-ball from $X$.  Note that $W$ is a simply connected homology cobordism from $Y$ to $S^3$.  Since $W$ is simply connected, the components of $L$ are freely homotopic in $W$ to the components of the unlink in $S^3$.  Thus, there exists a collection of immersed annuli $A_1,\dots, A_n$ so that $A_i$ is bounded by $L_i$ and the $i^{\text{th}}$ component of the unlink in $S^3$.  

By Lemma~\ref{lem: rel Whit tow exist}, there is an order $k$ relative Whitney tower $T$ whose order $0$ surfaces are precisely $A_1,\dots, A_n$.  Moreover, we can ensure that all relative Whitney arcs of $T$ are contained in $S^3$. By Lemma~\ref{lem: rel Whit tow to Whit tow}, there is an order $k$ Whitney tower $T'$ whose order 0 surfaces are homotopic to $A_1,\dots, A_n$ by a homotopy which is constant away from a neighbourhood of $S^3$. In particular, the $i^{\text{th}}$ component of the order 0 surface of $T'$ is bounded by $L_i$ in $Y$ and by some knot $J_i$ in $S^3$.   Setting $J = J_1\cup\dots\cup J_n$, we conclude that $L\simeq_k J$, completing the proof.  
\end{proof}

Next, we prove Lemmas~\ref{lem: rel Whit tow exist} and \ref{lem: rel Whit tow to Whit tow}.  The proof of Lemma~\ref{lem: rel Whit tow exist} follows a relatively straightforward induction, which we now present.

\begin{proof}[Proof of Lemma~\ref{lem: rel Whit tow exist}]

Let $W$ be a 4-manifold and suppose $Y$ is a submanifold of $\bdry W$ such that $\pi_1(Y)\to \pi_1(W)$ is surjective. Consider also a collection of connected oriented properly immersed surfaces $S_1,\dots, S_n$ in $W$ such that $\bdry S_i\cap Y\neq \emptyset$ for each $i$. 

We will  inductively prove that for every nonnegative integer $k$, there is an order $k$ relative Whitney tower with order $0$ surfaces $S_1,\dots, S_n$ and whose relative Whitney arcs are contained in $Y$.  When $k=0$ there is nothing to prove.

Let $T$ be an order $k$ relative Whitney tower satisfying the above properties.
Let $p$ be an order $k$ intersection point in $T$.  Then $p$ is contained in the intersection of two surfaces $A$ and $B$ in $T$ of order $a$ and $b$ respectively where $a+b=k$.  If $a=0$, then $A$ is an order $0$ surface $S_i$ and by assumption, there is an embedded arc $\alpha$ in $A$ running from $p$ to a point $q$ in $Y$. If $a>0$, then $A$ is  a relative Whitney disk of $T$ and there is an embedded arc $\alpha$ in $A$ running from $p$ to a point $q$ on the associated relative Whitney arc. A schematic including $\alpha$ appears to the left of Figure~\ref{fig: Second order Whit disk}. By assumption this relative Whitney arc is contained in $Y$ and hence so is $q$. For any other relative Whitney disk $\Delta$ in $T$, we have that  $\bdry \Delta \cap A$ is either empty or is an embedded arc in $A$ with one endpoint in $\bdry A$ and the other interior to $A$.  Thus, we can arrange that except for its endpoint on $\bdry A$, $\alpha$ is disjoint from the boundary of every surface in $T$.
Similarly, there is an embedded arc $\beta$ contained in $B$ running from $p$ to a point $r$ in $Y$.  

\begin{figure}[h!]
         \begin{tikzpicture}[scale=1]
         \node at (0,0){\includegraphics[width=37.5mm]{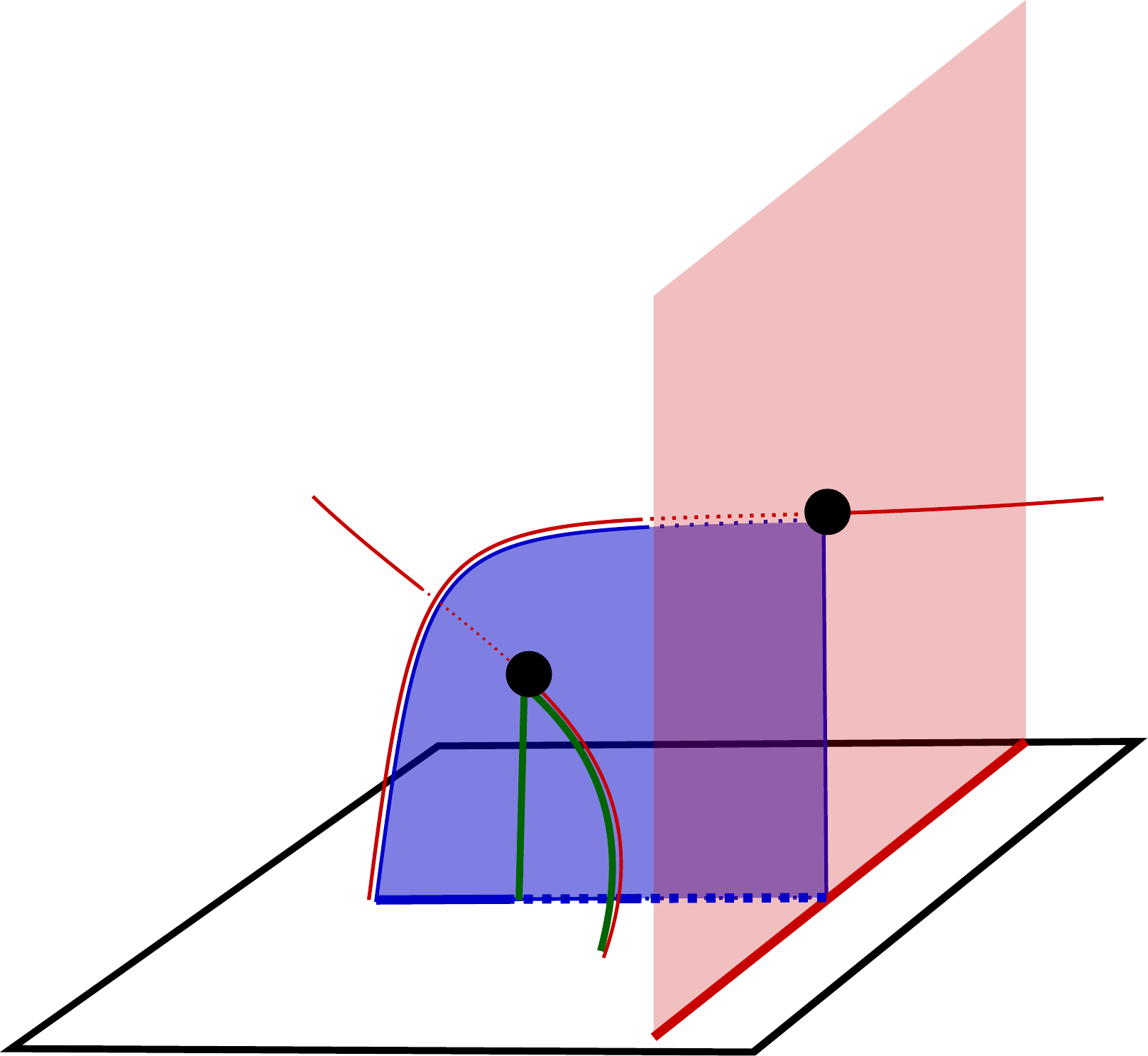}};
         \node at (-.1,.7) {$p$};
         \draw (-.1,.5)--(-.14,-.55);
         \node at (-1.5,-.4) {$\alpha$};
         \draw (-1.3,-.45)--(-.17,-.9);
         \node at (-1.05,-1.45) {$\beta$};
         \draw (-.9,-1.4)--(.11,-.95);
         \node at (5,0){\includegraphics[width=37.5mm]{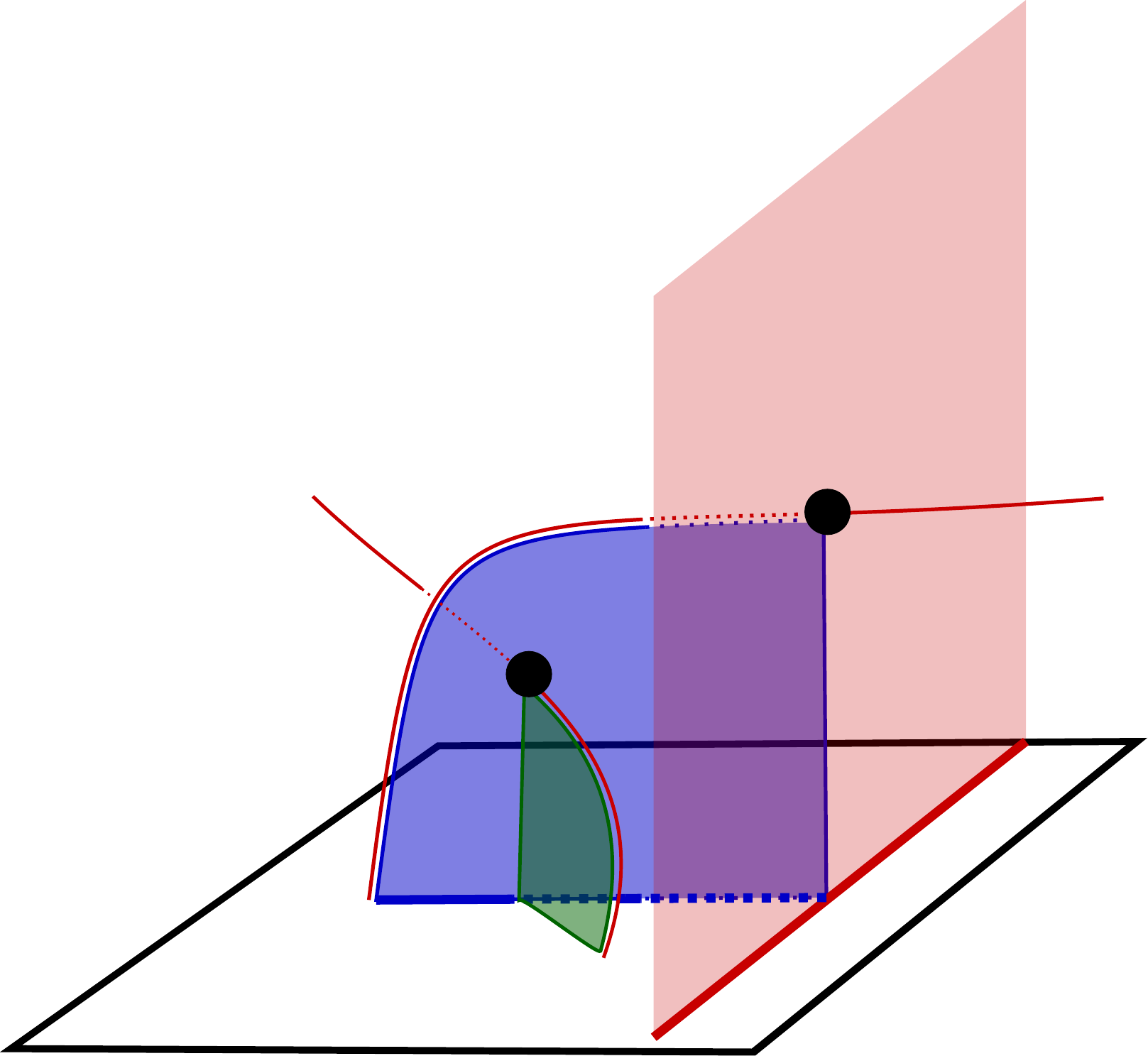}};
         \end{tikzpicture}     
        \caption{Left: A relative Whitney disk in $T$ containing an intersection point $p$ along with arcs $\alpha$ and $\beta$ in different sheets of $T$ from $p$ to the boundary.  Right: A relative Whitney disk associated with $p$.}
        \label{fig: Second order Whit disk}
\end{figure}
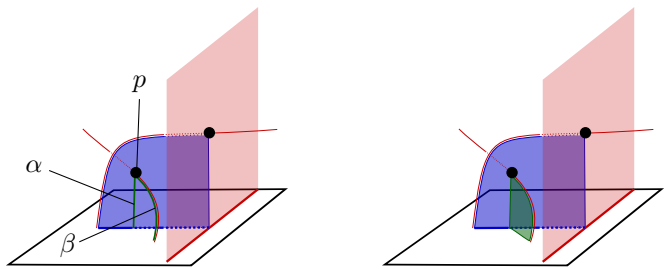

Since $\pi_1(Y)\to \pi_1(W)$ is onto, there is an embedded arc $\gamma$ in $Y$ from $r$ to $q$ so that $\alpha*\beta*\gamma$ bounds an immersed disk $\Delta_p$ in $W$.
  Thus every order $k$ intersection point in $Y$ admits a relative Whitney disk with relative Whitney arc in $Y$.  After isotoping the interior of $\Delta_p$ we have that it is disjoint from the boundary of each surface in $T$. By replacing $T$ by $T\cup \Delta_p$ we arrange that the order $k$ intersection now has a relative Whitney disk in $T$.   Observe that $\Delta_p$ has order $k+1$, so any intersections in its interior are of order at least $k+1$; adding $\Delta_p$ to $T$ adds no new intersections of order up to $k$. By adding a relative Whitney disk to $T$ for every such intersection point, we produce an order $k+1$ relative Whitney tower and complete the proof.\end{proof}

Next we explain how to use a relative Whitney tower to find an honest Whitney tower, and prove Lemma~\ref{lem: rel Whit tow to Whit tow}. Our argument will employ objects interpolating between Whitney towers and relative Whitney towers. They are akin to both Whitney towers and relative Whitney towers, allowing for both Whitney disks and relative Whitney disks.

A \emph{mixed Whitney tower} is defined recursively as follows.  A union of properly immersed surfaces in a 4-manifold $W$ which are transverse to each other is a mixed Whitney tower.  Let $T$ be a mixed Whitney tower and $\Delta$ be either a Whitney disk pairing two intersection points between surfaces in $T$ or a relative Whitney disk associated with a single intersection between surfaces in $T$.  If $\Delta$ is a Whitney disk then we require $\Delta$ be disjoint from the boundary of any surface in $T$.  If $\Delta$ is a relative Whitney disk then $\Delta$ is disjoint from the boundary of any surface in $T$ away from the endpoints of its relative Whitney arc. Then $T\cup \Delta$ is a mixed Whitney tower.

Mixed Whitney towers have an associated order.  The initial surfaces in a mixed Whitney tower $T$ are called \emph{order 0 surfaces} in $T$.  A point in the intersection of an order $k$ and an order $\ell$ surface in $T$ is called an \emph{order $k+\ell$ intersection}.  
A Whitney disk pairing two order $k$ intersections is called an \emph{order $k+1$ Whitney disk}.
A relative Whitney disk associated to an order $k$ intersection is called an \emph{order $k+1$ relative Whitney disk}.   If all intersection points of order less than $k$ have either associated Whitney disks or relative Whitney disks in $T$, then $T$ is called an \emph{order $k$ mixed Whitney tower.}

\begin{proof}[Proof of Lemma~\ref{lem: rel Whit tow to Whit tow}.]

For an order $k$ mixed Whitney tower $T$ and a natural number $\ell\leq k$, if $T$ has no relative Whitney disks of order greater than $\ell$ then we say $T$ \emph{transitions at $\ell$}.  Clearly a relative Whitney tower of order $k$ is a mixed Whitney tower which transitions at $k$.  Additionally, a mixed Whitney tower is a  Whitney tower if and only if it transitions at $0$.  Thus if we can explain how to lower the parameter $\ell$ at which a given mixed Whitney tower transitions, then {induction will complete the proof}. Precisely, let $W$ be a 4-manifold containing a mixed Whitney tower $T$ of order $k$ which transitions at $\ell$. Moreover, suppose $Y$ is a submanifold of $\bdry W$ and contains all relative Whitney arcs of $T$. We claim that there exists a Whitney tower $T'$ of order $k$ which transitions at $\ell-1$, so that all relative Whitney arcs are contained in $Y$, and so that the order 0 surfaces of $T$ and $T'$ differ by a homotopy which is constant outside of a small neighbourhood of $Y$.

The proof of this claim will proceed by changing $T$ by homotopies introducing new intersection points so that order $\ell$ relative Whitney disks become Whitney disks.  Let $p$ be an order $\ell-1$ intersection point sitting in the intersection of surfaces $A$ and $B$ in $T$ and let $\Delta_p$ be an associated order $\ell$ relative Whitney disk.  The move drawn schematically in Figure~\ref{fig: rel Whit disk to Whit disk} changes $A$ by a homotopy, and adds a new intersection point in $A\cap B$.  A subdisk $\Delta'$ of $\Delta_p$ forms a Whitney disk pairing $p$ and this new intersection point.  We proceed to explain this move.

\begin{figure}[h]
\begin{picture}(310,105)
%\put(310,100){*}
\put(0,0){\includegraphics[height=37mm]{relWhitDisk.pdf}}

\put(90,80){$A$}
\put(110,55){$B$}
\put(85,60){$p$}

\put(55,35){$\Delta_p$}

\put(160,0){\includegraphics[height=37mm]{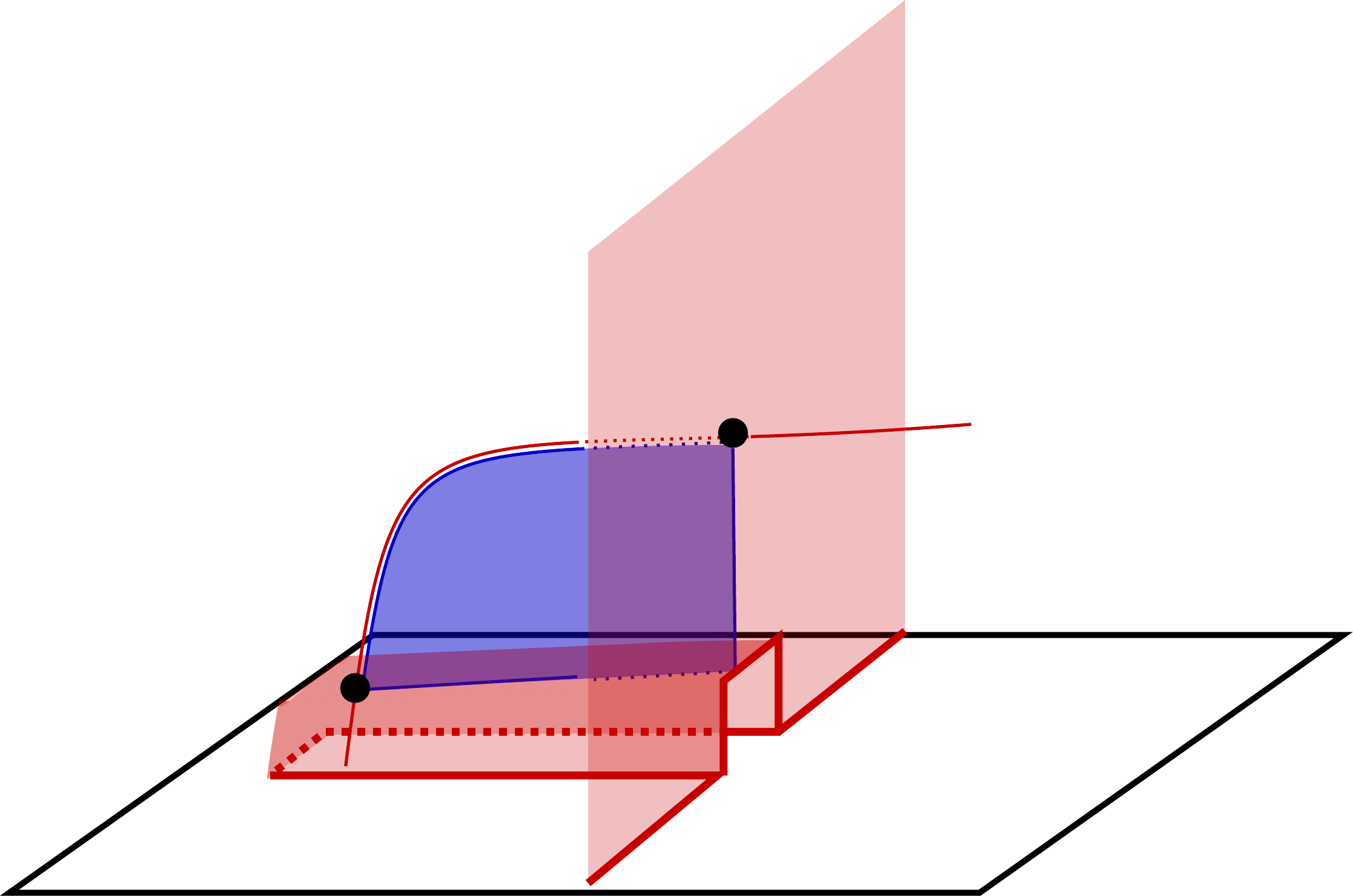}}

\put(250,80){$A'$}
\put(270,55){$B'$}
\put(245,60){$p$}

\put(215,40){$\Delta'$}
\end{picture}
\caption{Left:  A relative Whitney disk associated to the point $p\in A\cap B$.  Right:  After changing $A$ by a homotopy, we see a new point of intersection and a Whitney disk.}\label{fig: rel Whit disk to Whit disk}
\end{figure}

As in Section \ref{sect: rel Whit trick}, and using the same notation, we find an immersion $\Phi:\Delta_{xyz}\times \RR^2\to W$ parametrizing a tubular neighbourhood of $\Delta_p$.  Since $T$ contains no relative Whitney disks of order greater than $\ell$, we see that $T$ contains no relative Whitney disks associated with intersection points on $\Delta_p$.  As a consequence, if we take the regions $Q$ and $R$ in Figure~\ref{fig: triangle 2} close enough to $\lineseg{x}{z}$ then $\Phi$ will be an embedding when restricted to ${Q\cup R\times\RR^2 }$ and will have image disjoint from $T$ except that $\Phi(Q\cup R\times\{(0,0)\})\subseteq \Delta_p$,
 $\Phi(\lineseg{x}{x'}\times\RR\times\{0\})\subseteq A$ and $\Phi(\lineseg{z}{z''}\times\{0\}\times\RR)\subseteq B$.

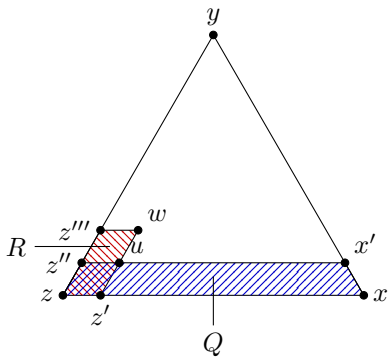
\begin{figure}[h]
\begin{tikzpicture}[scale=.5]
\tikzset{
    partial ellipse/.style args={#1:#2:#3}{
        insert path={+ (#1:#3) arc (#1:#2:#3)}
    }
}
%\node at (0, 0)   (z) {$z$};

\node[] at (4,0) (x) {};
\node[] at (0,{8*sin(60)}) (y) {};
\node[] at (-4,0) (z) {};
\node[] at ({4-cos(60)}, {sin(60)}) (x') {};
\node[] at (-3,0) (z') {};
\node[] at ({-4+cos(60)},{sin(60)}) (z'') {};
\node[] at ({-4+2*cos(60)},{2*sin(60)}) (z''') {};
\node[] at ({-4+2*cos(60)+1},{2*sin(60)}) (w) {};
\node[] at ({-4+cos(60)+1},{sin(60)}) (u) {};

%\draw[diagramred](x'.center)--(z''.center);

\node[left] at (z) {$z$};
\node[below] at (z') {$z'$};
\node[left] at (z'') {$z''$};
\node[left] at (z''') {$z'''$};
\node[right] at (x) {$x$};
\node[above] at (y) {$y$};
\node[above right] at (x') {$x'$};
\node[above right] at (w) {$w$};
\node[above right] at (u) {$u$};

\draw (x.center)--(y.center)--(z.center)--(x.center);
\draw[pattern=north west lines, pattern color=diagramred]  (z'.center)--(w.center)--(z'''.center)--(z.center);
\draw[pattern=north east lines, pattern color=diagramblue]  (x.center)--(x'.center)--(z''.center)--(z.center);
%\draw[fill=lightgray] (z'.center)--(z.center)--(y.center)--(y'.center)--(z'.center);

\draw[fill=black] (x) circle (.1);
\draw[fill=black] (y) circle (.1);
\draw[fill=black] (z) circle (.1);
\draw[fill=black] (x') circle (.1);
\draw[fill=black] (z') circle (.1);
\draw[fill=black] (z'') circle (.1);
\draw[fill=black] (z''') circle (.1);
\draw[fill=black] (w) circle (.1);
\draw[fill=black] (u) circle (.1);

\node at (0,-1.3){$Q$};
\draw (0,-.8)--(0,.5);

\node at ({-6+1.5*cos(60)},{1.5*sin(60)}) {$R$};
\draw({-5.5+1.5*cos(60)},{1.5*sin(60)})--({-3.5+1.5*cos(60)},{1.5*sin(60)});

\end{tikzpicture}
\caption{The triangle $\Delta_{xyz}$, with a highlighted quadrilaterals $Q$ and $R$ having vertex sets $\{x,x',z'',z\}$ and $\{z,z',w,z'''\}$ respectively.}  \label{fig: triangle 2}
\end{figure} 

Modify $A$ and $B$ using $\Phi(Q\times[-1,1]\times\{0\})$ and $\Phi(R\times\{0\}\times[-1,1])$  as guides.  That is, set
\begin{multline*}
A':=
\left(A\smallsetminus\Phi\left(\lineseg{x}{x'}\times[-1,1]\times\{0\}\right)\right) 
\cup \Phi\left(Q\times\{-1,1\}\times\{0\} \right)
\\\cup \Phi\left(\left(\lineseg{x'}{z''}\cup\lineseg{z''}{z}\right)\times[-1,1]\times\{0\}\right),
\end{multline*}
and 
\begin{multline*}
B':=\left(B\smallsetminus\Phi\left(\lineseg{z}{z''}\times\{0\}\times[-1,1]\right)\right) 
\cup \Phi\left(R\times\{0\}\times\{-1,1\}\right)
\\\cup\Phi\left( \left(\lineseg{z'''}{w}\cup\lineseg{w}{z'}\right)\times\{0\}\times[-1,1]\right).
\end{multline*}

The embedded (but non-disjoint) cubes $\Phi(Q\times[-1,1]\times\{0\})$ and $\Phi(R\times\{0\}\times[-1,1])$ parametrize a homotopy from $A\cup B$ to $A'\cup B'$ which is constant outside of $\Phi(Q\cup R\times[-1,1]\times[-1,1])$. 
{The relative Whitney arc associated to $\Delta_p$ is parametrized by $\Phi(\lineseg{x}{y})$, and by assumption this is contained in $Y$.} 
  Thus, by taking the tubular neighbourhood of $\Delta_p$ small enough, and the regions $Q$ and $R$ close enough to $\lineseg{x}{z}$, we can arrange that $\Phi(Q\cup R\times[-1,1]\times[-1,1])$ is contained in a small neighbourhood of $Y$.  The homotopy constructed is constant outside of this small neighbourhood of $Y$.

By direct inspection, $A'\cap B' = \left(A\cap B\right) \cup \{q\}$ where $q=\Phi(u,0,0)$.  Moreover, $\Delta'$, the closure of $\Phi(\Delta_{xyz}\smallsetminus (Q\cup R))$, gives a Whitney disk pairing $p$ and $q$.  Thus if $T'$ is given by replacing $A$, $B$, and $\Delta_p$ by $A'$, $B'$, and $\Delta'$ then $T'$ is still an order $k$ mixed Whitney tower and it has one fewer relative Whitney disks of order $\ell$ than $T$.  

By iterating the procedure above, we replace all order $\ell$ relative Whitney disks in $T$ with Whitney disks.  The result is an order $k$ mixed Whitney tower which transitions at $\ell-1$.  Induction on $\ell$ now completes the proof.
\end{proof}

 \section{Homotopy trivializing numbers for links in homology spheres.}\label{sect: Homotopy trivializing}

In the method we described for separating an immersed disk collection in Section \ref{sec:separating}, we obtained a precise relationship between the number of crossing changes between link components in the boundary, and the number of intersection points removed in the cobounding disk collection. In this section, we study this relationship in more detail. We introduce two link homotopy invariants (which we then prove to be the same). We provide precise calculations of these invariants for links of up to~3 components and more generally we provide bounds for these invariants.

\subsection{The homotopy trivializing number coincides with the disk intersection number}\label{subsec:nh=nd}

\begin{definition} Let $L=L_1\cup \dots \cup L_n$ be a link in a homology sphere $Y$. The \emph{disk intersection number} of $L$ is 
\[
n_d(L) := \min\left\{\textstyle{\sum_{i<j}\#(D_i\cap D_j)}\right\}
\]
where the minimum is taken over all collections of immersed disks $D_1\cup\dots \cup D_n$ in the contractible $4$-manifold bounded by $Y$, with boundary the link, and meeting one-another transversely.
\end{definition}

Recall that a link in a homology sphere is called \emph{4D-homotopically trivial} if it bounds disjoint immersed disks in the unique contractible $4$-manifold bounded by the homology sphere. 

\begin{definition}
The \emph{homotopy trivializing number} of a link $L$ is given by minimizing the Gordian distance $d_G$ from $L$ to a 4D-homotopically trivial link.  That is,
\[
n_h(L): = \min\{d_G(L,J) \mid J \text{ is 4D-homotopically trivial}\}.
\]
In more detail, we say that $d_G(L,J)\le m$ if there is a collection of disjoint 3-balls  $B_1,\dots, B_m$ in $Y$ so that $(B_i, B_i\cap L)$ is orientation preserving homeomorphic to one of the tangles in Figure~\ref{fig: crossing} and that $J$ is isotopic to the result of changing $L$ by a homotopy supported on $B_1\cup\dots \cup B_m$ replacing each positive crossing by a negative crossing and conversely. 
\end{definition}

We now recall Proposition~\ref{prop:nd=nh} with this added language and give a proof.
\begin{repproposition}{prop:nd=nh} 
For any link $L$ in any homology sphere $Y$ there is an equality
$n_d(L) = n_h(L)$.
\end{repproposition}
\begin{proof}
 Let $L$ be an $n$-component link in $Y$ with $n_d(L) = m$. Let $X$ be the contractible 4-manifold bounded by $Y$ and $D=D_1\cup\dots \cup D_n\subseteq W$ be a collection of immersed disks bounded by the components of $L$ with $\sum_{i<j}\#(D_i\cap D_j) = m$.  Since $X$ is contractible, the inclusion induced map $\pi_1(Y)\to\pi_1(X)$ is trivially surjective and we may apply Proposition~\ref{prop:get immersed disks} to see that there is a homotopy consisting of $m$ crossing changes from $L$ to a new link $J$ which is 4D-homotopically trivial. Thus, $n_d(L)\ge n_h(L)$, proving one of the needed inequalities.  

Now assume that $n_h(L)=m$ and let $J$ be a 4D-homotopically-trivial link obtained from $L$ by making $m$ crossing changes. If we cap off the trace of the homotopy from $L$ to $J$ in $Y\times[0,1]$ with a collection of disjoint immersed disks in $X$ bounded by the components of $J$ then we see a collection of immersed disks $D_1\cup\dots\cup D_n$ bounded by the components of $L$ with $\sum_{i<j}\#(D_i\cap D_j) \leq m$.  Thus, $n_d(L)\le n_h(L)$.
\end{proof}

\begin{remark}\label{rmk: homotopy invariance of nh}
The disk intersection number $n_d(L)$ is invariant under link homotopy.  As a consequence of Proposition~\ref{prop:nd=nh}, so is $n_h(L)$.  We take a moment to explain. Let $L$ and $J$ be links in $Y$. If they are link homotopic, then the trace of a link homotopy from $L$ to $J$ is an union of disjoint immersed annuli in $Y\times[0,1]$ cobounded by $L\times\{0\}$ and $J\times\{1\}$.  Let $X$ be the contractible 4-manifold bounded by $Y$.  By capping $J\times \{1\}$ with a collection of immersed disks bounded by $J$ in $X$ we see that $n_d(L)\le n_d(J)$.  By symmetry $n_d(L) = n_d(J)$. 
\end{remark}

\subsection{Computations of the homotopy trivializing number}\label{subsec:boundnh}
Any 4D-homotopically trivial link has vanishing linking numbers and each time when we perform a crossing change the linking number changes at most by one. Hence we have the following obvious lower bound on the homotopy trivializing number:
\begin{equation}\label{eq:lambda}
\Lambda(L):=\Sum_{i<j}|\lk(L_i,L_j)| \leq n_h(L).
\end{equation}
Moreover, since the linking number $\lk(L_i, L_j)$ can be computed by taking any immersed disks bounded by $L_i$ and $L_j$ in a homology ball, and counting points of intersection with sign, we see that $n_h(L) = n_d(L) \equiv \Lambda(L)\pmod 2$.

For a link with $2$ or $3$ components we
 determine $n_h(L)$ completely.  
For links of more than 3 components we find a bound on the difference between $n_h(L)$ and $\Lambda(L)$.  
Remarkably, this upper bound depends only on the number of components of $L$, and in particular is independent of the higher order link homotopy invariants of Milnor~\cite{M1}. We restate Theorem~\ref{thm: compute nh}.

\begin{reptheorem}{thm: compute nh}
Let $L$ be a link in a homology sphere. The the following holds.
\begin{itemize}[leftmargin=0.9cm]\setlength\itemsep{0em}
\item If $L$ is a 2-component link, then $$n_h(L) = \Lambda(L).$$ 
\item If $L$ is a 3-component link, then 
$$n_h(L) = \begin{cases}
\Lambda(L)&\text{if }\Lambda(L)\neq 0\\
2&\text{if }\Lambda(L)= 0\text{ and }\mu_{123}(L)\neq 0\\
0&\text{otherwise.}\\
\end{cases}
$$
\item In general, there is some $C_n$ so that for every $n$-component link $L$,
$$
\Lambda(L)\le n_h(L)\le \Lambda(L)+C_n.
$$
\end{itemize}
\end{reptheorem}

\begin{remark}By Corollary~\ref{cor: disjoint annuli to S3}, for any link $L$ in a homology sphere, there is a link $J$ in $S^3$ such that in a simply connected homology cobordism from $Y$ to $S^3$ the components of $L$ and $J$ cobound disjoint immersed annuli. It follows that $n_d(L) = n_d(J)$.  By Proposition~\ref{prop:nd=nh}, we have that $n_h(L) = n_h(J)$.  Lastly, recall that if two links in homology spheres cobound disjoint immersed annuli, then they have the same pairwise linking number and Milnor's triple linking number. Therefore it suffices to prove Theorem~\ref{thm: compute nh} for links in $S^3$, and for which the notion of 4D-homotopically trivial and Milnor's more classical notion of link homotopically trivial are the same.
\end{remark}

The proof of Theorem~\ref{thm: compute nh} passes though the string link classification of Habegger-Lin \cite{HL1}. We take a moment and recall some of their definitions and tools.

Pick $n$ points $p_1,\dots, p_n$ in the disk $D^2$.  An $n$-component \emph{string link} $T=T_1\cup \dots\cup T_n$ is a disjoint union of embedded arcs in $D^2\times[0,1]$ with $T_i$ running from $p_i\times\{0\}$ to $p_i\times \{1\}$. Let $\mathcal{LH}_n$ denote the set of $n$-component string links in $S^3$.

The notion of link homotopy extends in an obvious way to string links; let $\mathcal{SLH}_n$ denote the set of string links up to link homotopy.  Importantly, $\mathcal{SLH}_n$ is a group under the stacking operation of Figure~\ref{fig: stacking}.  The definition of homotopy trivializing number $n_h$ extends to string links and, just as for links in $S^3$, depends only on the link homotopy class.  Notice that $n_h$ is subadditive under the stacking operation. The operation of Figure~\ref{fig: closure} given by sending a string link $T$ to its closure $\widehat{T}$ gives a surjection $\mathcal{SLH}_n\onto \mathcal{LH}_n$.  For any string link $T$, it is now immediate that
\begin{equation}\label{eq:closurenh}
n_h(\widehat{T})\le n_h(T).
\end{equation}

 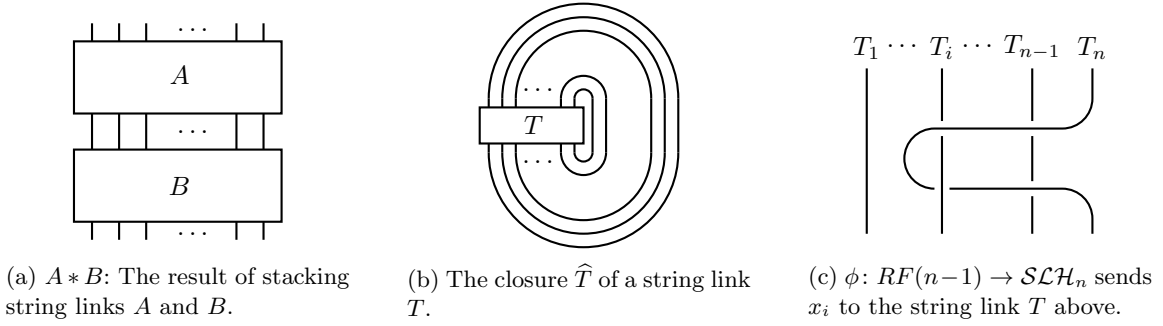
\begin{figure}[h]
\subcaptionbox{$A*B$: The result of stacking string links $A$ and $B$.\label{fig: stacking}}[0.3\linewidth]{
 \begin{tikzpicture}[scale=1.2]

\draw[thick, black] (0.2,-0.2) -- (0.2, 2.2);
\draw[thick, black] (0.5,-0.2) -- (0.5, 2.2);
\draw[thick, black] (0.8,-0.2) -- (0.8, 2.2);

\draw[thick, black] (1.8,-0.2) -- (1.8, 2.2);
\draw[thick, black] (2.1,-0.2) -- (2.1, 2.2);
\draw[thick, black, fill=white] (0,0) rectangle (2.3,0.8);

\draw[thick, black, fill=white] (0,1.2) rectangle (2.3,2);

\node[inner sep=0pt] at (1.17,0.4) {$B$};
\node[inner sep=0pt] at (1.17,1.63) {$A$};
\node[inner sep=0pt] at (1.34,-0.13) {\scalebox{1}{$\dots$}};
\node[inner sep=0pt] at (1.34,1) {\scalebox{1}{$\dots$}};
\node[inner sep=0pt] at (1.34,2.13) {\scalebox{1}{$\dots$}};

%fudgenode
\node[above] at (1, -0.3) {$\,$};

\end{tikzpicture}
}
\hfill
\subcaptionbox{{The closure $\widehat T$ of a string link $T$.\label{fig: closure}}}[0.3\linewidth]{
\begin{tikzpicture}[scale=0.6]

\draw[thick, black] (0.2,-0.2) -- (0.2, 1);
\draw[thick, black] (0.5,-0.2) -- (0.5, 1);
\draw[thick, black] (0.8,-0.2) -- (0.8, 1);

\draw[thick, black] (1.8,-0.2) -- (1.8, 1);
\draw[thick, black] (2.1,-0.2) -- (2.1, 1);
\draw[thick, black, fill=white] (0,0) rectangle (2.3,0.8);

%\draw[thick, black, fill=white] (0,1.2) rectangle (2.3,2);

\node[inner sep=0pt] at (1.2,0.4) {$T$};
\node[inner sep=0pt] at (1.34,-0.4) {\scalebox{1}{$\dots$}};
\node[inner sep=0pt] at (1.34,1.2) {\scalebox{1}{$\dots$}};
%\node[inner sep=0pt] at (1.34,2.13) {\scalebox{1}{$\dots$}};
\draw[thick, black] (2.1, 1) arc (-180:-360:.2);
\draw[thick, black] (1.8, 1) arc (-180:-360:{.2+.3});
\draw[thick, black] (.8, 1) arc (-180:-360:{.2+.3+1});
\draw[thick, black] (.5, 1) arc (-180:-360:{.2+.3+1+.3});
\draw[thick, black] (.2, 1) arc (-180:-360:{.2+.3+1+.3+.3});

\draw[thick, black] (2.1, -0.2) arc (180:360:.2);
\draw[thick, black] (1.8, -0.2) arc (180:360:{.2+.3});
\draw[thick, black] (.8, -0.2) arc (180:360:{.2+.3+1});
\draw[thick, black] (.5, -0.2) arc (180:360:{.2+.3+1+.3});
\draw[thick, black] (.2, -0.2) arc (180:360:{.2+.3+1+.3+.3});

\draw[thick, black] ({2*2.3-0.2},-0.2) -- ({2*2.3-0.2}, 1);
\draw[thick, black] ({2*2.3-0.5},-0.2) -- ({2*2.3-0.5}, 1);
\draw[thick, black] ({2*2.3-0.8},-0.2) -- ({2*2.3-0.8}, 1);

\draw[thick, black] ({2*2.3-1.8},-0.2) -- ({2*2.3-1.8}, 1);
\draw[thick, black] ({2*2.3-2.1},-0.2) -- ({2*2.3-2.1}, 1);

\end{tikzpicture}

}
\hfill
\subcaptionbox{{$\phi\colon RF(n-1)\to \mathcal{SLH}_n$ sends $x_i$ to the string link $T$ above.\label{fig: psi(x_i)}}}[0.3\linewidth]{
\begin{tikzpicture}[scale=1]

\node[above] at (0.2, 2) {$T_1$};
\draw[thick, black] (0.2,-0.2) -- (0.2, 2);

\node[inner sep=0pt] at (.7,2.3) {\scalebox{1}{$\dots$}};

\node[above] at (1.2, 2) {$T_i$};
\draw[thick, black] (1.2, 2)--(1.2,1.3);
\draw[thick, black] (1.2, 1.1)--(1.2,-0.2);
\node[inner sep=0pt] at (1.7,2.3) {\scalebox{1}{$\dots$}};

\node[above] at (2.4, 2) {$T_{n-1}$};
\draw[thick, black] (2.4, 2)--(2.4,1.3);
\draw[thick, black] (2.4, 1.1)--(2.4,.5);
\draw[thick, black](2.4,.3)--(2.4,-0.2);

\node[above] at (3.2, 2) {$T_{n}$};
\draw[thick, black] (3.2, 2)--(3.2,1.6) arc (0:-90:.4)--(1.1,1.2) arc(90:270:.4);
\draw[thick, black] (1.3,.4)--(2.8,.4) arc(90:0:.4)--(3.2,-.2);

%fudgenode
\node[above] at (1, -0.4) {$\,$};

\end{tikzpicture}
}
\caption{The stacking and closure operations, together with the map $\phi$.
}
\label{fig: stacking and closure}
\end{figure} 

The key tool we will use in our proof of Theorem~\ref{thm: compute nh} is the split exact sequence of \cite[Lemma 1.8]{HL1}:
\begin{equation}\label{eq:shortexact}
\begin{tikzcd}
1\ar[r] &RF(n-1)\ar[r,"\phi"]  &\mathcal{SLH}_{n}\ar[r,"\psi"] & \mathcal{SLH}_{n-1}\ar[r] \arrow[l, bend left=25,  dashed, "s"]&1.
\end{tikzcd}
\end{equation}
Here, $RF(n-1)$ indicates the \emph{reduced free group}.  That is, $RF(n-1)$ is the quotient of the free group $F(n-1)$ with generators $x_1,\dots, x_{n-1}$ so that for each $i$ any conjugate of $x_i$ commutes with any other conjugate of $x_i$.  The map $\phi\colon RF(n-1)\to \mathcal{SLH}_{n}$ is the homomorphism sending a generator $x_i$ to the string link of Figure~\ref{fig: psi(x_i)}.  The fact that this is well defined follows from work in~\cite{HL1}.  The map $\psi\colon \mathcal{SLH}_{n}\to \mathcal{SLH}_{n-1}$ is given by deleting the $n^{\text{th}}$ component of a string link.  The splitting $s\colon\mathcal{SLH}_{n-1}\to \mathcal{SLH}_{n}$ of $\psi$ is given by adding to an $(n-1)$-component string link $T$ an unknotted component which does not interact with the components of $T$.  In summary, any $T\in \mathcal{SLH}_n$ decomposes as $T=s(\psi(T))*\phi(\gamma)$ for some $\gamma\in RF(n-1)$.

\begin{definition}
Let $w = x_{i_1}^{\epsilon_1} x_{i_2}^{\epsilon_2} \dots x_{i_\ell}^{\epsilon_\ell}$ be a word in the letters $x_1^{\pm1},\dots, x_n^{\pm1}$.  The \emph{trivializing number} of $w$, denoted by $Z(w)$, is the minimum number of deletions needed to reduce $w$ to a word representing the trivial element of the free group $F(n)$.

Let $\gamma\in RF(n)$ be an element of the reduced free group. The \emph{reduced trivializing number} of $\gamma$, denoted by $RZ(\gamma)$, is the minimum of $Z(w)$ among all words $w\in F(n)$ which represent $\gamma$.
\end{definition}

The  proof of Theorem~\ref{thm: compute nh} will be inductive with the inductive step requiring bounds on $n_h(\phi(\gamma))$ where $\gamma$ is an element of  $RF(n)$.  Since $n_h(\phi(x_i^{\pm1})) = 1$ for each generator of $RF(n)$, it follows that 
\begin{equation}\label{eqn:RZ bounds nh}
n_h(\phi(\gamma)) \le RZ(\gamma).
\end{equation}
During the proof of Theorem~\ref{thm: compute nh}, we will furthermore derive an upper bound on $RZ(\gamma)$ in terms of the classical concepts of \emph{basic commutators} and their \emph{weights}, so we recall the definition of these now.

\begin{definition}Writing $\{x_1,\dots, x_{n}\}$ for a generating set of $F(n)$, the ordered set $\{c_1,c_2,\dots\}$ of \emph{basic commutators}, along with associated integer valued \emph{weights} $w(c_1)\le w(c_2)\le\dots$, are defined by the following:
\begin{itemize}[leftmargin=0.9cm]\setlength\itemsep{0em}
\item If $i=1,\dots, n$, then $c_i=x_i$ and $w(c_i)=1$.
\item If $i<j$, then $w(c_i)\le w(c_j)$.
\item If $i>n$, then $c_i= [c_\ell,c_j]$ for some $\ell<j<i$.  Additionally $w(c_i)= w(c_\ell)+w(c_j)$.
\item If $c_i= [c_\ell,c_j]$ and $c_j=[c_r,c_s]$, then $r\le \ell$.
\item Every $[c_i,c_j]$ satisfying the conditions above is a basic commutator.  
\end{itemize}
\end{definition}

The next lemma studies the reduced trivializing number for basic commutators.

\begin{lemma}\label{lem: RZ of basic comm}
Let $c_i\in F(n)$ be a basic commutator. Then, considering $c_i$ as an element of the reduced free group $RF(n)$, we have
\[
\left\{\begin{array}{rcccl}
RZ(c_i^a) &=& |a|&&\text{if }1\leq i\leq n,\\
RZ(c_i^a) &\le& w(c_i)&&\text{otherwise.}
\end{array}\right.
\]
\end{lemma}

\begin{proof}
First suppose $1\leq i\leq n$. Then $c_i^a = x_i^a$ is a length $|a|$ word, and so $RZ(c_i^a)\le |a|$.  In order to see the reverse inequality, note that $\phi(c_i^a)$ is an $(n+1)$-component string link with $\Lambda(\phi(c_i^a)) = |a|$. Thus by inequalities (\ref{eq:lambda}) and (\ref{eqn:RZ bounds nh}) we have $RZ(c_i^a)\ge n_h(\psi(c_i^a)) \ge \Lambda(\psi(c_i^a)) = |a|$.  

The proof of the result when $n<i$ begins with an inductive argument showing that each basic commutator $c_i$ is a product of conjugates of a single generator $x_t$ for some $t$ and that $RZ(c_i)\le w(c_i)$.  When $w(c_i)=1$, then $c_i=x_i$ and so we are done.  When $w(c_i)>1$, then $c_i=[c_\ell, c_j]$ where $w(c_i) = w(c_\ell)+w(c_j)$.  In particular, $w(c_j)<w(c_i)$.  We can therefore inductively assume that $c_j$ is a product of conjugates of some $x_t$.
Thus, 
$
c_i = [c_\ell, c_j] = (c_\ell c_j c_\ell^{-1}) c_j^{-1},
$
and so since both $(c_\ell c_j c_\ell^{-1})$ and $c_j^{-1}$ are products of conjugates of $x_t$, we have now expressed $c_i$ as a product of conjugates of $x_t$.  

Notice next that in the expression $(c_\ell c_j c_\ell^{-1}) c_j^{-1}$ if we make $2 \cdot RZ(c_\ell)$ letter deletions then we may replace the $c_\ell$ and $c_\ell^{-1}$ subwords with words representing the trivial element.  As a consequence $RZ(c_i)\le 2\cdot RZ(c_\ell)$.  By the definition of basic commutators we have $w(c_\ell)\le w(c_j)$ and $w(c_i)=w(c_\ell)+w(c_j)$.  Therefore $w(c_\ell)\le \frac{1}{2} w(c_i)$.  By our inductive assumption, $RZ(c_\ell)\le w(c_\ell)$.  Thus
$$
RZ(c_i)\le 2 \cdot RZ(c_\ell)\le 2 w(c_\ell)\le w(c_i).
$$
This completes the inductive argument.

 Midway through that induction, we saw that both $c_\ell c_j c_\ell^{-1}$ and $c_j^{-1}$ are products of conjugates of a single $x_t$ for some $t$.  By the definition of $RF(n)$, they commute.  Thus, 
$$
c_i^{a} = \left((c_\ell c_j c_\ell^{-1}) c_j^{-1}\right)^{a} = (c_\ell c_j ^{a}c_\ell^{-1}) c_j^{-a}
$$
Similarly to our inductive argument, by making $2\cdot RZ(c_\ell)$ letter deletions, the $c_\ell$ and $c_{\ell}^{-1}$ subwords appearing in the expression above can be reduced to words representing the trivial element of $RF(n)$.  Thus, just as in the inductive argument, $RZ(c_i^a)\le 2 \cdot RZ(c_\ell)\le 2 w(c_\ell)\le w(c_i)$.  
\end{proof}

In \cite[Lemma 1.3]{HL1}, Habegger-Lin prove that the reduced free group $RF(n)$ is nilpotent of class $n$. In the proof of Theorem~\ref{thm: compute nh}, we will use a simple expression for elements of $RF(n)$ that is obtained by combining their result with the following classical theorem; see e.g.~\cite[Theorem~5.13A]{MKS}.

\begin{theorem}[P.~Hall's Basis Theorem]\label{thm:Hall}
Any $\gamma\in F(n)/F(n)_{n+1}$ can be expressed as 
\[
\gamma = c_1^{a_1}c_2^{a_2}\dots c_N^{a_N}\in F(n)/F(n)_{n+1}
\]
where $N$ is the number of basic commutators of weight at most $n$ and $a_1,\dots, a_N\in \Z$.
\end{theorem}

We now have everything we need to complete the proof of Theorem~\ref{thm: compute nh}.

\begin{proof}[Proof of Theorem~\ref{thm: compute nh}]
 
Let $L$ be a 2-component link in $S^3$. Since the linking number is a complete homotopy invariant for 2-component links~\cite{M1}, we have that $n_h(L) = \Lambda(L)$.

Now let $L$ be an $3$-component link in $S^3$.  Assume first that $\Lambda(L)\neq 0$ and after reordering the components of $L$ and changing the orientations of a component, if needed, we may assume that $\lk(L_2, L_3)>0$.  Let $T\subseteq D^2\times[0,1]$ be a string link with $\widehat{T}=L$. 
Using short exact sequence~\eqref{eq:shortexact}, we have that
$T=s(\psi(T))*\phi(\gamma)\in \mathcal{SLH}_3$ for some $\gamma\in RF(2)$. Note that the linking number of $\psi(T) \in \mathcal{SLH}_2$ is equal to $\lk(L_1,L_2)$.  Therefore $n_h(s(\psi(T))) = |\lk(L_1,L_2)|$.

As mentioned above, the reduced free group $RF(2)$ is nilpotent of class $2$~\cite[Lemma 1.3]{HL1}. Hence by Theorem~\ref{thm:Hall}, we may express $\gamma$ in terms of basic commutators: $\gamma = x_1^a * x_2^b * [x_1,x_2]^c$ for some $a,b,c\in\Z$.  It follows by inspection that $a=\lk(L_1,L_3)$ and $b=\lk(L_2, L_3)>0$. In $RF(2)$, $x_2$ and $x_1x_2x_1^{-1}$ commute as do $x_1$ and~$x_2x_1^{-1}x_2^{-1}$.  Thus,
\begin{eqnarray*}\gamma &=& x_1^a  x_2^b  [x_1,x_2]^c \\
&=&  x_1^a x_2^b ((x_1x_2 x_1^{-1}) x_2^{-1})^c \\
&=& x_1^a  (x_1(x_2 x_1^{-1} x_2^{-1}))^cx_2^b\\
& =& x_1^a  x_1^c (x_2x_1^{-1}x_2^{-1})^c x_2^b \\
&=& x_1^a  x_1^c x_2x_1^{-c}x_2^{b-1} 
\end{eqnarray*}

Deleting $|a| = |\lk(L_1,L_3)|$ instances of $x_1$ and $1+|b-1|= b =  \lk(L_2,L_3)$ instances of $x_2$ reduces this word to $x_1^cx_1^{-c}$ which is the trivial word.   Thus, by using inequality~\eqref{eqn:RZ bounds nh}, we obtain $n_h(\phi(\gamma))\le RZ(\gamma)\le  |\lk(L_1,L_3)|+ |\lk(L_2,L_3)|$.  
Finally, by using the above inequalities combined with inequality~\eqref{eq:closurenh}, we conclude that 
\begin{multline*}n_h(L)\le n_h(T) \le n_h(s(\psi(T)))+n_h(\phi(\gamma))\\
\le |\lk(L_1,L_2)|+|\lk(L_1,L_3)|+|\lk(L_2,L_3)|
 = \Lambda(L).\end{multline*}  This gives the claimed result when $L$ has 3 components and $\Lambda(L)\neq 0$.

Now suppose that the 3-component link $L$ has $\Lambda(L)=0$ and $\mu_{123}(L)\neq 0$. As above, if $T\subseteq D^2\times[0,1]$ is a string link with $\widehat{T}=L$, then by short exact sequence~\eqref{eq:shortexact}, we have that
$T=s(\psi(T))*\phi(\gamma)\in \mathcal{SLH}_3$ for some $\gamma\in RF(2)$. In this case, $\psi(T)$ is a 2-component string link with vanishing linking numbers so that $\psi(T)$ is homotopically trivial.  Furthermore, for some $c\in \Z$,
$$\gamma=[x_1,x_2]^c = (x_1x_2x_1^{-1}x_2^{-1})^c =  x_1^c(x_2x_1^{-1}x_2^{-1})^c =x_1^cx_2x_1^{-c}x_2^{-1}$$
where the third equality follows since $x_1$ commutes with $(x_2x_1^{-1}x_2^{-1})$.  The resulting word reduces to the trivial element of the free group after two letter deletions (specifically, deleting $x_2$ and $x_2^{-1}$).   As we have explained above, this affects $\phi(\gamma)$ by two crossing changes.  Thus, $n_h(L)\le n_h(T)\le2$.  Since $\mu_{123}(L)\neq 0$, $L$ is not homotopically-trivial, and so $n_h(L)>0$.  As $n_h(L)\equiv \Lambda(L)\pmod 2$, we conclude that  $n_h(L)=2$, as claimed.

Next we address the case that $L$ has 3 components and $\Lambda(L) = \mu_{123}(L)=0$.  In this case, Milnor~\cite{M1} concludes that $L$ is link homotopically trivial and so $n_h(L)=0$, as claimed.

We now move on to the proof of the statement concerning links with 4 or more components.  We begin with the inductive assumption that there is a constant $C_n$ so that for every string link $Q\in \mathcal{SLH}_n$ there is an inequality $n_h(Q)\le \Lambda(Q)+C_n$.  Let $T\in \mathcal{SLH}_{n+1}$, then by short exact sequence~\eqref{eq:shortexact}, we have that $T=s(\psi(T))*\phi(\gamma)$ for some $\gamma\in RF(n)$. 

As before, combining the fact that $RF(n)$ is nilpotent of class $n$ with Theorem~\ref{thm:Hall}, we may express $\gamma \in RF(n)$ in terms of basic commutators:
$$
\gamma = c_1^{a_1}c_2^{a_2}\dots c_N^{a_N}
$$
where $N$ is the number of basic commutators of weight at most $n$ and $a_1,\dots, a_N\in \Z$.  By inspection $a_i = \lk(L_{n+1},L_i)$ for $i=1,\dots n$.   Appealing to Lemma~\ref{lem: RZ of basic comm}, we have that
$$
n_h(\phi(\gamma))\le RZ(\gamma)\le \Sum_{i=1}^N RZ(c_i^{a_i}) \le \Sum_{i=1}^n |a_i| +\Sum_{i=n+1}^N w(c_i).
$$
Thus, by the above inequality combined with the inductive hypothesis we get
\begin{eqnarray*}
n_h(T)\le n_h(s(\psi(T)))+n_h(\phi(\gamma))&\le& \Lambda(s(\psi(T)))+C_n+\Sum_{i=1}^n |a_i| +\Sum_{i=n+1}^N w(c_i) \\
&=& \Lambda(T)+C_n+\Sum_{i=n+1}^N w(c_i).
\end{eqnarray*}
Setting $C_{n+1}=C_n+\Sum_{i=n+1}^N w(c_i)$ completes the induction.\end{proof}

\begin{remark}The statement of Theorem~\ref{thm: compute nh} does not give a precise value for the sequence of numbers $C_n$ when $n>3$. We note that by combining the recurrence relation $C_{n+1}=C_n+\sum_{i=n+1}^N w(c_i)$ (from the end of the proof) with Witt's formula for the number of basic commutators of a fixed weight~\cite{Witt37} one can find an upper bound for the $C_n$ constructed in the proof, and these upper bounds could themselves function as the $C_n$ in the statement of the theorem. However, a formula produced this way would be far from sharp. This is because any basic commutator with repeated indices ($[x_2, [x_1,x_2]]$ for instance) is zero in $RF(n)$ and so should not be counted in a formula for $C_n$. Thus to obtain a less crude formula for $C_n$ we would desire a Witt-type formula counting only the number of basic commutators without repeated indices. While such a result might be within reach of current technology, it is definitely beyond the scope of this paper. \end{remark}

\appendix
\section{Freely slicing boundary links}\label{appendix}

Cha-Kim-Powell \cite{Cha-Kim-Powell:2020-1} describe a set of conditions on a link in $S^3$ that ensure the link is freely slice. In Section \ref{sec:homotopictoslice}, we generalized these conditions to links in a general homology $3$-sphere $Y$ and claimed in Theorem~\ref{thm:CKP} that our conditions guaranteed the link was freely slice in the contractible $4$-manifold~$X$ bounded by $Y$. The proof of this is a close imitation of the argument from Cha-Kim-Powell \cite[\textsection 4 \& \textsection 5]{Cha-Kim-Powell:2020-1} and, as such, we only sketch the argument below. An attempt has been made to include enough detail to follow the argument, but without repeating too much of what already appears in \cite{Cha-Kim-Powell:2020-1}.

\medskip

We begin by recalling some terminology and a theorem from Freedman-Quinn \cite{Freedman-Quinn:1990-1}. A \emph{transverse pair} is two copies of $S^2\times D^2$ plumbed together at one point. This model is a neighbourhood of the pair of spheres
\[
(S^2\times\{pt\})\cup(\{pt\}\times S^2)\subseteq S^2\times S^2.
\]
Take the disjoint union $N_1, \dots, N_\ell$ of copies of the transverse pair and perform further plumbings between the copies, possibly including self-plumbings, then map the result into a topological $4$-manifold $W$ via a continuous map that is a homeomorphism to its image. The result of this process is a map $f\colon \coprod_i N_i\to W$ which is called an \emph{immersion of a union of transverse pairs}.

An immersion of a union of transverse pairs is said to have \emph{algebraically trivial intersections} if the images of the further plumbings we performed can be arranged in pairs by Whitney disks in $W$ (that may \emph{a priori} meet $\coprod_i N_i$). Such a map $f$ is called \emph{$\pi_1$-null} if the inclusion induced map $\pi_1\left(f\left(\coprod_i N_i\right)\right)\to \pi_1(W)$ is trivial.

If middle-dimensional homology classes can be represented in this arrangement, then one is able to use a result of Freedman-Quinn \cite[Theorem 6.1]{Freedman-Quinn:1990-1} to conclude that $f$ is $s$-cobordant rel.~boundary to an embedding.  In the particular case of interest to us, this gives the following.

\begin{theorem}
\label{thm:FQ} Suppose $W$ is a compact topological $4$-manifold, bounded by $M_L$, with $\pi_1(W)$  free and generated by the meridians of $L$. Let $f\colon \coprod_i N_i\to W$ be a $\pi_1$-null immersion of a union of transverse pairs with algebraically trivial intersections, and inducing an isomorphism $f_*\colon H_2(\coprod_i N_i)\to H_2(W)$. Then there exists a compact topological 4-manifold $W'$, bounded by $M_L$, with $\pi_1(W')$  free and generated by the meridians of $L$, and a locally flat embedding $f'\colon \coprod_i N_i\into W'$ inducing an isomorphism  $f'_*\colon H_2(\coprod_i N_i)\to H_2(W')$.
\end{theorem}

We now follow the standard surgery-theoretic approach to slice $L$, sketched in the introduction. Recall that the 0-surgery on $L$ is denoted by $M_L$.

\begin{proposition}\label{prop:CKP}
Let $L$ be a boundary link with a good disky basis in a homology sphere $Y$, then $M_L$ bounds a compact oriented $4$-manifold $W$ such that
\begin{enumerate}[leftmargin=0.9cm, font=\upshape]\setlength\itemsep{0em}
\item $\pi_1(W)$ is free and generated by the meridians of $L$, and
\item $H_2(W;\Z)$ is free and represented by a $\pi_1$-null immersion of a union of transverse pairs with algebraically trivial intersections.
\end{enumerate}
\end{proposition}

The argument we now use is almost identical to that appearing in \cite[Section 5]{Cha-Kim-Powell:2020-1}.

\begin{proof}[Summary of the proof of Proposition~\ref{prop:CKP}]
Let $F=F_1\cup\dots\cup F_n$ be a boundary link Seifert surface for $L$ and let $\{a_i, b_i\}_{1\leq i \leq g}$ be a good disky basis, with $$\left\{\Delta_j^+,\Delta_i \mid 1\leq j \leq 2g, 1\leq i \leq g\right\}$$  the immersed disks as in Definition~\ref{def:disky}.  Recalling these conditions, for each $i$, $\bdry \Delta_{i}^+ = a_i$, $\bdry \Delta^+_{g+i} = (b_i')^+$, and $\bdry \Delta_i = b_i'$, where $b_i'$ is the result of pushing $b_i$ off $F$ such that it has zero linking with $a_i$, and  $(b_i')^+$ is a zero linking parallel copy of $b_i'$.  These disks are all disjoint except that the disks $\{\Delta_j^+\}_{1\leq j \leq 2g}$ might intersect each other.  Write $X$ for the contractible $4$-manifold bounded by $Y$.

 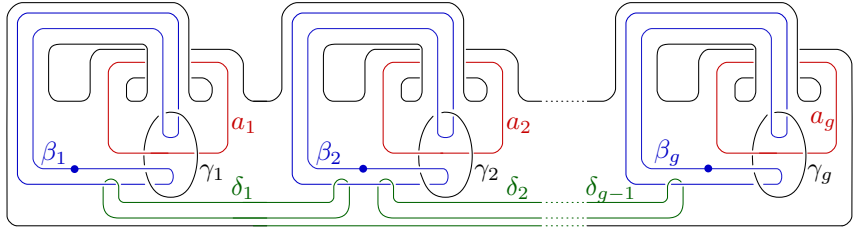
\begin{figure}[h]
\begin{picture}(325,105)
%\put(400,105){*}
%\put(300,105){*}
\put(0,5){\includegraphics[width=112.5mm]{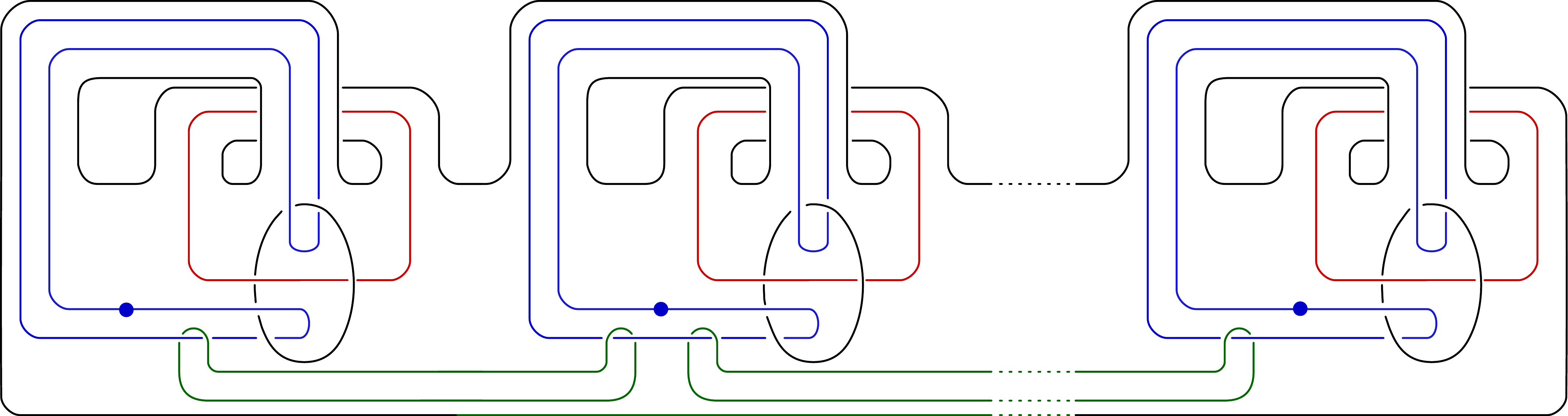}}
\put(13,30){\textcolor{diagramblue}{$\beta_1$}}
\put(85,42){\textcolor{diagramred}{$a_1$}}
\put(73,24){\textcolor{black}{$\gamma_1$}}
\put(85,17){\textcolor{diagramgreen}{$\delta_1$}}

%+104
\put(117,30){\textcolor{diagramblue}{$\beta_2$}}
\put(189,42){\textcolor{diagramred}{$a_2$}}
\put(177,24){\textcolor{black}{$\gamma_2$}}
\put(189,17){\textcolor{diagramgreen}{$\delta_2$}}

%+(another) 128
\put(245,31){\textcolor{diagramblue}{$\beta_g$}}
\put(304,44){\textcolor{diagramred}{$a_g$}}
\put(303,24){\textcolor{black}{$\gamma_g$}}
\put(219,17){\textcolor{diagramgreen}{$\delta_{g-1}$}}

\end{picture}
\caption{Curves $a_i$, $\beta_i$, $\gamma_i$ and $\delta_i$ sitting in a produce neighbourhood of a Seifert surface for $L$.  Attaching a 1-handle using the dotted $\beta_i$ curves and attaching 2-handles along the 0-framings of $a_i$, $\gamma_i$, and $\delta_i$.}\label{fig:ManifoldW1}
\end{figure}

For each $i$, let $\beta_i\cup \gamma_i$ be the Bing double of $b_i$ appearing in Figure~\ref{fig:ManifoldW1}. 
Attach 1-handles to $X$ along $\beta_i$ and 2-handles to $X$ along the 0-framings of $a_i$, $\gamma_i$, and $\delta_i$ to get a 4-manifold~$W$.  A straightforward argument shows that $W$ has boundary $M_L$ and has fundamental group freely generated by the meridians of $L$; see~\cite[Claim A]{Cha-Kim-Powell:2020-1} for details. Clearly $H_2(W;\Z)\cong \Z^{2g}$. This basis is generated by framed immersed spheres $\Sigma_1,\dots,\Sigma_{2g}$ described as follows.  For each $i$, take $\Sigma_{2i-1}$ to be the union of $\Delta_i^+$ and the core of the 2-handle attached to $a_{i}$. For each $i$, we can and will assume $b_i'$ and $(b'_i)^+$ lie on the gray surface bounded by $\gamma_i$ depicted in Figure~\ref{fig:ManifoldW} (left). We use this to define a planar surface $P_i$ bounded by $\gamma_i$, $b_i'$ and $(b_i')^+$ as in Figure~\ref{fig:ManifoldW} (right). Take $\Sigma_{2i}$ to be the union of $\Delta_{g+i}^+$, $\Delta_i$, the core of the handle attached to $\gamma_i$, and  $P_i$; cf.~\cite[Claim B]{Cha-Kim-Powell:2020-1}.  
   
   \begin{figure}[h]
\begin{picture}(360,130)
%\put(360,130){*}
\put(0,5){\includegraphics[height=44.4mm]{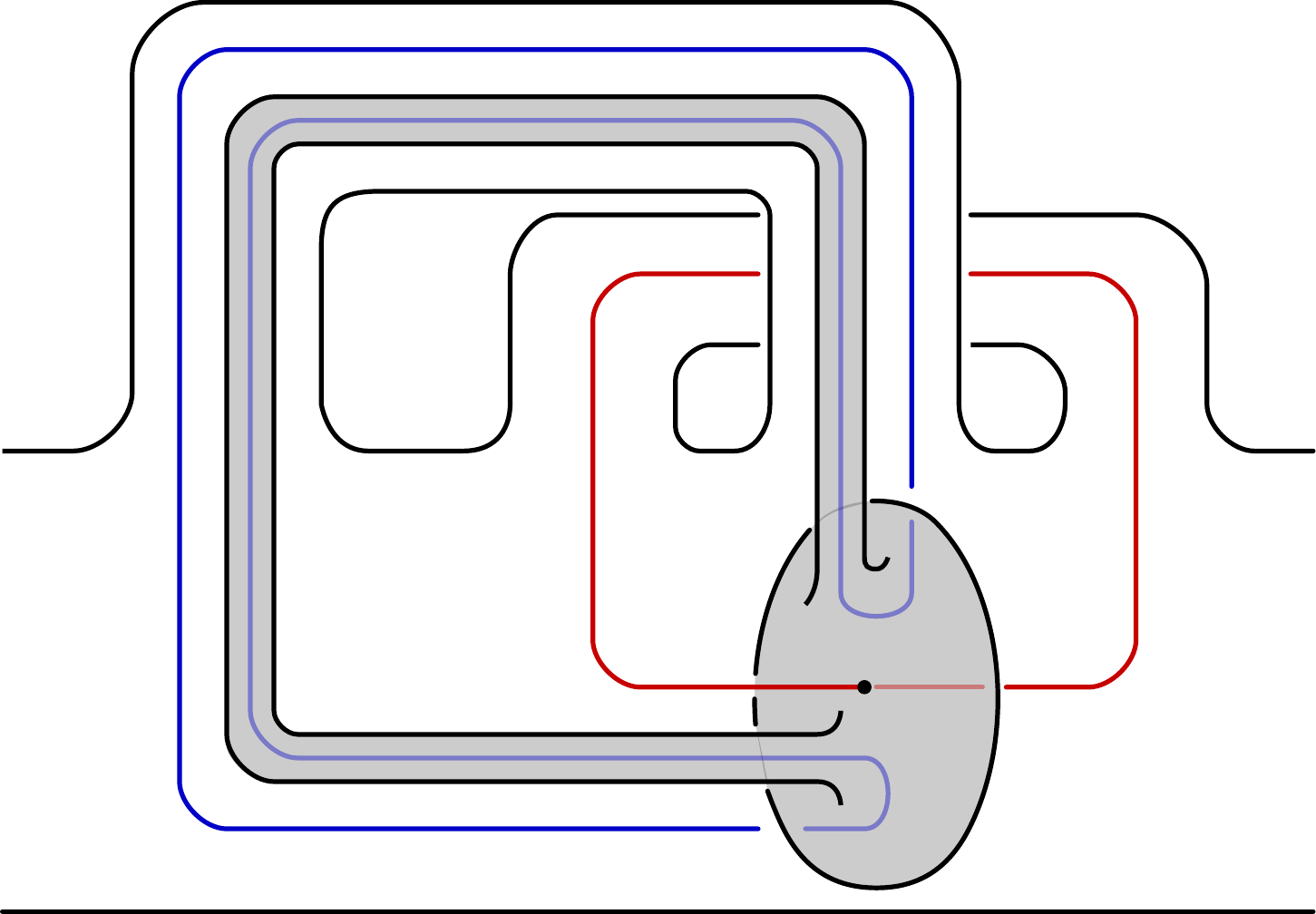}}
\put(15,40){\textcolor{diagramblue}{$\beta_i$}}
\put(160,45){\textcolor{diagramred}{$a_i$}}
\put(140,17){\textcolor{black}{$\gamma_i$}}
\put(220,5){\includegraphics[height=40.7mm]{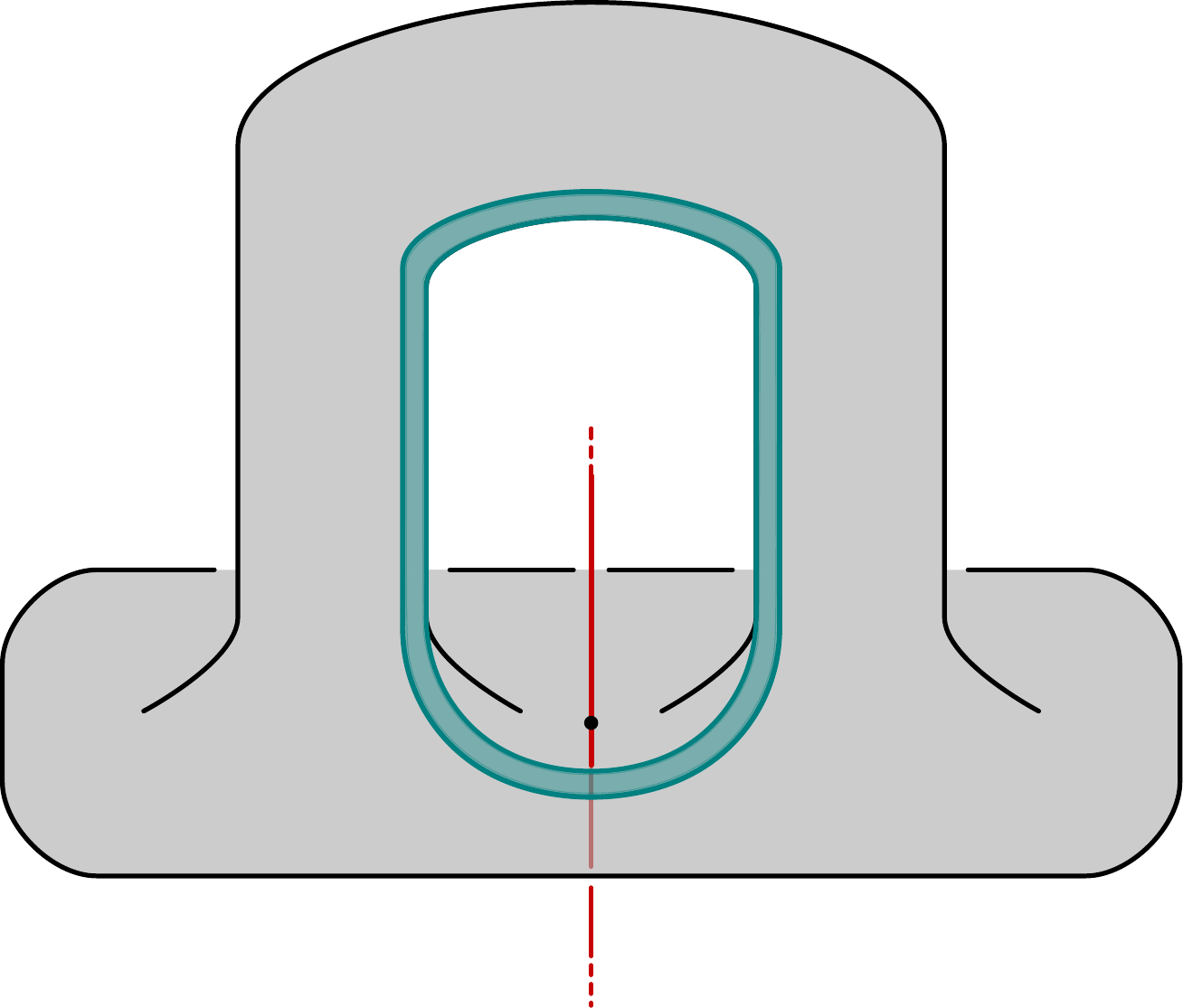}}
\put(325,12){\textcolor{black}{$\gamma_i$}}
\put(277,8){\textcolor{diagramred}{$a_i$}}
%\put(320,100){\textcolor{teal}{$b_k$}}
%\put(302,85){\textcolor{teal}{$b'_k$}}
\put(270,103){\textcolor{teal}{$b_i'$}}
\put(270,83){\textcolor{teal}{$(b'_i)^+$}}
\put(345,86){\textcolor{teal}{$A_i$}}
%\put(398,95){\color{teal}\line(-5,-2){30}}
\put(345,85){\color{teal}\line(-5,-2){38}}
\end{picture}
\caption{Left: A section of the surface $F$ containing $\{a_i,b_i\}$ ($b_i$ not depicted). The curves $\beta_i$ and $\gamma_i$ form a bing double of the curve $b_i$ in a neighbourhood of $F$.  A gray genus one surface disjoint from $\beta_i$ with boundary $\gamma_i$ is also depicted. Right: A close-up of the gray surface. The annulus $A_i$ on the gray surface with boundary $b_i$ and $b'_i$ is depicted. The complement of the interior of the annulus in the gray surface is the planar surface $P_i$. 
}\label{fig:ManifoldW}
\end{figure}
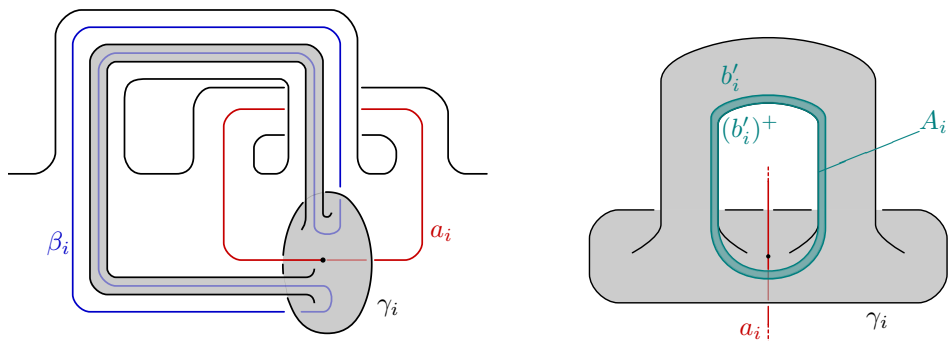
   
For each $i$, a regular neighbourhood of $\Sigma_{2i-1}\cup \Sigma_{2i}$ can now be viewed as an immersed transverse pair. The same arguments from \cite[Claim C]{Cha-Kim-Powell:2020-1} and \cite[Claim D]{Cha-Kim-Powell:2020-1} now reveal that $\cup_{i=1}^{2g} \Sigma_i$ has algebraically trivial intersections and is $\pi_1$-null.
\end{proof}

Finally, we can confirm that Cha-Kim-Powell \cite[Theorem A]{Cha-Kim-Powell:2020-1} generalizes as claimed.

\begin{proof}[Proof of Theorem \ref{thm:CKP}] Let $W$ be the $4$-manifold and $f\colon \coprod_i N_i\to W$ the immersion of a union of transverse pairs representing $H_2(W;\Z)$ described in Proposition \ref{prop:CKP}. Applying Theorem \ref{thm:FQ}, we obtain $W'$ and $f'$. 
Note that 
the image of $f'$ consists of a tubular neighbourhood of locally flat embedded $2$-spheres representing generators for $H_2(W';\Z)\cong H_2(W;\Z)$. These embedded 2-spheres come in transverse pairs and we now perform surgery on one sphere from each transverse pair. Since the second sphere from each transverse pair intersected the surgered sphere geometrically once, these surgeries preserve $\pi_1(W')$.  Thus, we obtain $W''$ with boundary $M_L$, with $H_2(W'';\Z)=0$, and with $\pi_1(W'')$ freely generated by the meridians of $L$. Now attach $2$-handles to $M_L$ along the meridians of the link components, with framing so that the $0$-surgery is reversed. This has $Y$ as the effect of surgery, and by glueing across meridians we ensure that $\pi_1(W'')=0$. The resultant $4$-manifold is contractible and has boundary $Y$. The link $L$ has slice disks given by the cocores of the 2-handles we have just attached so it is slice and moreover freely slice as $\pi_1(W'')$ is free.
\end{proof}

\bibliographystyle{abbrv}

\bibliography{LinksHomotopicToSlice}

%%%%%%%%%%%%%%%%%%%%%%%%%%%%%%%%%%%
\end{document}